\documentclass{article}
\usepackage{amsmath, amsthm, amsfonts, amssymb, enumerate}

\title{Flat Cyclotomic Polynomials: A New Approach}
\author{Sam Elder\\Massachusetts Institute of Technology\\Cambridge, MA\\\texttt{same@mit.edu}}

\newtheorem{thm}{Theorem}
\newtheorem{lem}[thm]{Lemma}
\newtheorem{prop}[thm]{Proposition}
\newtheorem{cor}[thm]{Corollary}
\newtheorem{conj}{Conjecture}

\theoremstyle{definition}
\newtheorem*{defn}{Definition}
\newtheorem{ques}[conj]{Question}

\theoremstyle{remark}

\newcommand{\Z}{\mathbb Z}

\begin{document}

\maketitle

\begin{abstract}
The \emph{height} of a polynomial is its greatest coefficient in absolute value. Polynomials of unit height are \emph{flat}. The \emph{cyclotomic polynomial} $\Phi_n(x)$ is the minimal polynomial of any primitive $n$th root of unity. The \emph{order} of $\Phi_n(x)$ is the number of distinct odd primes dividing $n$. All cyclotomic polynomials of orders 0, 1 and 2 are flat, and some of orders 3 and 4 are flat as well.

In this paper, we build a new theory for analyzing the coefficients of $\Phi_n(x)$ by considering it as a gcd of simpler polynomials. We first obtain a generalization of a result known as periodicity in Theorem \ref{extendedperiodicitythm}: If $n$ is a positive integer and $s$ and $t$ primes such that $n-\varphi(n)<s<t$ and $s\equiv\pm t\pmod{n}$, then $\Phi_{ns}(x)$ and $\Phi_{nt}(x)$ have the same height.

We also use this theory to provide two new families of flat cyclotomic polynomials. One, of order 3, was conjectured by Broadhurst: Let $p<q<r$ be primes and $w$ a positive integer such that $r\equiv\pm w\pmod{pq}$, $p\equiv1\pmod w$ and $q\equiv 1\pmod{wp}$. Then, as we prove in Theorem \ref{BroadhurstII}, $\Phi_{pqr}(x)$ is flat. The other is the first general family of order 4. We prove in Theorem \ref{pqrsflatthm} that $\Phi_{pqrs}(x)$ is flat for primes $p,q,r,s$ where $q\equiv-1\pmod p$, $r\equiv\pm1\pmod{pq}$, and $s\equiv\pm1\pmod{pqr}$. Finally, we prove in Theorem \ref{pqrstnotflatthm} that the natural extension of this family to order 5, $\Phi_{pqrst}(x)$ with $r\equiv\pm1\pmod{pq}$, $s\equiv\pm1\pmod{pqr}$ and $t\equiv\pm1\pmod{pqrs}$, is never flat.
\end{abstract}

\section{Introduction}

The cyclotomic polynomial $\Phi_n(x)$ is the minimal polynomial of any primitive $n$th root of unity. That is,
\begin{equation}\Phi_n(x)=\prod_{\substack{d=1\\(d,n)=1}}^n\left(x-e^{\frac{2\pi id}n}\right).\label{definition}\end{equation}
Collecting all $n$th roots of unity yields
\begin{equation}x^n-1=\prod_{d\mid n}\Phi_d(x).\label{x^n-1}\end{equation}

From the definition, we see that the degree of $\Phi_n(x)$ is Euler's totient $\varphi(n)$, and that $\Phi_n(x)$ is monic. It is well-known that $\Phi_n(x)\in\Z[x]$, and Gauss showed that $\Phi_n(x)$ is irreducible over $\Z[x]$. Additionally, $\Phi_n(x)$ is reciprocal for $n>1$; the coefficients of $x^k$ and $x^{\varphi(n)-k}$ are equal.

By M\"obius inversion, equation \eqref{x^n-1} is equivalent to
\begin{equation}\Phi_n(x)=\prod_{d\mid n}(x^d-1)^{\mu(n/d)},\label{Mobius}\end{equation}
where $\mu$ is the M\"obius function. This suggests a well-known reduction to the case where $n$ is squarefree:
\begin{prop}\label{radicalprop}Let $n$ be a positive integer, and let $m$ be the product of the distinct primes dividing $n$. Then $\Phi_n(x)=\Phi_m(x^{n/m})$.\end{prop}

We can also reduce to the case where $n$ is odd.
\begin{prop}\label{2nprop}$\Phi_{2n}(x)=\Phi_n(-x)$ for odd $n>1$.\end{prop}

Therefore, we can consider $n$ odd and squarefree, and use Propositions \ref{radicalprop} and \ref{2nprop} to extend results to all $n\in\Z^+$. Let the \emph{order} of the cyclotomic polynomial $\Phi_n(x)$ be the number of distinct odd primes dividing $n$. We call cyclotomic polynomials of order 1 \emph{prime}, order 2 \emph{binary}, order 3 \emph{ternary}, order 4 \emph{quaternary}, and order 5 \emph{quinary}, respectively.

Since 2006, there has been a surge of interest in the coefficients of the cyclotomic polynomials, e.g. [1-7], [11-12]. The binary cyclotomic polynomials are completely characterized; this characterization is repeated in Section 2. After presenting a new general theory for analyzing these coefficients in Sections 3-7, we use this theory to prove old and new results and make conjectures on ternary cyclotomic polynomials in Section 8, quaternary cyclotomic polynomials in Section 9, and quinary and higher order cyclotomic polynomials in Section 10. Almost all previous work has been concerned with ternary cyclotomic polynomials, the first nontrivial case.

To move from one order to the next, we can either use equation \eqref{Mobius} directly or the following:
\begin{prop}\label{npprop}Let $n$ be a positive integer and $p$ a prime not dividing $n$. Then $\Phi_{np}(x)=\dfrac{\Phi_n(x^p)}{\Phi_n(x)}$.\end{prop}

We will be interested in the set of coefficients of $\Phi_n(x)$, and in particular, the greatest in absolute value. Since we will be examining the coefficients of a wide array of polynomials, we need very general notation.
\begin{defn}For a polynomial $p(x)\in\Z[x]$, let $[x^k]p(x)$ denote the coefficient of $x^k$ in $p(x)$, and $V(p(x))=\{[x^k]p(x)| k\in\Z\}$ denote its set of coefficients, including zero. Abbreviate $V(\Phi_n(x))$ by $V_n$.\end{defn}
\begin{defn}The \emph{height} $A(n)$ of the cyclotomic polynomial $\Phi_n(x)$ is its greatest coefficient in absolute value: $A(n)=\max(V_n\cup -V_n)$. Any polynomial of height 1 is \emph{flat}.\end{defn}

The prime cyclotomic polynomials are very simple. Equation \eqref{Mobius} yields
\begin{equation}\Phi_p(x)=\frac{x^p-1}{x-1}=1+x+\dotsb+x^{p-2}+x^{p-1}.\label{p}\end{equation}

Thus, both $\Phi_1(x)=x-1$ and $\Phi_p(x)$ are flat. We will show shortly that $\Phi_{pq}(x)$ is also flat. The first ternary cyclotomic polynomial is $\Phi_{3\cdot5\cdot7}(x)=\Phi_{105}(x)$, which includes the term $-2x^7$, making $A(105)\ge2$. In fact, $A(105)=2$. Schur showed that $A(n)$ can be made arbitrarily large by increasing the order. For that proof, see \cite{Lenstra}.

Additionally, we know the evaluations of $\Phi_n(x)$ at 0 and at 1.
\begin{prop}\label{0and1}For $n>1$, $\Phi_n(0)=1$, while $\Phi_1(0)=-1$. If $n=p^k$ with $p$ prime, $\Phi_n(1)=p$; $\Phi_1(1)=0$; and when at least two distinct primes divide $n$, $\Phi_n(1)=1$.\end{prop}

Finally, Moree \cite{ICP} has investigated a related family of polynomials, the \emph{inverse cyclotomic polynomials}, defined by $\Psi_n(x)=\dfrac{x^n-1}{\Phi_n(x)}$. The degree of $\Psi_n(x)$ is $n-\varphi(n)$, and it is also found in the power series $\dfrac1{\Phi_n(x)}=-\dfrac{\Psi_n(x)}{1-x^n}=-\Psi_n(x)(1+x^n+x^{2n}+\dotsb)$. They arise naturally from Proposition \ref{npprop}, which becomes
\begin{equation}\Phi_{np}(x)=\frac{\Phi_n(x^p)}{\Phi_n(x)}=-\Psi_n(x)\Phi_n(x^p)(1+x^n+x^{2n}+\dotsb).\label{npreciprocal}\end{equation}

The polynomial $\Psi_n(x)$ is monic, and for $n>1$ by Proposition \ref{0and1}, $\Psi_n(1)=0$ and $\Psi_n(0)=-1$. We can compute easily that for primes $p<q$, $\Psi_1(x)=1$, $\Psi_p(x)=x-1$, and $\Psi_{pq}(x)=-1-x-\dotsb-x^{p-1}+x^q+x^{q+1}+\dotsb+x^{p+q-1}$. Finally, the analog of Proposition \ref{npprop} for the inverse cyclotomic polynomials is
\begin{equation}\Psi_{np}(x)=\Psi_n(x^p)\Phi_n(x).\label{inversenp}\end{equation}

We will additionally need several properties of $\Phi_{pq}(x)$, which we derive next.

\section{Binary Cyclotomic Polynomials}\label{binarysection}

From equation \eqref{Mobius}, we have
\begin{equation}\Phi_{pq}(x)=\frac{(x^{pq}-1)(x-1)}{(x^p-1)(x^q-1)}.\label{Mobiuspq}\end{equation}
Two papers, \cite{LL} and \cite{Lenstra}, independently give the following characterization of $\Phi_{pq}(x)$, each in a slightly different form:

\begin{prop}\label{Ldiagramprop}Let $p$ and $q$ be distinct odd primes, $\mu$ the inverse of $p$ modulo $q$, and $\lambda$ the inverse of $q$ modulo $p$, where $1\le\mu\le q-1$ and $1\le \lambda\le p-1$. Then
\begin{equation}\Phi_{pq}(x)=(1+x^p+\dotsb+x^{p(\mu-1)})(1+x^q+\dotsb+x^{q(\lambda-1)})-x(1+x^p+\dotsb+x^{p(q-\mu-1)})(1+x^q+\dotsb+x^{q(p-\lambda-1)}).\label{pq}\end{equation}\end{prop}
We will prove generalizations of this proposition as Lemma \ref{pq(1+x)lem} and Proposition \ref{pseudopq(1+x)prop} and special cases of those lemmas as Corollary \ref{pq(1+x)q=1cor} and Proposition \ref{pseudopq(1+x)q=1prop}.

Note that this shows that $A(pq)=1$, since all coefficients are $+1$, $-1$ or $0$. The lemma can be illustrated as follows: At the position $(a,b)$ in the first quadrant with $0\le a\le q-1$ and $0\le b\le p-1$ write the residue of $ap+bq$ modulo $pq$ between $0$ and $pq-1$. This produces all $pq$ residues because $p$ and $q$ are relatively prime. Draw a line to the left of $a=\mu$ and below $b=\lambda$. The diagram for $p=5,q=7$ is shown below.
\begin{center}\begin{tabular}{ccc|cccc}
28&33&3&8&13&18&23\\
21&26&31&1&6&11&16\\\hline
14&19&24&29&34&4&9\\
7&12&17&22&27&32&2\\
0&5&10&15&20&25&30
\end{tabular}\end{center}
Since $pq+1=p\mu+q\lambda$, $1$ is always just above and to the right of both lines. The residue $k$ is in the lower left quadrant if and only if $[x^k]\Phi_{pq}(x)=1$, and in the upper right quadrant if and only if $[x^k]\Phi_{pq}(x)=-1$. Therefore, in this example,
\[\Phi_{35}(x)=x^{24}-x^{23}+x^{19}-x^{18}+x^{17}-x^{16}+x^{14}-x^{13}+x^{12}-x^{11}+x^{10}-x^8+x^7-x^6+x^5-x+1.\]

In this example, the coefficients alternate in sign. This can be proven in general as follows: Multiply $\Phi_{pq}(x)$ by the power series $1+x+x^2+\dotsb=\frac1{1-x}$. We have $[x^n](\Phi_{pq}(x)(1+x+x^2+\dotsc))=\sum_{k=0}^n[x^k]\Phi_{pq}(x)$. By equation \eqref{Mobiuspq}, the product is also \[\frac{\Phi_{pq}(x)}{1-x}=\frac{1-x^{pq}}{(1-x^p)(1-x^q)}=(1+x^p+x^{2p}+\dotsb+x^{p(q-1)})(1+x^q+x^{2q}+\dotsb).\]
The $q$ exponents in the left sum each have different residues modulo $q$ since $(p,q)=1$. Hence, no two termwise products are equal, and this power series has only 0 and 1 as coefficients. Therefore, the nonzero coefficients of $\Phi_{pq}(x)$, obtained as the difference of two distinct consecutive terms, alternate $+1,-1,+1,\dotsc,+1$, as desired. Note that this also shows that $A(pq)=1$ again, as Proposition \ref{Ldiagramprop} does explicitly.

We refer to this diagram as the L diagram, abbreviating Lam, Leung and Lenstra. In dealing with various multiples of $\Phi_{pq}(x)$, it will often be helpful to locate the exponents of the positive and negative terms in order to classify these as flat or not flat. Proofs of such will consist of showing that sets of exponents are distinct or not, but can be rather unintelligible without an L diagram in mind. Where relevant, we will indicate where these are useful.

\section{A New Approach}

The L diagram characterization could also be written in the following way:
\[\Phi_{pq}(x)=(1+x^p+\dotsb+x^{p(\mu-1)})(1+x^q+\dotsb+x^{q(p-1)})-x(1+x^p+\dotsb+x^{p(q-1)})(1+x^q+\dotsb+x^{q(p-\lambda-1)}).\]
In the L diagram, this corresponds to adding all terms with exponents to the left of $a=\mu$ and subtracting those at or above $b=\lambda$, so the terms with exponents in the upper left quadrant cancel out, leaving the desired quadrants. In this way, $\Phi_{pq}(x)$ is written as a linear combination of the (relatively sparse) polynomials $1+x^p+\dotsb+x^{p(q-1)}=\Phi_q(x^p)$ and $1+x^q+\dotsb+x^{q(p-1)}=\Phi_p(x^q)$. We will see that this observation generalizes.

First, we note that $x^n-1=\Phi_1(x^n)$, which suggests this generalization of equation \eqref{x^n-1}.
\begin{prop}\label{mnprop}Let $m$ and $n$ be relatively prime positive integers. Then
\begin{equation}\Phi_m(x^n)=\prod_{d\mid n}\Phi_{md}(x).\label{md}\end{equation}\end{prop}
\begin{proof}This follows by applying equation \eqref{x^n-1} and induction on $m$, or from repeated application of Proposition \ref{npprop}.\end{proof}

The basis for this new approach is the following proposition, written here in its most general form.
\begin{prop}\label{gcdprop}Let $d_1,d_2,\dotsc,d_k$ be pairwise relatively prime positive integers with $d_1d_2\dotsm d_k=n$. Then
\begin{equation}\Phi_n(x)=\gcd\{\Phi_{d_i}(x^{n/d_i})\}_{i=1}^k.\label{n/d_i}\end{equation}\end{prop}
\begin{proof}By Proposition \ref{mnprop}, $\Phi_{d_i}(x^{n/d_i})=\prod_{dd_i\mid n}\Phi_{dd_i}(x)$. Again, the cyclotomic polynomials are irreducible, so the greatest common divisor of the $\Phi_{d_i}(x^{n/d_i})$ is the only cyclotomic polynomial appearing in all products, $\Phi_n(x)$, as desired.\end{proof}
Naturally, since $\Phi_p(x)$ is very easy to deal with, we will be interested in the following special cases.

\begin{cor}\label{gcdcor}Let $n>1$ be a positive integer. Then $\Phi_n(x)$ is the greatest common divisor of the polynomials $1+x^{n/p}+\dotsb+x^{(p-1)n/p}$ where $p$ is a prime dividing $n$.\end{cor}
\begin{proof}We have $1+x^{n/p}+\dotsb+x^{(p-1)n/p}=\Phi_p(x^{n/p})$. The result follows from application of Proposition \ref{gcdprop}, when the $d_i$ are prime.\end{proof}

\begin{cor}\label{npcor}Let $n$ be a positive integer, and let $p$ be a prime not dividing $n$. Then $\Phi_{np}(x)=\gcd(1+x^n+\dotsb+x^{n(p-1)},\Phi_n(x^p))$.\end{cor}
\begin{proof}This is merely Proposition \ref{gcdprop} with $k=2$, $d_1=n$ and $d_2=p$. Also, note that it follows from Corollary \ref{gcdcor} applied both to $n$ and to $np$.\end{proof}

By the Extended Euclidean Algorithm in the Euclidean Domain $\Z[x]$, these greatest common divisors can be written as linear combinations in $\Z[x]$ of the given polynomials, as in the motivating example of $\Phi_{pq}(x)$. Knowing the coefficients in the linear combination provides information about $\Phi_n(x)$, and forms the basis of this method. Our method will be entirely based on Corollary \ref{npcor}. Although it seems unlikely, it might be possible to utilize other consequences of Proposition \ref{gcdprop}. Before establishing the method, however, we generalize the class of polynomials under scrutiny.

\section{Pseudocyclotomic Polynomials}

Nowhere in the discussion of binary cyclotomic polynomials did we refer to the fact that $p$ and $q$ are both primes. In fact, taking equation \eqref{Mobiuspq} as the definition of $\Phi_{pq}(x)$, Proposition \ref{Ldiagramprop} still completely characterizes the analogous polynomials, which Liu \cite{Liu} named the ``pseudocyclotomic polynomials'' and which Bachman \cite{Bach11} calls the ``inclusion-exclusion polynomials.'' To highlight their relationship to the cyclotomic polynomials, we adopt the former name. In this section, we flesh out this notion and provide corresponding theorems in this new approach. The key realization is that everything proved in this paper for cyclotomic polynomials also applies to pseudocyclotomic polynomials. Later on, pseudocyclotomic polynomials additionally play a key role in the proof of Theorem \ref{BroadhurstII}. There are also computational consequences as a result of Corollary \ref{Phi(x^w)/Phi(x)cor} below.

First, let us define the pseudocyclotomic polynomials.

\begin{defn}For pairwise relatively prime positive integers $p_1,p_2,\dotsc,p_k$, define the \emph{pseudocyclotomic polynomial of order $k$}, $\tilde\Phi_{p_1,p_2,\dotsc,p_k}(x)$, inductively using $\tilde\Phi(x)=x-1$ (with no subscripts) an analog of Proposition \ref{npprop}:
\begin{equation}\tilde\Phi_{p_1,\dotsc,p_k}(x)=\frac{\tilde\Phi_{p_1,\dotsc,p_{k-1}}(x^{p_k})}{\tilde\Phi_{p_1,\dotsc,p_{k-1}}(x)}.\label{pseudodef}\end{equation}
Alternatively, they can be defined by the resulting general formula, the analog to equation \eqref{Mobius}:
\begin{equation}\tilde\Phi_{p_1,\dotsc,p_k}(x)=\prod_{I\subseteq[k]}(x^{\prod_{i\in I}p_i}-1)^{(-1)^{k-|I|}},\label{pseudoMobius}\end{equation}
where $[k]=\{1,2,\dots,k\}$.\end{defn}

Note that as a consequence of this second formula, the order of the subscripted $p_i$ does not matter. If all the $p_i$ are prime, equations \eqref{pseudodef} and \eqref{pseudoMobius} reduce to Proposition \ref{npprop} and equation \eqref{Mobius}, respectively, so (by induction) $\tilde\Phi_{p_1,\dotsc,p_k}(x)=\Phi_{p_1\dotsm p_k}(x)$. Therefore, the pseudocyclotomic polynomials are a strict generalization of the cyclotomic polynomials, at least the relevant ones ($\Phi_n(x)$ for squarefree $n$).

We have an analogous set of facts for the pseudocyclotomic polynomials:
\begin{itemize}
\item From equation \eqref{pseudoMobius}, $\tilde\Phi_{p_1,\dotsc,p_k}(x)$ has degree $(p_1-1)\dotsm(p_k-1)$ and is reciprocal.
\item If $k>0$, $\tilde\Phi_{p_1,\dotsc,p_k}(0)=1$.
\item If $k>1$, $\tilde\Phi_{p_1,\dotsc,p_k}(1)=1$, while $\tilde\Phi_p(1)=p$.
\item If any $p_i=1$, then $\tilde\Phi_{p_1,\dotsc,p_k}(x)=1$.
\item The analog of equation \eqref{p} is clear from equation \eqref{pseudoMobius}, and applies for all integers $p$:
\begin{equation}\tilde\Phi_p(x)=\frac{x^p-1}{x-1}=1+x+\dotsb+x^{p-1}.\label{pseudop}\end{equation}
\item The analog of Proposition \ref{Ldiagramprop} is the following. The proof is identical.
\begin{prop}\label{pseudoLdiagramprop}Let $p$ and $q$ be relatively prime positive integers, $\mu$ the inverse of $p$ modulo $q$, and $\lambda$ the inverse of $q$ modulo $p$, where $1\le\mu\le q-1$ and $1\le\lambda\le p-1$. Then
\begin{equation}\tilde\Phi_{p,q}(x)=(1+x^p+\dotsb+x^{p(\mu-1)})(1+x^q+\dotsb+x^{q(\lambda-1)})-x(1+x^p+\dotsb+x^{p(q-\mu-1)})(1+x^q+\dotsb+x^{q(p-\lambda-1)}).\label{pseudopq}\end{equation}\end{prop}
\item Finally, we can define the analogous inverse pseudocyclotomic polynomials $\tilde\Psi_{p_1,\dotsc,p_k}(x)=\dfrac{x^{p_1\dotsm p_k}-1}{\tilde\Phi_{p_1,\dotsc,p_k}(x)}$, which then are monic with constant term $-1$. If $p<q$ are relatively prime, $\tilde\Psi_{p,q}(x)=-1-x-\dotsb-x^{p-1}+x^q+x^{q+1}+\dotsb+x^{p+q-1}$.
\end{itemize}

One of the few properties of the cyclotomic polynomials that the pseudocyclotomic polynomials do not generally share is irreducibility. Not surprisingly, the pseudocyclotomic polynomials are products of cyclotomic polynomials:
\begin{prop}\label{pseudoformula}Let $p_1,p_2,\dotsc,p_k$ be pairwise relatively prime positive integers. Then
\begin{equation}\tilde\Phi_{p_1,\dotsc,p_k}(x)=\prod_{\substack{m_i\mid p_i\\m_i>1}}\Phi_{m_1\dotsm m_k}(x).\label{mi|pi}\end{equation}\end{prop}
\begin{proof}We induct on $k$. For the base case of $k=1$, note that
\[\tilde\Phi_{p_1}(x)=\frac{x^{p_1}-1}{x-1}=\prod_{\substack{m_1\mid p_1\\m_1>1}}\Phi_{m_1}(x),\]
as desired. Then
\[\tilde\Phi_{p_1,\dotsc,p_k}(x)=\frac{\tilde\Phi_{p_1,\dotsc,p_{k-1}}(x^{p_k})}{\tilde\Phi_{p_1,\dotsc,p_{k-1}}(x)}=\prod_{\substack{m_i\mid p_i\\m_i>1}}\frac{\Phi_{m_1\dotsm m_{k-1}}(x^{p_k})}{\Phi_{m_1\dotsm m_{k-1}}(x)}=\prod_{\substack{m_i\mid p_i\\m_i>1}}\frac{\prod_{m_k\mid p_k}\Phi_{m_1\dotsm m_k}(x)}{\Phi_{m_1\dotsm m_{k-1}}(x)}=\prod_{\substack{m_i\mid p_i\\m_i>1}}\Phi_{m_1\dotsm m_k}(x),\]
as desired.\end{proof}

Finally, to use this new approach, we need to establish an analog of Corollary \ref{npcor}. We choose not to generalize Proposition \ref{gcdprop} since this would be needlessly notationally complicated. Instead, we generalize Corollary \ref{gcdcor} to something very similar.
\begin{prop}\label{pseudogcd}Let $p_1,p_2,\dotsc,p_k$ be pairwise relatively prime positive integers and $n=p_1p_2\dotsm p_k$. Then $\tilde\Phi_{p_1,\dotsc,p_k}(x)$ is the greatest common divisor of the polynomials $1+x^{n/p_i}+\dotsb+x^{n(p_i-1)/p_i}$, $i=1,2,\dotsc,k$.\end{prop}
\begin{proof}We note that
\[1+x^{n/p_i}+\dotsb+x^{n(p_i-1)/p_i}=\tilde\Phi_{p_i}(x^{n/p_i})=\prod_{\substack{m\mid p_i\\m>1}}\Phi_m(x^{n/p_i})=\prod_{\substack{m\mid p_i\\m>1}}\prod_{l\mid n/p_i}\Phi_{ml}(x)=\prod_{\substack{m'\mid n\\(p_i,m')>1}}\Phi_{m'}(x).\]
We have made the substitution $m'=ml$; note that since $(p_i,n/p_i)=1$, $m=(p_i,m')$ and $l=(n/p_i,m')$ can be recovered from $m'$. In this form, we can take advantage of the fact that the cyclotomic polynomials are irreducible, so this gives a factorization of $1+x^{n/p_i}+\dotsb+x^{n(p_i-1)/p_i}$ into irreducible polynomials. Therefore, the greatest common divisor of the polynomials $1+x^{n/p_i}+\dotsb+x^{n(p_i-1)/p_i}$ is the product of those cyclotomic polynomials that appear in all factorizations. Since the $p_i$ are relatively prime, any divisor $m'\mid n$ can be decomposed uniquely as $m'=m_1m_2\dotsm m_k$ where $m_i\mid p_i$ for all $i=1,2,\dotsc,k$. Then $\Phi_{m_1\dotsm m_k}(x)\mid(1+x^{n/p_i}+\dotsb+x^{n(p_i-1)/p_i})$ for all $i$ if and only if $m_i>1$ for all $i$. That is,
\[\gcd\{1+x^{n/p_i}+\dotsb+x^{n(p_i-1)/p_i}\}=\prod_{\substack{m_i\mid p_i\\m_i>1}}\Phi_{m_1\dotsm m_k}(x).\]
This is exactly the formula for $\tilde\Phi_{p_1,\dotsc,p_k}(x)$ given in Proposition \ref{pseudoformula}, as desired.\end{proof}

We can then derive the analog of Corollary \ref{npcor} from Proposition \ref{pseudogcd}, in a completely analogous fashion.

\begin{cor}\label{pseudonpcor}Let $p_1,p_2,\dotsc,p_k$ and $p$ be relatively prime positive integers, and $n=p_1p_2\dotsb p_k$. Then $\tilde\Phi_{p_1,\dotsc,p_k,p}(x)=\gcd(1+x^n+\dotsb+x^{n(p-1)},\tilde\Phi_{p_1,\dotsc,p_k}(x^p))$.\end{cor}

To date, the pseudocyclotomic polynomials have not appeared in print, so this notation is new. It allows one to observe which numbers (the $p_i$) are acting as ``pseudoprimes'' and clearly distinguishes (with the tilde) pseudocyclotomic polynomials from the cyclotomic polynomials. Notice that the placement of commas is essential: $\tilde\Phi_{pq}(x)$ and $\tilde\Phi_{p,q}(x)$ are distinct. The former is $\frac{x^{pq}-1}{x-1}$ and the latter is $\frac{(x^{pq}-1)(x-1)}{(x^p-1)(x^q-1)}$.

With that diversion, we continue with our general method. Where appropriate, we will point out how the various results concerning the cyclotomic polynomials correspond to pseudoresults about the pseudocyclotomic polynomials. Since the pseudocyclotomic polynomials are a strict generalization of the cyclotomic polynomials, we could more generally write all results in terms of them. However, the notation is easier and the connections to previous results much more obvious if we write the results in terms of cyclotomic polynomials, so we do so and relegate the pseudocyclotomic polynomials to remarks following the major definitions and theorems.

In Section \ref{BroadhurstIIsection}, we will actually utilize one of these pseudoresults to achieve new results about cyclotomic polynomials, demonstrating that the pseudocyclotomic polynomials are useful to define and study.

\section{General Setup}
Let $n$ be a positive integer and let $p$ be a prime not dividing $n$. Define
\[f_{n,p}(x)=\Phi_{np}(x);\quad g_{n,p}(x)=\Phi_p(x^n)=1+x^n+\dotsb+x^{n(p-1)};\quad h_{n,p}(x)=\Phi_n(x^p).\]
Throughout the rest of this paper, where the terminology is unambiguous, we drop one or both of the subscripts.

By the Extended Euclidean Algorithm on $\Z[x]$ and Corollary \ref{npcor}, there are unique polynomials $a_{n,p}(x)$ and $b_{n,p}(x)$ of minimal degree with
\begin{equation}f_{n,p}(x)=a_{n,p}(x)g_{n,p}(x)+b_{n,p}(x)h_{n,p}(x).\label{f=ag+bh}\end{equation}
To clarify, $(a(x),b(x))$ is uniquely defined if it satisfies this equation and either of the following hold.
\begin{align}
\deg a(x)&<\deg(h(x)/f(x))=\deg\Phi_n(x^p)-\deg\Phi_{np}(x)=\varphi(n)p-\varphi(n)(p-1)=\varphi(n),\label{adegree}\\
\deg b(x)&<\deg(g(x)/f(x))=n(p-1)-\varphi(n)(p-1)=(n-\varphi(n))(p-1).\label{bdegree}
\end{align}
Naturally, equation \eqref{adegree} will typically be easier to prove, so we will use that.

Motivated by the fact that $h(x)$ only has terms of degrees that are multiples of $p$, we sort the terms on both sides of equation \eqref{f=ag+bh} according to the residue class of their exponents modulo $p$. Define two families of polynomials $\{F_{n,p,j}(x)\}_{j=0}^{p-1}$ and $\{G_{n,p,j}(x)\}_{j=0}^{p-1}$ by
\[F_j(x)=\sum_{i\ge0}x^i[x^{ip+j}]f(x)\text{ and }G_j(x)=\sum_{i\ge0}x^i[x^{ip+j}](a(x)g(x)),\]
so that
\[f(x)=\sum_{j=0}^{p-1}x^jF_j(x^p)\text{ and }a(x)g(x)=\sum_{j=0}^{p-1}x^jG_j(x^p).\]

For pseudocyclotomic analogs, let
\begin{align*}
\tilde f_{p_1,\dotsc,p_k,p}(x)&=\tilde\Phi_{p_1,\dotsc,p_k,p}(x);\quad
\tilde g_{p_1,\dotsc,p_k,p}(x)=\tilde\Phi_p(x^{p_1\dotsm p_k})=1+x^{p_1\dotsm p_k}+\dotsb+x^{p_1\dotsm p_k(p-1)};\\
\tilde h_{p_1,\dotsc,p_k,p}(x)&=\tilde\Phi_{p_1,\dotsc,p_k}(x^p);\quad
\tilde f(x)=\tilde a(x)\tilde g(x)+\tilde b(x)\tilde h(x);\\
\tilde F_j(x)&=\sum_{i\ge0}x^i[x^{jp+i}]\tilde f(x);\quad
\tilde G_j(x)=\sum_{i\ge0}x^i[x^{jp+i}]\tilde g(x).
\end{align*}

Now we will reveal some of the properties of these polynomials, which we will eventually use to prove facts about $f(x)$. Adding tildes to all polynomials and replacing $n$ with $p_1\dotsm p_k$ and $\varphi(n)$ with $(p_1-1)\dotsm(p_k-1)$ yields the corresponding statements for pseudocyclotomic polynomials.

Write $b(x)\Phi_n(x^p)=f(x)-a(x)g(x)=\sum_{j=0}^{p-1}x^j(F_j(x^p)-G_j(x^p))$. Consider the terms on each side of this equation with exponents congruent to $j$ modulo $p$, where $0\le j\le p-1$. That is, apply $\sum_{i\ge0}x^{ip+j}[x^{ip+j}]$ to both sides. This yields
\[\displaystyle\left(\sum_{i\ge0}x^{ip}[x^{ip+j}]b(x)\right)x^j\Phi_n(x^p)=x^j(F_j(x^p)-G_j(x^p)),\]
so we must have $F_j(x^p)\equiv G_j(x^p)\pmod{\Phi_n(x^p)}$, or with a change of variables,
\begin{equation}F_j(x)\equiv G_j(x)\pmod{\Phi_n(x)}.\label{F=G}\end{equation}

Next, we have some degree conditions on the $F_j(x)$. Since $\deg f(x)=\varphi(np)=\varphi(n)(p-1)$, we get $\deg x^jF_j(x^p)=j+p\deg F_j(x)\le \varphi(n)(p-1)$, or
\begin{equation}\deg F_j(x)\le\varphi(n)-\frac{\varphi(n)+j}p,\qquad0\le j\le p-1.\label{degF}\end{equation}
In particular, $\deg F_j(x)<\varphi(n)=\deg\Phi_n(x)$. In other words, we have the following result.
\begin{prop}\label{Funiqueprop}$F_j(x)$ is the unique polynomial with degree less than $\varphi(n)$ congruent to $G_j(x)$ modulo $\Phi_n(x)$.\end{prop}
In the case that $p>\varphi(n)$, we can simplify equation \eqref{degF} as follows:
\begin{equation}\deg F_j(x)\le\begin{cases}\varphi(n)-1\qquad\text{if }0\le j\le p-\varphi(n),\\\varphi(n)-2\qquad\text{if }p-\varphi(n)+1\le j\le p-1.\end{cases}\label{degFp>n}\end{equation}

For $p>n$, a periodic structure develops in the $G_j$:
\begin{prop}\label{G=Gprop}For all $0\le j<p-n$,
\begin{equation}G_j(x)=G_{n+j}(x).\label{G=G}\end{equation}\end{prop}
\begin{proof}We will show that $x^{n+j}G_j(x^p)=x^{n+j}G_{n+j}(x^p)$, using their definitions. In short, we show that the terms of $a(x)g(x)=a(x)+a(x)x^n+\dotsb+a(x)x^{n(p-1)}$ collected by $x^jG_j(x^p)$ and $x^{n+j}G_{n+j}(x^p)$ are the same, except that the latter terms have degrees that are $n$ larger.

First note that since $\deg a(x)<\varphi(n)$, $[x^k]a(x)=0$ for $k\ge n$. Therefore, by definition,
\begin{align*}
x^{n+j}G_{n+j}(x^p)&=\sum_{k\equiv n+j\:(p)}x^k[x^k](a(x)+a(x)x^n+\dotsb+a(x)x^{n(p-1)})\\
&=\sum_{k\equiv n+j\:(p)}x^k[x^k](a(x)x^n+\dotsb+a(x)x^{n(p-1)}).
\end{align*}

We can similarly remove the last summand from consideration for $x^jG_j(x^p)$. In fact, $[x^k](a(x)x^{n(p-1)})=0$ unless $n(p-1)\le k\le n(p-1)+\varphi(n)-1$. Reducing modulo $p$, as $p>n$, $\sum_{k\equiv j\:(p)}x^k[x^k](a(x)x^{n(p-1)})\neq0$ only if $j\in\{p-n,p-n+1,\dotsc,p-n+\varphi(n)-1\}\pmod p$. Since $j<p-n$ and $p-n+\varphi(n)-1<p$, $j\not\in\{p-n,p-n+1,\dotsc,p-n+\varphi(n)-1\}\pmod p$ and $\sum_{k\equiv n+j\:(p)}x^k[x^k](a(x)x^{n(p-1)})=0$. Therefore,
\begin{align*}
x^jG_j(x^p)&=\sum_{k\equiv j\:(p)}x^k[x^k](a(x)+a(x)x^n+\dotsb+a(x)x^{n(p-2)}),\\
x^{n+j}G_j(x^p)&=\sum_{k\equiv j\:(p)}x^{n+k}[x^{n+k}](a(x)x^n+a(x)x^{2n}+\dotsb+a(x)x^{n(p-1)})\\
&=\sum_{k\equiv n+j\:(p)}x^k[x^k](a(x)x^n+a(x)x^{2n}+\dotsb+a(x)x^{n(p-1)})=x^{n+j}G_{n+j}(x^p),
\end{align*}
as desired.\end{proof}

Applying Proposition \ref{Funiqueprop} yields the following:
\begin{cor}\label{F=Fcor}For all $0\le j<p-n$,
\begin{equation}F_j(x)=F_{n+j}(x).\label{F=F}\end{equation}\end{cor}

We can extend this periodicity to all $j$ and include the case $p<n$ most naturally by defining additional polynomials. For $j<0$, define inductively
\begin{align}G_j(x)&=xG_{j+p}(x),\label{Gextended}\\
F_j(x)&=xF_{j+p}(x).\label{Fextended}
\end{align}
We define these polynomials so that $x^jG_j(x^p)=x^{j+p}G_{j+p}(x^p)$ and $x^jF_j(x^p)=x^{j+p}F_{j+p}(x^p)$. Equations \eqref{F=G} and \eqref{degF} still hold, although we no longer have $\deg F_j(x)<\phi(n)$ for $j\le-\phi(n)$. We generalize Proposition \ref{G=Gprop} with the following:
\begin{prop}\label{G=GF=Fprop}For all $j<p-n$,
\begin{align}
G_j(x)&\equiv G_{n+j}(x)\pmod{x^n-1},\label{G=Ggeneral}\\
F_j(x)&\equiv F_{n+j}(x)\pmod{\Phi_n(x)}.\label{F=Fgeneral}
\end{align}\end{prop}
\begin{proof}We prove equation \eqref{G=Ggeneral} first. The proof is similar to that of Proposition \ref{G=Gprop}. For any $j<p-n$,
\begin{align*}
x^{n+j}G_{n+j}(x^p)&=\sum_{k\equiv n+j\:(p)}x^k[x^k](a(x)g(x)),\\
x^{n+j}G_j(x^p)&=x^n\sum_{k\equiv j\:(p)}x^k[x^k](a(x)g(x))=
\sum_{k\equiv n+j\:(p)}x^k[x^k](a(x)x^ng(x))\\
&=\sum_{k\equiv n+j\:(p)}x^k[x^k](a(x)g(x)+a(x)(x^{pn}-1))\\
&=x^{n+j}G_{n+j}(x^p)+\sum_{k\equiv n+j\:(p)}x^k[x^k](a(x)x^{pn})-\sum_{k\equiv n+j\:(p)}x^k[x^k]a(x)\\
x^{n+j}(G_j(x^p)-G_{n+j}(x^p))&=(x^{pn}-1)\sum_{k\equiv n+j\:(p)}x^k[x^k]a(x)\\
G_j(x^p)&\equiv G_{n+j}(x^p)\pmod{x^{pn}-1}\\
G_j(x)&\equiv G_{n+j}(x)\pmod{x^n-1}
\end{align*}
as desired. Between the second and third lines, we have used the fact that $(x^n-1)g(x)=x^{pn}-1$. For equation \eqref{F=Fgeneral}, we simply notice that $\Phi_n(x)\mid(x^n-1)$, so by equations \eqref{F=G} and equation \eqref{G=Ggeneral}, $F_j(x)\equiv G_j(x)\equiv G_{n+j}(x)\equiv F_{n+j}(x)\pmod{\Phi_n(x)}$ as desired.\end{proof}

This congruence relation is quite powerful. From a single $F_j(x)$, we can generate all the $F_j(x)$ up to congruence modulo $\Phi(x)$. Since the $F_j(x)$ with nonnegative $j$ have degree less than that of $\Phi(x)$, this determines all of them. More succinctly, the following lemma lists a set of properties that are enough to define the $F_j(x)$:
\begin{lem}\label{Fjlem}Extend a family $\{F_j(x)\}_{j=0}^{p-1}\subset\Z[x]$ by $F_j(x)=xF_{j+p}(x)$ for all $j<0$. If each of these conditions hold:
\begin{enumerate}[(i)]
\item $F_{j-n}(x)\equiv F_j(x)\pmod{\Phi_n(x)}$ for all $0\le j<p$.
\item Equation \eqref{degF} holds for the original polynomials.
\item $F_0(0)=1$.
\end{enumerate}
then $\Phi_{np}(x)=\sum_{j=0}^{p-1}x^jF_j(x^p)$.\end{lem}
\begin{proof}Let $f'(x)=\sum_{j=0}^{p-1}x^jF_j(x^p)$, and assume that (i) holds. First note that when $j<0$, $F_{j-n}(x)=xF_{j-n+p}(x)\equiv xF_{j+p}(x)=F_j(x)\pmod{\Phi_n(x)}$, so in fact, $F_{j-n}(x)\equiv F_j(x)\pmod{\Phi_n(x)}$ for all $j<p$. Recall that $F_j(x)=xF_{j+p}(x)$ implies that $x^jF_j(x^p)=x^{j+p}F_{j+p}(x^p)$, so we can choose any set of representatives of the $p$ congruence classes modulo $p$ to express $f'(x)$. Since $(n,p)=1$, the numbers $\{-in\}_{i=0}^{p-1}$ are representatives of the $p$ congruence classes modulo $p$, and hence
\begin{align*}
f'(x)&=\sum_{j=0}^{p-1}x^jF_j(x^p)=\sum_{i=0}^{p-1}x^{-in}F_{-in}(x^p)\equiv x^{-n(p-1)}(1+x^n+\dotsb+x^{n(p-1)})F_0(x^p)\pmod{\Phi_n(x^p)}\\
&\equiv x^ng(x)F_0(x^p)\pmod{h(x)}.
\end{align*}
As $f(x)\mid g(x)$ and $f(x)\mid h(x)$, this congruence implies that $f(x)\mid f'(x)$. Next, we show that (ii), equation \eqref{degF}, implies that $\deg f'(x)\le\deg f(x)$. For $0\le j\le p-1$,
\[\deg f'(x)=\max_j\{\deg (x^jF_j(x^p))\}=\max_j\{j+p\cdot\deg F_j\}\le\varphi(n)(p-1)=\deg f(x).\]
Therefore, $f'(x)=cf(x)$ for some constant $c$. Evaluating both sides at $x=0$, (iii) yields $c=1$, so if all three conditions hold, $f'(x)=f(x)$ as desired.\end{proof}

Next we derive some replacements for equation \eqref{degF} in condition (ii) of Lemma \ref{Fjlem}. Note that when $p>n$, equation \eqref{degFp>n} is equivalent. Additionally, note that with the extension, $\deg F_{j-p}(x)=1+\deg F_j(x)$ so $\deg F_{j-p}(x)\le\varphi(n)-\frac{\varphi(n)+j-p}p$ if and only if $\deg F_j(x)\le\varphi(n)-\frac{\varphi(n)+j}p$. Therefore, equation \eqref{degF} for $0\le j<p$ is equivalent to equation \eqref{degF} for $1-\varphi(n)\le j\le p-\varphi(n)$. For this latter range, $0<\frac{\varphi(n)+j}p\le1$, so equation \eqref{degF} becomes
\begin{equation}\deg F_j(x)\le\varphi(n)-1,\qquad1-\varphi(n)\le j\le p-\varphi(n).\label{degFshifted}\end{equation}
This gives us three equivalent expressions to satisfy the degree requirement of Lemma \ref{Fjlem}: equations \eqref{degF}, \eqref{degFshifted}, and when $p>n$, \eqref{degFp>n}.

Also, all three of these hold when $\Phi_{np}(x)=\sum_{j=0}^{p-1}x^jF_j(x^p)$. This also lets us extend equation \eqref{F=F}:
\begin{equation}F_j(x)=F_{n+j}(x),\qquad1-\varphi(n)\le j\le p-n-1.\label{F=Fextended}\end{equation}

We now have a general method for proving facts about the $F_j(x)$, and thus $\Phi_{np}(x)$: Guess the solution, and verify that it satisfies the conditions of Lemma \ref{Fjlem}. Any proof that refers to Lemma \ref{Fjlem} will be divided into proving each of the criteria (i), (ii) and (iii).

Recall the general goal: to prove properties of the coefficients of cyclotomic polynomials. As we defined previously, the complete set of coefficients (including 0) of a polynomial $p(x)$ is denoted by $V(p(x))$ and $V(\Phi_m(x))=V_m$. We have partitioned the coefficients so that $V_{np}=\bigcup_{j=0}^{p-1}V(F_j(x))$. Since $F_{j-p}(x)=xF_j(x)$, $V(F_{j-p}(x))=V(F_j(x))$ so we can even write $V_{np}=\bigcup_{j<p}V(F_j(x))$. Alternatively, for $p>n$, we can apply Corollary \ref{F=Fcor} to get
\begin{equation}V_{np}=\bigcup_{j=0}^{n-1}V(F_j(x))=\bigcup_{j=p-n}^{p-1}V(F_j(x)).\label{firstnFj}\end{equation}
In general, with the extended definition, we only need to include one $V(F_j(x))$ for each congruence class of $j$ modulo $n$. Most usefully, we can look at the coefficients of $n$ consecutive $F_j(x)$ to determine $V_{np}$, as in equation \eqref{firstnFj}.

\section{Periodicity}
\subsection{Vanilla Periodicity}
In 2010, Kaplan \cite{Kap10} proved the following:
\begin{thm}\label{periodicitythm}Let $n$ be a positive integer. Let $s,t$ be primes satisfying $n<s<t$ and $s\equiv t\pmod n$. Then $V_{ns}=V_{nt}$.\end{thm}
We will first prove this using our methods, then generalize this result to the most general Theorem \ref{extendedperiodicitythm}. We choose to provide a proof of Theorem \ref{periodicitythm} because the methods (in particular, the lemma we will use) have broader application.
\begin{proof}First let $s<t$ without assuming $n<s$. We claim that the $F_j(x)$ are the same for $s$ as they are for $t$:
\begin{lem}\label{periodicitylem}Let $s<t$ be primes satisfying $s\equiv t\pmod n$. Then for $0\le j<s$, $F_{s,j}(x)=F_{t,j}(x)$.\end{lem}
\begin{proof}Define $F'_{s,j}(x)=F_{t,j}(x)$ for $0\le j<s$. Extend $F'_{s,j}(x)=xF'_{s,j+s}(x)$ for $j<0$ as prescribed by Lemma \ref{Fjlem}. We must check the conditions of the lemma with $p=s$ on $F'_{s,j}(x)$:
\begin{enumerate}[(i)]
\item For $n\le j<s$, $F'_{s,j-n}(x)=F_{t,j-n}(x)\equiv F_{t,j}(x)=F'_{s,j}(x)\pmod{\Phi_n(x)}$. For $0\le j<n$, let $k$ be the unique integer such that $0\le j-n+ks<s$. Then
\[F'_{s,j-n}(x)\equiv x^kF'_{s,j-n+ks}(x)=x^kF_{t,j-n+ks}(x)=F_{t,j-n+k(s-t)}(x)\equiv F_{t,j}(x)=F'_{s,j}(x)\pmod{\Phi_n(x)}\]
since $n\mid(-n+k(s-t))$, as desired.
\item We will verify equation \eqref{degFshifted}. For $1-\varphi(n)\le j\le s-\varphi(n)$, let $k\in\Z$ satisfy $0\le j+ks<s$. Then $F'_{s,j}(x)=x^kF'_{s,j+ks}(x)=x^kF_{t,j+ks}(x)$. We claim that this is equal to $F_{t,j}(x)$. Indeed, for $1\le l\le k$, $1-\varphi(n)\le j+ls<s$ so equation \eqref{F=Fextended} implies $F_{t,j+ls}(x)=F_{t,j+(l-1)s+t}(x)$. Multiplying by $x^l$ and applying equation \eqref{Fextended}, $x^lF_{t,j+ls}(x)=x^lF_{t,j+(l-1)s+t}(x)=x^{l-1}F_{t,j+(l-1)s}(x)$. Therefore, $F'_{s,j}(x)=x^kF_{t,j+ks}(x)=F_{t,j}(x)$. Therefore, $\deg F'_{s,j}(x)=\deg F_{t,j}(x)\le\varphi(n)-1$ for all $1-\varphi(n)\le j\le s-\varphi(n)$, as desired.
\item $F'_{s,0}(0)=F_{t,0}(0)=1$, as desired.\qedhere
\end{enumerate}\end{proof}
Note that Lemma \ref{periodicitylem} holds even when $s<n$. To prove Theorem \ref{periodicitythm}, however, we need $n<s<t$. From equation \eqref{firstnFj}, \[V_{nt}=\bigcup_{j=0}^{t-1}V(F_{t,j}(x))=\bigcup_{j=0}^{n-1}V(F_{t,j}(x))=\bigcup_{j=0}^{n-1}V(F_{s,j}(x))=\bigcup_{j=0}^{s-1}V(F_{s,j}(x))=V_{ns},\]
as desired.\end{proof}

\subsection{Signed Periodicity}
Additionally, this method allows us to prove a similar result in the case of $s\equiv-t\pmod n$, which is a new result.

\begin{thm}\label{signedperiodicitythm}Let $n$ be a positive integer. Let $s,t$ be primes satisfying $n<s<t$ and $s\equiv-t\pmod n$. Then $V_{ns}=-V_{nt}$.\end{thm}
\begin{proof}This time we start with the assumption that $n<s<t$. Again, we utilize a lemma to relate the $F_j(x)$.
\begin{lem}\label{signedperiodicitylem}Let $s,t$ be primes such that $n<s<t$ and $s\equiv-t\pmod n$. Then for $0\le j<s$,
\begin{equation}F_{s,j}(x)=\begin{cases}
-F_{t,t+s-\varphi(n)-j}(x)&\text{ if }s-\varphi(n)+1\le j\le s-1\\
-F_{t,t-n+s-\varphi(n)-j}(x)&\text{ if }s-n\le j\le s-\varphi(n)\\
F_{s,j+n}(x)&\text{ if }0\le j\le s-n-1.
\end{cases}\label{s=-tF}\end{equation}\end{lem}
\begin{proof}Let the right side of equation \eqref{s=-tF} be $F'_{s,j}(x)$. Intuitively, this definition makes $F'_{s,j}(x)$ into the same $F_{t,j}(x)$ in the opposite order with a sign change, which we might expect to work because $s\equiv-t\pmod n$. The first two cases describe $F'_{s,j}(x)$ for the largest $n$ values of $j$, $s-n\le j\le s-1$ and the last case basically applies Corollary \ref{F=Fcor} to define the rest of them.

First we must show that the polynomials in this definition have appropriate subscripts. If $s-\varphi(n)+1\le j\le s-1$, then $t-\varphi(n)+1\le t+s-\varphi(n)-j\le t-1$, which is appropriate. If $s-n\le j\le s-\varphi(n)$, then $t-n\le t-n+s-\varphi(n)-j\le t-\varphi(n)$, which is also appropriate since $n<t$. Notice that these are the last $n$ of the $F_{t,j}(x)$. We have shown that $F'_{s,j}$ is well-defined for $0\le j<s$. Extend this with $F'_{s,j}(x)=xF'_{s,j+s}(x)$ for $j<0$ as usual. We want to show that $F_{s,j}(x)=F'_{s,j}(x)$. As in the proof of Theorem \ref{periodicitythm}, we show that $\{F'_{s,j}(x)\}$ satisfy the conditions of Lemma \ref{Fjlem}.
\begin{enumerate}[(i)]
\item Proving that $F'_{s,j}(x)\equiv F'_{s,j-n}(x)\pmod{\Phi_n(x)}$ for $0\le j<s$ is the difficult part. For $j\ge n$, we have $F'_{s,j-n}(x)=F'_{s,j}(x)$ by definition. When $0\le j\le n-\varphi(n)$, $n\mid(s+t-n)$, so
\begin{align*}
F'_{s,j-n}(x)&=xF'_{s,j-n+s}(x)=-xF_{t,t-\varphi(n)-j}(x)\\
&=-F_{t,-\varphi(n)-j}(x)\equiv-F_{t,t-n+s-\varphi(n)-j}(x)\pmod{\Phi_n(x)}\\
&=F'_{s,j}(x),
\end{align*}
as desired. Finally, when $n-\varphi(n)+1\le j\le n-1$,
\begin{align*}
F'_{s,j-n}(x)&=xF'_{s,j-n+s}(x)=-xF_{t,t+n-\varphi(n)-j}(x)\\
&=-F_{t,n-\varphi(n)-j}(x)\equiv-F_{t,t+s-\varphi(n)-j}(x)\pmod{\Phi_n(x)}\\
&=F'_{s,j}(x),
\end{align*}
as desired.
\item We verify equation \eqref{degFp>n}. First, if $j\ge0$, $\deg F'_{s,j}(x)\le\varphi(n)-1$ by definition. If $s-\varphi(n)+1\le j\le s-1$, then $t-\varphi(n)+1\le t+s-\varphi(n)-j\le t-1$ so $\deg F'_{s,j}(x)=\deg F_{t,t+s-\varphi(n)-j}(x)\le\varphi(n)-2$. Therefore, equation \eqref{degFp>n} is satisfied.
\item Let $w\equiv s\pmod n$ with $1\le w<n$. Then $n\mid s-w$ so $F'_{s,0}(0)=F'_{s,s-w}(0)$. If $w<\varphi(n)$, then $F'_{s,s-w}(0)=-F_{t,t+w-\varphi(n)}(0)=-F_{t,n-\varphi(n)}(0)$. Otherwise, $F'_{s,s-w}(0)=-F_{t,t-n+w-\varphi(n)}(0)=-F_{t,n-\varphi(n)}(0)$. Thus, we must show that $[x^{n-\varphi(n)}]f_{n,t}(x)=-1$. Recall equation \eqref{npreciprocal}:
\[f_{n,t}(x)=\Phi_{nt}(x)=-\Psi_n(x)\Phi_n(x^t)(1+x^n+\dotsb).\]
Since $t,n>n-\varphi(n)$ and the constant term of $\Phi_n(x^t)$ is 1, we have $[x^{n-\varphi(n)}]\Phi_{nt}(x)=-[x^{n-\varphi(n)}]\Psi_n(x)$. We know that $\deg\Psi_n(x)=n-\varphi(n)$ and $\Psi_n(x)$ is monic, so $-[x^{n-\varphi(n)}]\Psi_n(x)=-1$, as desired.
\end{enumerate}
Therefore, $F'_{s,j}(x)=F_{s,j}(x)$ for all $j<s$ as desired.\end{proof}
As $n<s<t$, by equation \eqref{firstnFj},
\[V_{ns}=\bigcup_{j=0}^{s-1}V(F_{s,j}(x))=\bigcup_{j=s-n}^{s-1}V(F_{s,j}(x))=\bigcup_{j=t-n}^{t-1}-V(F_{t,j}(x))=-\bigcup_{j=0}^{t-1}V(F_{t,j}(x))=-V_{nt},\]
as desired.\end{proof}

\subsection{Extending Periodicity}
We can utilize properties of the $F_{n,s,j}(x)$ to extend periodicity to $n-\varphi(n)<s<t$. We combine all our periodicity results for cyclotomic polynomials together in the following theorem.

\begin{thm}\label{extendedperiodicitythm}Let $n$ be a positive integer. Let $s,t$ be primes satisfying $n-\varphi(n)<s<t$ and $s\equiv\pm t\pmod n$. Then $V_{ns}=\pm V_{nt}$, with the same signs taken in both $\pm$.\end{thm}
\begin{proof}First consider $s\equiv t\pmod n$. We have $V_{ns}=\bigcup_{j=0}^{s-1}V(F_{s,j}(x))$ and $V_{nt}=\bigcup_{j=0}^{t-1}V(F_{t,j}(x))$. Because $F_{t,j+t}(x)=xF_{t,j}(x)$, we have $V(F_{t,j+t}(x))=V(F_{t,j}(x))$. Since $(n,t)=1$, we have, $V_{nt}=\bigcup_{i=0}^{t-1}V(F_{t,-in}(x))$.

By equation \eqref{F=Fextended} and $V(F_{t,j+t}(x))=V(F_{t,j}(x))$, $F_{t,-in}(x)=F_{t,-(i-1)n}(x)$ unless the residue $k\equiv-in\pmod t$ with $0\le k<t$ satisfies $t-n\le k\le t-\varphi(n)$. If we have $F_{t,-in}(x)=F_{t,-(i-1)n}(x)$, we can ignore $F_{-in}(x)$ in calculating the union. Since $V(F_{t,-in}(x))=V(F_{t,j}(x))$ when $j\equiv-in\pmod t$, we can write $V_{nt}=\bigcup_{j=t-n}^{t-\varphi(n)}V(F_{t,j}(x))$.

Next, we apply the fact that $\Phi_{nt}(x)$ is reciprocal: $[x^m]\Phi_{nt}(x)=[x^{\varphi(nt)-m}]\Phi_{nt}(x)$. Since $\varphi(nt)=(t-1)\varphi(n)$, we have
\begin{align*}
V(F_{t,j}(x))&=\{[x^{j+kt}]\Phi_{nt}(x)\}_{k\in\Z}=\{[x^{(t-1)\varphi(n)-j-kt}]\Phi_{nt}(x)\}_{k\in\Z}\\
&=\{[x^{(t-j-\varphi(n))+t(\varphi(n)-k-1)}]\Phi_{nt}(x)\}_{k\in\Z}=\{[x^{t-j-\varphi(n)+kt}]\Phi_{nt}(x)\}_{k\in\Z}=V(F_{t,t-\varphi(n)-j}(x)).
\end{align*}
When $t-n\le j\le t-\varphi(n)$, $0\le t-\varphi(n)-j\le n-\varphi(n)$. Therefore, $V_{nt}=\bigcup_{j=0}^{n-\varphi(n)}V(F_{t,j}(x))$. With $n-\varphi(n)<s$, we now have
\begin{align}
V_{nt}&=\bigcup_{j=0}^{n-\varphi(n)}V(F_{t,j}(x))=\bigcup_{j=0}^{n-\varphi(n)}V(F_{s,j}(x))\subseteq\bigcup_{j=0}^{s-1}V(F_{s,j}(x))=V_{ns}\label{Vns<Vnt}\\
V_{ns}&=\bigcup_{j=0}^{s-1}V(F_{s,j}(x))=\bigcup_{j=0}^{s-1}V(F_{t,j}(x))\subseteq\bigcup_{j=0}^{t-1}V(F_{t,j}(x))=V_{nt}.\label{Vnt<Vns}
\end{align}
Therefore, $V_{ns}=V_{nt}$ as desired.

The case $s\equiv-t\pmod n$ follows: By Dirichlet's theorem on primes in arithmetic progressions, there exists some prime $s'>n$ such that $s\equiv s'\pmod n$. Again by Dirichlet's theorem, there exists some prime $t'>s'>n$ such that $t\equiv t'\pmod n$. By the extended periodicity we have just proven for $s\equiv s'\pmod n$, $V_{ns}=V_{ns'}$. Then $t'\equiv -s'\pmod n$ and $n<s'<t'$ so by regular periodicity (Theorem \ref{signedperiodicitythm}), $V_{ns'}=-V_{nt'}$. Finally, by extended periodicity for $t\equiv t'\pmod n$, $V_{nt'}=V_{nt}$. Therefore, we have $V_{ns}=-V_{nt}$ as desired.\end{proof}

When $s\le n-\varphi(n)<t$, equation \eqref{Vns<Vnt} does not necessarily hold. For instance, $V(3\cdot5\cdot2)=\{-1,0,1\}$ but $V(3\cdot5\cdot17)=\{-1,0,1,2\}$. Plenty of similar examples with $s=2$ are analyzed in Section \ref{r=2section}. Further research can investigate the cases where $V(ns)\neq V(nt)$. We still have equation \eqref{Vnt<Vns}, though, which implies the following.
\begin{cor}\label{Vns<Vntcor}Let $n$ be a positive integer. Let $s,t$ be primes with $t>n-\varphi(n)$ and $s\equiv\pm t\pmod n$. Then $V_{ns}\subseteq\pm V_{nt}$, with the same signs taken in both $\pm$, and consequently, $A(ns)\le A(nt)$.\end{cor}

\subsection{Periodicity with Pseudocyclotomics}

Everything proven in this section also applies to pseudocyclotomic polynomials, where as usual tildes are placed over all polynomials; $n$ is replaced with $p_1\dotsm p_k$, and $\varphi(n)$ with $(p_1-1)\dotsm(p_k-1)$. Most importantly, the analogs of Lemma \ref{periodicitylem} and Theorem \ref{extendedperiodicitythm} are
\begin{prop}\label{pseudoF=Fprop}Let $p_1,p_2,\dotsc,p_k$ be pairwise relatively prime positive integers and $n=p_1p_2\dotsm p_k$. Let $s<t$ be positive integers relatively prime to $n$ such that $s\equiv t\pmod n$. Then for $0\le j<s$, $\tilde F_{p_1,\dotsc,p_k,s,j}(x)=F_{p_1,\dotsc,p_k,t,j}(x)$.\end{prop}
\begin{prop}\label{pseudoperiodicityprop}Let $p_1,p_2,\dotsc,p_k$ be pairwise relatively prime positive integers and $n=p_1p_2\dotsm p_k$. Let $s,t$ be positive integers relatively prime to $n$ such that $s\equiv\pm t\pmod n$ and $p_1p_2\dotsm p_k-(p_1-1)(p_2-1)\dotsm(p_k-1)<s<t$. Then $V(\tilde\Phi_{p_1,\dotsc,p_k,s}(x))=\pm V(\tilde\Phi_{p_1,\dotsc,p_k,t}(x))$, taking the same signs.\end{prop}
The proofs of these propositions are completely identical to the proofs in the cyclotomic case, so we do not bother repeating them. The following special case is very useful:
\begin{cor}\label{Phi(x^w)/Phi(x)cor}Let $p_1,p_2,\dotsc,p_k$ be primes and $n=p_1p_2\dotsm p_k$. Let $p>n$ be another prime such that $p\equiv w\pmod n$ where $0<w<n$. Then for $0\le j<w$, $F_{n,p,j}(x)=\tilde F_{p_1,\dotsc,p_k,w,j}(x)$.\end{cor}
Note that this allows us to calculate $F_{n,p,0}(x)$ easily without calculating $\Phi_{np}(x)$, which could be computationally slow if $p$ is large. (Of course, Kaplan's periodicity ensures that we would choose the smallest possible $p$, but this might too be large.) From $F_0(x)$ we can calculate $F_j(x)$ for all $j$ using equation \eqref{F=Fgeneral}. We will explicitly use this calculation in Section \ref{BroadhurstIIsection}.

\section{Useful Tools For Investigating $F_j(x)$}

Before finally jumping into the particulars of order 3, 4 or 5 cyclotomic polynomials, we compile in this section several useful general facts about the $F_j(x)$. We will draw upon these facts multiple times in later proofs, so we state them here in their full generality.

\subsection{$p\equiv\pm1\pmod n$}

Our theory treats the case when $p\equiv1\pmod n$ particularly well.

\begin{prop}\label{p=1generalprop}Let $p\equiv1\pmod n$ be a prime. Then $F_{n,p,j}(x)\equiv x^{-j}\pmod{\Phi_n(x)}$. In particular, $F_{n,p,0}(x)=1$.\end{prop}
\begin{proof}Note that $p>n$. To apply Lemma \ref{Fjlem}, we let $F'_j(x)$ be the polynomial with degree less than $\varphi(n)$ congruent to $x^{-j}$ modulo $\Phi_{pq}(x)$ and extend with $F'_{j-p}(x)=xF'_j(x)$. Then we simply check the conditions of Lemma \ref{Fjlem}.
\begin{enumerate}[(i)]
\item If $j\ge n$, $F'_{j-n}(x)\equiv x^{-j+n}\equiv x^{-j}\equiv F'_j(x)\pmod{\Phi_n(x)}$. If $j<n$, $F'_{j-n}(x)=xF'_{j-n+p}(x)\equiv x\cdot x^{-j+n-p}\equiv x^{-j}\equiv F'_j(x)\pmod{\Phi_n(x)}$ as desired.
\item By definition, for $0\le j<p$, $\deg F'_j(x)\le\varphi(n)-1$. For $j\ge p-\varphi(n)+1$, we have $-j\equiv p-1-j\pmod n$ and $p-1-j\le\varphi(n)-2<\varphi(n)$, so $F'_j(x)=x^{p-1-j}$, which has degree at most $\varphi(n)-2$ as desired.
\item $F'_0(x)\equiv1\pmod{\Phi_n(x)}$, so $F'_0(x)=1$, making $F'_0(0)=1$ as desired.\qedhere
\end{enumerate}\end{proof}
Here is a second, much quicker proof using pseudocyclotomic polynomials, illustrating the power of using them.
\begin{proof}Let $n=p_1\dotsm p_k$ be a prime factorization. By Corollary \ref{Phi(x^w)/Phi(x)cor}, where $w=1$, $F_{n,p,0}(x)=\tilde F_{p_1,\dotsc,p_k,1,0}(x)$. As we've previously remarked, $\tilde f_{p_1,\dotsc,p_k,1}(x)=1$, so $\tilde F_{p_1,\dotsc,p_k,1,0}(x)=1$ and $F_{n,p,0}(x)=1$ as desired. Then for $0\le j<n$, $F_{n,p,j}(x)=x^{-j}F_{n,p,-j(p-1)}(x)\equiv x^{-j}F_{n,p,0}(x)=x^{-j}\pmod{\Phi_n(x)}$, as desired.\end{proof}

Lemma \ref{signedperiodicitylem} results in a similar formula when $p\equiv-1\pmod n$:
\begin{prop}\label{p=-1generalprop}Let $p\equiv-1\pmod n$ be a prime. Then $F_{n,p,j}(x)\equiv-x^{j+\varphi(n)}\pmod{\Phi_n(x)}$.\end{prop}
\begin{proof}By Dirichlet's theorem on primes in arithmetic progressions, there exists some prime $q\equiv1\pmod n$. By Proposition \ref{p=1generalprop}, $F_{q,j}(x)\equiv x^{-j}\pmod{\Phi_n(x)}$. Lemma \ref{signedperiodicitylem} yields
\begin{align*}
F_{p,j}(x)&=-F_{q,q+p-\varphi(n)-j}(x)\equiv-x^{-q-p+\varphi(n)+j}\equiv-x^{\varphi(n)+j}\pmod{\Phi_n(x)}\text{ if }p-\varphi(n)+1\le j\le p-1\\
F_{p,j}(x)&=-F_{q,q-n+p-\varphi(n)-j}(x)\equiv-x^{-q+n-p+\varphi(n)+j}\equiv-x^{\varphi(n)+j}\pmod{\Phi_n(x)}\text{ if }p-n\le j\le p-\varphi(n)\\
F_{p,j}(x)&=F_{p,j+n}(x)\equiv-x^{\varphi(n)+j+n}\equiv-x^{\varphi(n)+j}\pmod{\Phi_n(x)}\text{ if }0\le j\le p-n-1,
\end{align*}
since $x^n\equiv1\pmod{\Phi_n(x)}$, as desired.\end{proof}

The pseudocyclotomic analogs are natural, and the proofs are again identical.
\begin{prop}\label{p=1pseudoprop}Let $p_1,p_2,\dotsc,p_k$ be pairwise relatively prime positive integers and $p\equiv1\pmod{p_1\dotsm p_k}$ a positive integer. Then $\tilde F_{p_1,\dotsc,p_k,p,j}(x)\equiv x^{-j}\pmod{\tilde\Phi_{p_1,\dotsc,p_k}(x)}$.\end{prop}
\begin{prop}\label{p=-1pseudoprop}Let $p_1,p_2,\dotsc,p_k$ be pairwise relatively prime positive integers and $p\equiv-1\pmod{p_1\dotsm p_k}$ a positive integer. Then $\tilde F_{p_1,\dotsc,p_k,p,j}(x)\equiv-x^{j+(p_1-1)\dotsm(p_k-1)}\pmod{\tilde\Phi_{p_1,\dotsc,p_k}(x)}$.\end{prop}

\subsection{Reciprocity in general}\label{reciprocity}

As utilized in the above proof of Theorem \ref{extendedperiodicitythm}, we can use the reciprocal property of $\Phi_{np}(x)$ to equate the coefficients of pairs of the $F_j(x)$. We repeat the argument from that section in general, limiting ourselves to $F_j(x)$ for $0\le j<p$. Let $-\varphi(n)=yp+z$ where $0\le z<p$. For $0\le j\le z$,
\begin{align*}
V(F_j(x))&=\{[x^{j+kp}]\Phi_{np}(x)\}_{k\in\Z}=\{[x^{(p-1)\varphi(n)-j-kp}]\Phi_{np}(x)\}_{k\in\Z}\\
&=\{[x^{z-j+p(y+\varphi(n)-k)}]\Phi_{np}(x)\}_{k\in\Z}=\{[x^{z-j+kp}]\Phi_{np}(x)\}_{k\in\Z}=V(F_{z-j}(x)).
\end{align*}
For $z<j<p$,
\begin{align*}
V(F_j(x))&=\{[x^{(p-1)\varphi(n)-j-kp}]\Phi_{np}(x)\}_{k\in\Z}=\{[x^{p+z-j+p(y+\varphi(n)-k-1)}]\Phi_{np}(x)\}_{k\in\Z}\\
&=\{[x^{p+z-j+kp}]\Phi_{np}(x)\}_{k\in\Z}=V(F_{p+z-j}(x)).
\end{align*}
We can therefore summarize this result in the following form:
\begin{prop}\label{reciprocitygeneralprop}Let $p$ be a prime relatively prime to $n$, and $-\varphi(n)=yp+z$ for $0\le z<p$. Then
\begin{equation}V(F_j(x))=\begin{cases}V(F_{z-j}(x))&\qquad0\le j\le z\\V(F_{p+z-j}(x))&\qquad z<j<p.\end{cases}\label{reciprocitygeneral}\end{equation}\end{prop}

\subsection{Which $F_j(x)$ have nonzero constant term?}\label{Fj(0)section}

In this section, we provide a different angle on the $F_j(x)$, which will occasionally prove useful later, for instance in the proofs of Theorems \ref{r=2thm} and \ref{|w|thm}. This angle consists of a reparameterization of the polynomial family $F_j(x)$, which we will denote $F^*_j(x)$.

Assume that prime $p>n$ and consider the first $n$ of the $F_j(x)$:  $\{F_{n,p,j}(x)\}_{j=0}^{n-1}$. By Corollary \ref{F=Fcor}, we already know that this set is identical to $\{F_{n,p,j}(x)\}_{j=0}^{p-1}$. Instead of extending the $F_j(x)$ with $F_j(x)=xF_{j+p}(x)$, we consider the indices $j$ modulo $n$; in other words, let $F_{j+n}(x)=F_j(x)$ for all $j$. Then $F_{j-p}(x)\equiv xF_j(x)\pmod{\Phi_n(x)}$, and $\deg F_j(x)<\varphi(n)$ for all $j$. Since $(p,n)=1$, this set can also be written as $\{F_{jp}(x)\}_{j=0}^{n-1}$, and we have $F_{-(j+1)p}(x)\equiv xF_{-jp}(x)\pmod{\Phi_n(x)}$ for all $j$. Applying this repeatedly gives $F_{-jp}(x)\equiv x^jF_0(x)\pmod{\Phi_n(x)}$, and
\begin{equation}\{F_j(x)\}_{j=0}^{p-1}=\{F_{jp}(x)\}_{j=1-n}^{0}\equiv\{x^jF_0(x)\}_{j\in\Z}\pmod{\Phi_n(x)}.\label{xjF0}\end{equation}

To make this easier to discuss, we introduce some new notation. Let $F^*_j(x)\equiv x^jF_0(x)\pmod{\Phi_n(x)}$ and $\deg F^*_j(x)<\varphi(n)$. That is, with the $F_{j+n}(x)=F_j(x)$ extension, $F^*_j(x)=F_{-jp}(x)$. Extend the $F^*_j(x)$ to all $j\in\Z$ in the same way: $F^*_{j+n}(x)=F^*_j(x)$. Of course, the point of this is that $\{F^*_j(x)\}_{j=0}^{n-1}=\{F_j(x)\}_{j=0}^{n-1}$, so by looking at the $F^*_j(x)$, we discover properties of the $F_j(x)$.

One consequence is that if one can somehow calculate $F_0(x)$, one can then calculate all $F^*_j(x)$ relatively easily. In the sequence $\{F^*_j(x)\}$, we have $F^*_j(x)\equiv xF^*_{j-1}(x)\pmod{\Phi_n(x)}$. To know when to add or subtract $\Phi_n(x)$ and when just to multiply by $x$, one can either use the degrees of the $F_j(x)$ or, equivalently, their constant terms. In other words, since $\deg xF^*_j(x)\le\varphi(n)$, we must have $F^*_j(x)=xF^*_{j-1}(x)+c\Phi_n(x)$ for some constant $c$. Evaluating at $x=0$, $c=F^*_j(0)$, or
\begin{equation}F^*_j(x)=xF^*_{j-1}(x)+F^*_j(0)\Phi_n(x).\label{F'recursion}\end{equation}

Therefore, if $F^*_j(0)=0$, $F^*_j(x)=xF^*_{j-1}(x)$, so $V(F^*_j(x))=V(F^*_{j-1}(x))$. Since the $F^*_j(x)$ are cyclic in $j$, this implies that $\bigcup_jV(F^*_j(x))=\bigcup_{j,F^*_j(0)\neq0}V(F^*_j(x))$. In other words, to compute the coefficients of the $F_j(x)$, we only need to consider those whose constant terms are nonzero. We will use this fact in Section \ref{shortcuts} to prove Theorem \ref{r=2thm}.

Therefore, we are interested in the values of the $F_j(0)$. Since $F_j(0)=[1]F_j(x)=[x^j]\Phi_{np}(x)$, this amounts to determining the smallest degree terms of $\Phi_{np}(x)$. We recall equation \eqref{npreciprocal}:
\[\Phi_{np}(x)=\frac{\Phi_n(x^p)}{\Phi_n(x)}=-\Psi_n(x)\Phi_n(x^p)(1+x^n+x^{2n}+\dotsb).\]
Since we are assuming that $p>n$, for $0\le j<n$, $F_j(0)=[x^j]\Phi_{np}(x)=-[x^j]\Psi_n(x)$. If necessary, we can use equation \eqref{inversenp} to tell us $\Psi_n(x)$.

Note that the values of $F_j(0)$ are independent of $p$, but how those values translate into the $F^*_j(0)$ is not.

Before applying this to calculate the $F_j(x)$ directly, we use this observation to provide another proof of extending periodicity. The polynomial $-\Psi_n(x)$ has degree $n-\varphi(n)$ so if $n-\varphi(n)<j<n$, $F_j(0)=0$. If $F^*_j(0)=0$, $V(F^*_j(x))=V(F^*_{j-1}(x))$, so we can ignore $V(F^*_j(x))$ in computing $V_{np}$. Therefore, $V_{np}=\bigcup_{j=0}^{n-\varphi(n)}V(F_j(x))$. Then equation \eqref{Vnt<Vns} follows and $V_{ns}=V_{nt}$.

We can also use this method to compute particular coefficients of $F_j(x)$. Applying equation \eqref{F'recursion} $k+1$ times, the $x^k$ coefficient of $F_j(x)$ is given by
\begin{equation}[x^k]F^*_j(x)=\sum_{i=0}^kF^*_{j-i}(0)[x^{k-i}]\Phi_n(x)=\sum_{i=0}^kF^*_{j-k+i}(0)[x^i]\Phi_n(x).\label{x^kFj}\end{equation}

In the ternary case, with a little bit of work, one sees that this is equivalent to Kaplan's lemma, Lemma 1 in \cite{Kap07}. Although we prove Kaplan's results for ternary cyclotomic polynomials again more directly below, his work shows that they can also be proved in this fashion.

With a shift of index, we see that equation \eqref{x^kFj} implies
\begin{equation}[x^k]F^*_{j+k}(x)=\sum_{i=0}^kF^*_{j+i}(0)[x^i]\Phi_n(x).\label{F'coeffseq}\end{equation}
This is significant because it makes $\{[x^k]F^*_{j+k}\}_{k=0}^{\varphi(n)}$ the sequence of partial sums of the terms of the form $F^*_{j+i}(0)[x^i]\Phi_n(x)$, from $i=0$ to $i=k$. We can think of this summation as a dot product of sorts between the coefficient vector of $\Phi_n(x)$ and the $F^*_{j+i}(0)$. We will use equation \eqref{F'coeffseq} in the computation of certain $F_j(x)$ in the proof of Theorem \ref{r=2thm}. We will also use the $F^*_j(x)$ in the proof of Theorem \ref{|w|thm}.

\section{Ternary Cyclotomic Polynomials}\label{ternarysection}

We will next use this method to prove various results about $\Phi_{pqr}(x)$ for odd primes $p<q<r$, known as the \emph{ternary cyclotomic polynomials}. In this section, we begin by letting $n=pq$, so $\varphi(n)=(p-1)(q-1)$ and $n-\varphi(n)=p+q-1$. We have therefore shown that periodicity extends to any $r>p+q-1$.

\subsection{$r\equiv\pm1\pmod{pq}$}\label{r=1section}

This is the easiest case. In 2007, Kaplan \cite{Kap07} significantly generalized a theorem of Bachman \cite{Bach06} into
\begin{thm}\label{r=1thm}Let $p<q<r$ be primes such that $r\equiv\pm1\pmod{pq}$. Then $A(pqr)=1$.\end{thm}
We will prove this again, as our method additionally reveals the structure of such $\Phi_{pqr}(x)$. This structure will be useful later, for instance in the proof of Theorem \ref{pqrsflatthm}.
\begin{proof}Note that $r\ge pq-1>p+q-1$, so periodicity applies. Therefore, it suffices to consider $r\equiv1\pmod{pq}$. By Corollary \ref{Phi(x^w)/Phi(x)cor} with $j=0$, $F_{n,p,0}(x)=\tilde F_{p,q,1,0}(x)=1$. Also, from equations \eqref{Fextended} and \eqref{F=Fgeneral},
\[F_j(x)=x^{-j}F_{j(1-r)}(x)\equiv x^{-j}F_0(x)\equiv x^{-j}\pmod{\Phi_{pq}(x)}.\]
We must show that these polynomials are flat. In order to do this, we will need some properties of multiples of $\Phi_{pq}(x)$.

The following generalization of Proposition \ref{Ldiagramprop} helps us here and will be useful repeatedly later. First, a classic result in number theory states that if $p$ and $q$ are relatively prime, any integer greater than $pq-p-q$ can be expressed as a sum of nonnegative multiples of $p$ and $q$, and that for integers less than $pq$, this summation is unique. We can use this result to prove the following lemma.

\begin{lem}\label{pq(1+x)lem}Let $p<q$ be primes, and $1\le l\le p+q-1$ an integer. Let $pq-p-q+l=(\mu-1)p+(\lambda-1)q$ for $1\le\mu\le q$ and $1\le\lambda\le p$. Equivalently, $pq+l=\mu p+\lambda q$. Then
\begin{multline}
(1+x+\dotsb+x^{l-1})\Phi_{pq}(x)=(1+x^p+\dotsb+x^{p(\mu-1)})(1+x^q+\dotsb+x^{q(\lambda-1)})\\
-x^l(1+x^p+\dotsb+x^{p(q-\mu-1)})(1+x^q+\dotsb+x^{q(p-\lambda-1)}).\label{pq(1+x)}
\end{multline}\end{lem}
\begin{proof}Multiply the right side of equation \eqref{pq(1+x)} by $(1-x^p)(1-x^q)$. This yields
\begin{align*}
(1-x^p)(1-x^q)\eqref{pq(1+x)}&=(1-x^{p\mu})(1-x^{q\lambda})-x^l(1-x^{p(q-\mu)})(1-x^{q(p-\lambda)})\\
&=1-x^{p\mu}-x^{q\lambda}+x^{pq+l}-x^l+x^{l+pq-p\mu}+x^{l+pq-q\lambda}-x^{pq}\\
&=1-x^l-x^{pq}+x^{pq+l}=(1-x^l)(1-x^{pq})\\
&=(1+x+\dotsb+x^{l-1})(1-x^p)(1-x^q)\Phi_{pq}(x),
\end{align*}
as desired.\end{proof}
Note that Proposition \ref{Ldiagramprop} is the case $l=1$. Just like Proposition \ref{Ldiagramprop}, Lemma \ref{pq(1+x)lem} has the same L diagram interpretation: Draw lines to the left of $a=\mu$ and below $b=\lambda$, and $[x^k]((1+x+\dotsb+x^{l-1})\Phi_{pq}(x))=1$ if and only if $k$ is in the lower left quadrant while $[x^k](1+x+\dotsb+x^{l-1})\Phi_{pq}(x)=-1$ if and only if $k$ is in the upper right quadrant. Analogously, $l$ is the number just above and to the right of both lines, rather than 1, since $pq+l=p\mu+q\lambda$.

Note that the highest degree positive term of $(1+x+\dotsb+x^{l-1})\Phi_{pq}(x)$ is $x^{p(\mu-1)+q(\mu-1)}=x^{pq-p-q+l}$ and the highest degree negative term is $-x^{l+p(q-\mu-1)+q(p-\lambda-1)}=-x^{pq-p-q}$. The second highest degree positive term is $x^{pq-2p-q+l}$. Also, since $1+x+\dotsb+x^{l-1}$ and $\Phi_{pq}(x)$ are reciprocal, so is their product. This makes the smallest degree positive terms $1+x^p$ and negative term $-x^l$.

The claim that the unique minimal-degree polynomial $F_j(x)\equiv x^{-j}\pmod{\Phi_{pq}(x)}$ is flat is equivalent to the following statement.

\begin{lem}\label{x^klem}Let $p<q$ be primes. For all integers $k$, $x^k$ is congruent to a flat polynomial with degree less than $(p-1)(q-1)$ modulo $\Phi_{pq}(x)$.\end{lem}
\begin{proof}If $0\le k<(p-1)(q-1)$, we are done. Since $\Phi_{pq}(x)\mid(x^{pq}-1)$ it suffices to prove the statement for $(p-1)(q-1)\le k<pq$. We will provide a flat polynomial $f'(x)$ such that $[x^k]f'(x)=1$, $[x^m]f'(x)=0$ for $k\neq m\ge(p-1)(q-1)$, and $\Phi_{pq}(x)\mid f'(x)$. Then $x^k\equiv x^k-f'(x)\pmod{\Phi_{pq}(x)}$, and $x^k-f'(x)$ will have the appropriate degree.

When $(p-1)(q-1)\le k\le pq-q-1$, let $l=k-(pq-p-q)$, so $1\le l\le p-1$. Then let $f'(x)=(1+x+\dotsb+x^{l-1})\Phi_{pq}(x)$. Indeed, we calculated above the highest degree terms of this to be $x^{pq-p-q+l}=x^k$, $-x^{pq-p-q}$, and $x^{k-p}$. As $k-p<pq-p-q$, $[x^m]\Phi_{pq}(x)=0$ for $k\neq m\ge(p-1)(q-1)$. By Lemma \ref{pq(1+x)lem}, this is flat.

When $(p-1)q\le k\le p(q-1)$, let $f'(x)=x^{k-(p-1)q}\Phi_p(x^q)$. Indeed, $\Phi_{pq}(x)\mid\Phi_p(x^q)$, and the highest degree terms of $f'(x)$ are $x^k+x^{k-q}$. Since $k\le p(q-1)$, $k-q\le pq-p-q$ and $[x^m]\Phi_{pq}(x)=0$ for $k\neq m\ge(p-1)(q-1)$.

Finally, when $pq-p+1\le k\le pq-1$, let $l=pq-k$, so $1\le l\le p-1$. Then let $f'(x)=x^{-l}(1-(1+x+\dotsb+x^{l-1})\Phi_{pq}(x))$. Indeed, as $x^{pq}\equiv1\pmod{\Phi_{pq}(x)}$, $f'(x)\equiv x^k\pmod{\Phi_{pq}(x)}$. Since the degree of $f'(x)$ is $-l+(p-1)(q-1)+l-1=pq-p-q$, we simply must show that this is actually a polynomial. Indeed, the lowest degree terms of $f'(x)$ are $x^{-l}(1-(1-x^l+x^p))=1-x^{p-l}$. Therefore, this is a polynomial, as desired.\end{proof}

Therefore, we have proven Theorem \ref{r=1thm}.\end{proof}

The generalizations of Theorem \ref{r=1thm} and Lemma \ref{pq(1+x)lem} to pseudocyclotomic polynomials are easy to see and the proofs are identical:
\begin{prop}\label{pseudor=1prop}Let $p,q,r$ be pairwise relatively prime positive integers such that $r\equiv\pm1\pmod{pq}$. Then $\tilde\Phi_{p,q,r}(x)$ is flat.\end{prop}

\begin{prop}\label{pseudopq(1+x)prop}Let $p<q$ be relatively prime positive integers, and $1\le l\le p+q-1$ an integer. Let $pq-p-q+l=(\mu-1)p+(\lambda-1)q$ for $1\le\mu\le q$ and $1\le\lambda\le p$. Equivalently, $pq+l=\mu p+\lambda q$. Then
\begin{multline}
(1+x+\dotsb+x^{l-1})\tilde\Phi_{p,q}(x)=(1+x^p+\dotsb+x^{p(\mu-1)})(1+x^q+\dotsb+x^{q(\lambda-1)})\\
-x^l(1+x^p+\dotsb+x^{p(q-\mu-1)})(1+x^q+\dotsb+x^{q(p-\lambda-1)}).\label{pseudopq(1+x)}
\end{multline}\end{prop}

\subsection{$r\equiv\pm2\pmod{pq}$}\label{r=2section}

We might expect $r\equiv\pm2\pmod{pq}$ to be the next easiest case to handle. Indeed it is, and the goal of this section is to use our methods to prove the following theorem, conjectured by Broadhurst \cite{BH} and proven previously in unpublished work by Liu \cite{Liu}. Although we will later generalize one direction of this result in Theorem \ref{BroadhurstII}, we include this proof because it reveals the structure of these cyclotomic polynomials.
\begin{thm}\label{r=2thm}Let $p<q<r$ be odd primes such that $r\equiv\pm2\pmod{pq}$. Then $A(pqr)=1$ if and only if $q\equiv1\pmod p$.\end{thm}
\begin{proof}Since $r\ge pq-2>pq-(p-1)(q-1)$, by Theorem \ref{extendedperiodicitythm}, it suffices to consider $r\equiv2\pmod{pq}$. By Lemma \ref{periodicitylem}, $F_{pq,r,j}(x)=F_{pq,2,j}(x)$ for $j=0$ and $j=1$. We know that $F_{pq,2,0}(x^2)+xF_{pq,2,1}(x^2)=\Phi_{2pq}(x)$. By Proposition \ref{2nprop}, $\Phi_{2pq}(x)=\Phi_{pq}(-x)$. Therefore, we have $F_0(x^2)+xF_1(x^2)=\Phi_{pq}(-x)$. Substituting $-x$ for $x$, $F_0(x^2)-xF_1(x^2)=\Phi_{pq}(x)$, so $F_0(x^2)=\frac12(\Phi_{pq}(x)+\Phi_{pq}(-x))=\sum_{k\ge0}x^{2k}[x^{2k}]\Phi_{pq}(x)$, or
\begin{equation}F_{pq,2,0}(x)=\sum_{k\ge0}x^k[x^{2k}]\Phi_{pq}(x).\label{pq20}\end{equation}
Meanwhile, $xF_1(x^2)=\frac12(\Phi_{pq}(-x)-\Phi_{pq}(x))=\sum_{k\ge0}x^{2k+1}[x^{2k+1}]\Phi_{pq}(-x)=-\sum_{k\ge0}x^{2k+1}[x^{2k+1}]\Phi_{pq}(x)$, or
\begin{equation}F_{pq,2,1}(x)=-\sum_{k\ge0}x^k[x^{2k+1}]\Phi_{pq}(x).\label{pq21}\end{equation}

Thus, $\deg F_0(x)=\frac{(p-1)(q-1)}2$ and $\deg F_1(x)=\frac{(p-1)(q-1)}2-1$. From $F_0(x)$ and $F_1(x)$, we can calculate the rest of the $F_{pq,r,j}(x)$ from $F_{r-2}(x)=F_{-2}(x)=xF_0(x)$, $F_{r-1}(x)=xF_1(x)$, and $F_{j-2}(x)\equiv xF_j(x)\pmod{\Phi_{pq}(x)}$, with $\deg F_j(x)<(p-1)(q-1)$, for $4\le j<r$.

How does the congruence enter in? Recall the L diagram representation of Proposition \ref{Ldiagramprop}: Let $0\le\mu\le q-1$ be the inverse of $p$ modulo $q$, and $0\le\lambda\le p-1$ the inverse of $q$ modulo $p$. Then the lemma says that
\[\Phi_{pq}(x)=(1+x^p+\dotsb+x^{p(\mu-1)})(1+x^q+\dotsb+x^{q(\lambda-1)})-x(1+x^p+\dotsb+x^{p(q-\mu-1)})(1+x^q+\dotsb+x^{q(p-\lambda-1)}).\]
Since $p<q$, $\mu>1$. We know that $q\equiv1\pmod p$ if and only if $\lambda=1$, which forms the key to the proof. The proof we have for each direction is different, so we split them up.
\subsubsection{$A(pqr)=1\Rightarrow q\equiv1\pmod p$}
Suppose that $q\not\equiv1\pmod p$. Then $\lambda>1$, and both $1$ and $x^{p+q}$ are terms of $\Phi_{pq}(x)$ with even degree. So $1$ and $x^{\frac{p+q}2}$ are terms of $F_0(x)$. We claim that $F_{r-(p-1)(q-1)}(x)$ is not flat.

We have established that $F_{r-(p-1)(q-1)}(x)\equiv x^{\frac{(p-1)(q-1)}2}F_0(x)\pmod{\Phi_{pq}(x)}$. Since $F_0(x)$ has leading term $x^{\frac{(p-1)(q-1)}2}$ and $\deg F_{r-(p-1)(q-1)}(x)<(p-1)(q-1)$, we must have
\[F_{r-(p-1)(q-1)}(x)=x^{\frac{(p-1)(q-1)}2}F_0(x)-\Phi_{pq}(x).\]
We will show that either $[x^{\frac{(p-1)(q-1)}2}]F_{r-(p-1)(q-1)}(x)=2$ or $[x^{\frac{(p-1)(q-1)}2+\frac{p+q}2}]F_{r-(p-1)(q-1)}(x)=2$, which shows that $F_{r-(p-1)(q-1)}(x)$, and therefore $\Phi_{pqr}(x)$, are not flat.

Note that $[x^{\frac{(p-1)(q-1)}2}](x^{\frac{(p-1)(q-1)}2}F_0(x))=[1]F_0(x)=1$ and $[x^{\frac{pq+1}2}](x^{\frac{(p-1)(q-1)}2}F_0(x))=[x^{\frac{p+q}2}]F_0(x)=1$. We then must show that either $[x^{\frac{(p-1)(q-1)}2}]\Phi_{pq}(x)=-1$ or $[x^{\frac{pq+1}2}]\Phi_{pq}(x)=-1$, and which one depends on the parity of $\mu$ and $\lambda$. First notice that $p\mu+q\lambda=pq+1$ is even while $p$ and $q$ are odd, so $\mu$ and $\lambda$ always have the same parity. (Since $r\equiv2\pmod{pq}$ and $(r,p)=(r,q)=1$, $p$ and $q$ must both be odd, even when we are considering the pseudocyclotomic case and they are not assumed to be prime.)

First suppose that $\mu$ and $\lambda$ are even. Then $-x\cdot x^{p\frac{q-\mu-1}2}\cdot x^{q\frac{p-\lambda-1}2}=-x^{1+\frac{2pq-p\mu-q\lambda-p-q}2}=-x^{1+\frac{pq-p-q-1}2}=-x^{\frac{(p-1)(q-1)}2}$. Moreover, $0\le\frac{q-\mu-1}2\le q-\mu-1$ and $0\le\frac{p-\lambda-1}2\le p-\lambda-1$, so indeed $[x^{\frac{(p-1)(q-1)}2}]\Phi_{pq}(x)=-1$ by Proposition \ref{Ldiagramprop}, as desired.

Otherwise, $\mu$ and $\lambda$ are odd. Then $-x\cdot x^{p\frac{q-\mu}2}\cdot x^{q\frac{p-\lambda}2}=-x^{1+\frac{2pq-p\mu-q\lambda}2}=-x^{\frac{pq+1}2}$. Moreover, $q>\mu$ so $q-\mu\ge2$ and $0\le\frac{q-\mu}2\le q-\mu-1$ and similarly, $0\le\frac{p-\lambda}2\le p-\lambda-1$. Therefore, $[x^{\frac{pq+1}2}]\Phi_{pq}(x)=-1$, as desired.

To relate these results back to the original $\Phi_{pqr}(x)$, we have shown that when $q\not\equiv1\pmod p$, either  $[x^{r-(p-1)(q-1)+r\frac{(p-1)(q-1)}2}]\Phi_{pqr}(x)=2$ or $[x^{r-(p-1)(q-1)+r\frac{pq+1}2}]\Phi_{pqr}(x)=2$.

\subsubsection{$q\equiv1\pmod p\Rightarrow A(pqr)=1$}
As outlined above, we will compute $F_0(x)$ and $F_1(x)$ from $\Phi_{pq}(x)$, then use these to compute the other $F_j(x)$, and demonstrate that these are all flat.

Let $q=2kp+1$, as both are odd. Then $\mu=q-2k$ and $\lambda=1$. We have
\[\Phi_{pq}(x)=(1+x^p+\dotsb+x^{p(q-2k-1)})-x(1+x^p+\dotsb+x^{p(2k-1)})(1+x^q+\dotsb+x^{q(p-2)}).\]
We can also write the first sum as $1+x^p+\dotsb+x^{2k(p-1)p}$. The even degree terms of this are
\begin{align}
F_0(x^2)&=(1+x^{2p}+\dotsb+x^{2k(p-1)p})-(x^{p+1}+x^{q+1})(1+x^{2p}+\dotsb+x^{2p(k-1)})(1+x^{2q}+\dotsb+x^{q(p-3)})\\
F_0(x)&=(1+x^p+\dotsb+x^{k(p-1)p})-(x^{\frac{p+1}2}+x^{\frac{q+1}2})(1+x^p+\dotsb+x^{p(k-1)})(1+x^q+\dotsb+x^{q\frac{p-3}2}).\label{r=2F_0}
\end{align}
$F_0(x)$ is flat by inspection or because $\Phi_{pq}(x)$ is flat. It has degree $k(p-1)p=\frac{(p-1)(q-1)}2$. Therefore, we know that $F_{pq-2}(x)=xF_0(x)$, $F_{pq-4}(x)=x^2F_0(x)$, $\dotsc$, $F_{p+q+1}(x)=x^{\frac{pq-p-q-1}2}F_0(x)$ are also flat.

Similarly, $F_1(x)$ is flat with degree $\frac{(p-1)(q-1)}2-1$, and therefore, $F_{pq-1}(x),F_{pq-3}(x),\dotsc,F_{p+q}(x)$ are also flat. It remains to show that $F_j(x)$ is flat for $2\le j\le p+q-1$.

Now we invoke reciprocity. We have $-(p-1)(q-1)=-r+(r-(p-1)(q-1))$, so in the notation of Section \ref{reciprocity}, $y=-1$ and $z=r-(p-1)(q-1)\ge(pq+2)-(p-1)(q-1)=p+q+1$. Therefore, \eqref{reciprocitygeneral} says that for $0\le j\le p+q+1$, $V(F_j(x))=V(F_{z-j}(x))$. Moreover, $z\equiv p+q+1\pmod{pq}$, so by \eqref{F=F}, $V(F_j(x))=V(F_{p+q+1-j}(x))$ for all $0\le j\le p+q+1$. Because $p+q+1$ is odd, it suffices to show that $F_j(x)$ is flat for all odd $3\le j\le p+q-1$; this result from reciprocity then covers the even $j$. From Section \ref{Fj(0)section}, we use knowledge of which $F_j(0)$ are nonzero to calculate that
\begin{align*}
F_{p+q+1-2l}(x)&=x^{\frac{(p-1)(q-1)}2+l-1}F_0(x)-(1+x+\dotsb+x^{l-1})\Phi_{pq}(x),\quad 1\le l\le\frac{p+1}2\\
F_q(x)&=x^{\frac{(p-1)q}2}F_0(x)-(1+x+\dotsb+x^{\frac{p-1}2})\Phi_{pq}(x)\\
F_{q-2l}(x)&=x^lF_q(x),\quad1\le l\le\frac{q-p}2\\
F_p(x)&=x^{\frac{q-p}2}F_q(x)\\
F_{p-2l}(x)&=x^{\frac{q-p}2+l}F_q(x)+(1+x+\dotsb+x^{l-1})\Phi_{pq}(x),\quad 1\le l\le\frac{p-1}2\\
F_1(x)&=x^{\frac{q-1}2}F_q(x)+(1+x+\dotsb+x^{\frac{p-3}2})\Phi_{pq}(x).
\end{align*}
Therefore, we are interested in Lemma \ref{pq(1+x)lem} applied when $q=2kp+1$. For this special case, we have the following:
\begin{cor}\label{pq(1+x)q=1cor}Let $p$ and $q=2kp+1$ be primes, and $1\le l\le p$ an integer. Then
\begin{multline}
(1+x+\dotsb+x^{l-1})\Phi_{pq}(x)=(1+x^p+\dotsb+x^{p(q-2kl-1)})(1+x^q+\dotsb+x^{q(l-1)})\\
-x^l(1+x^p+\dotsb+x^{p(2kl-1)})(1+x^q+\dotsb+x^{q(p-l-1)}).\label{pq(1+x)q=1}
\end{multline}\end{cor}

With this result, we first show that $F_{p+q+1-2l}(x)$ is flat for $1\le l\le\frac{p+1}2$. It is informative with these sorts of arguments to consider the L diagram and trace why the exponents of the terms of each polynomial fail to intersect. Let $u(x)=x^{\frac{(p-1)(q-1)}2+l-1}F_0(x)=x^{kp(p-1)+l-1}F_0(x)$ and $v(x)=(1+x+\dotsb+x^{l-1})\Phi_{pq}(x)$. We have $F_{p+q+1-2l}(x)=u(x)-v(x)$. Both $u(x)$ and $v(x)$ are flat by their explicit constructions in equations \eqref{r=2F_0} and \eqref{pq(1+x)q=1}. Let $U_+=\{i|[x^i]u(x)=1\}$ and $U_-=\{i|[x^i]u(x)=-1\}$ and define $V_+,V_-$ analogously. Then we must show that $U_+\cap V_-=U_-\cap V_+=\emptyset$.

By equations \eqref{r=2F_0} and \eqref{pq(1+x)q=1},
\begin{align*}
U_+&=l-1+p\{k(p-1),k(p-1)+1,\dotsc,2k(p-1)\}\\
U_-&=U_{-,1}\cup U_{-,2},\text{ where}\\
U_{-,1}&=l-1+\frac{p+1}2+(p\{kp-k,kp-k+1,\dotsc,kp-1\})\oplus\left(q\left\{0,1,\dotsc,\frac{p-3}2\right\}\right)\\
U_{-,2}&=l-1+\frac{q+1}2+(p\{kp-k,kp-k+1,\dotsc,kp-1\})\oplus\left(q\left\{0,1,\dotsc,\frac{p-3}2\right\}\right)\\
V_+&=(p\{0,1,\dotsc,q-2kl-1\})\oplus(q\{0,1,\dotsc,l-1\})\\
V_-&=l+(p\{0,1,\dotsc,2kl-1\})\oplus(q\{0,1,\dotsc,p-l-1\}),
\end{align*}
where $A\oplus B=\{a+b:a\in A,b\in B\}$. Each direct sum of sequences corresponds to a rectangular region in the L diagram. To show that $U_+\cap V_-=\emptyset$, we consider each modulo $p$. We have $U_+\equiv l-1\pmod p$, while $V_-\equiv l+\{0,1,\dotsc,p-l-1\}\equiv\{l,l+1,\dotsc,p-1\}\pmod p$. None of these are congruent to $l-1$, as desired. Through this, we have shown that these rectangular regions occupy different rows, and therefore, don't intersect.

Showing that $U_{-,1}\cap V_+=\emptyset$ is more difficult. Consider each modulo $q$, that is, we will show they occupy different columns. We have $V_+\equiv p\{0,1,\dotsc,q-2kl-1\}\pmod q$ and
\begin{align*}
U_{-,1}&\equiv l-1+\frac{p+1}2+p\{kp-k,\dotsc,kp-1\}\equiv l-1+\frac{p+1}2+p(kp-\{1,\dotsc,k\})\\
&\equiv l+p\frac{2kp+1}2-\frac12-p\{1,\dotsc,k\}\equiv l+\frac{pq-1}2-p\{1,\dotsc,k\}\\
&\equiv l+\frac{q-1}2-p\{1,\dotsc,k\}\equiv l+kp-p\{1,\dotsc,k\}\equiv l+p\{0,\dotsc,k-1\}\pmod q.
\end{align*}
Now suppose that $V_+$ and $U_{-,1}$ intersect modulo $q$. That is, let $pa\equiv l+pb\pmod q$ for some integers $0\le a\le q-2kl-1$ and $0\le b\le k-1$. Then $2kp\equiv-1\pmod q$, so we have $a-b\equiv-2kl\equiv q-2kl\pmod q$. However, $-2kl<1-k\le a-b\le q-2kl-1<q-2kl$, so there is no pair $(a,b)$, as desired. Again, we have shown that $V_+$ and $U_{-,1}$ occupy different columns in the L diagram.

Finally, to show that $U_{-,2}\cap V_+=\emptyset$, again consider each modulo $p$, i.e. we show they occupy different rows. We have $U_{-,2}\equiv l-1+1+\{0,1,\dotsc,\frac{p-3}2\}\equiv l+\{0,1,\dotsc,\frac{p-3}2\}\pmod p$ and $V_+\equiv\{0,1,\dotsc,l-1\}\pmod p$. Since $l\le\frac{p+1}2$, $l+\frac{p-3}2<p$ and these sets are disjoint, making $U_-\cap V_+=\emptyset$ as desired.

We have proved that $F_{p+q+1-2l}(x)$ is flat for all $1\le l\le \frac{p+1}2$. Now we consider the particular case $l=\frac{p+1}2$ to get a simple formula for $F_q(x)$:
\begin{align*}
F_q(x)&=x^{\frac{p-1}2q}F_0(x)-\left(1+x+\dotsb+x^{\frac{p-1}2}\right)\Phi_{pq}(x)\\
&=x^{\frac{p-1}2q}\Bigl[\left(1+x^p+\dotsb+x^{k(p-1)p}\right)\\
&\qquad-\left(x^{\frac{p+1}2}+x^{\frac{q+1}2}\right)\left(1+x^p+\dotsb+x^{(k-1)p}\right)\left(1+x^q+\dotsb+x^{q\frac{p-3}2}\right)\Bigr]\\
&\qquad-\left(1+x^p+\dotsb+x^{k(p-1)p}\right)\left(1+x^q+\dotsb+x^{q\frac{p-1}2}\right)\\
&\qquad+x^{\frac{p+1}2}\left(1+x^p+\dotsb+x^{p(k(p+1)-1)}\right)\left(1+x^q+\dotsb+x^{q\frac{p-3}2}\right)\\
&=\left(1+x^q+\dotsb+x^{q\frac{p-3}2}\right)\Bigl[-x^{\frac{p-1}2q}\left(x^{\frac{p+1}2}+x^{\frac{q+1}2}\right)\left(1+x^p+\dotsb+x^{(k-1)p}\right)\\
&\qquad-(1+x^p+\dotsb+x^{k(p-1)p})+x^{\frac{p+1}2}(1+x^p+\dotsb+x^{p(k(p+1)-1)})\Bigr]\\
&=\left(1+x^q+\dotsb+x^{q\frac{p-3}2}\right)\Bigl[-\left(x^{p(kp-k+1)}+x^{\frac{p+1}2+kp^2}\right)\left(1+x^p+\dotsb+x^{(k-1)p}\right)\\
&\qquad-(1+x^p+\dotsb+x^{p(kp-k)})+x^{\frac{p+1}2}(1+x^p+\dotsb+x^{p(k(p+1)-1)})\Bigr]\\
&=\left(1+x^q+\dotsb+x^{q\frac{p-3}2}\right)\left[x^{\frac{p+1}2}\left(1+x^p+\dotsb+x^{p(kp-1)}\right)-\left(1+x^p+\dotsb+x^{kp^2}\right)\right].
\end{align*}
This is flat, so $F_{q-2l}(x)=x^lF_q(x)$ is flat as well, for $1\le l\le\frac{q-p}2$. Now, for $0<l\le\frac{p-1}2$, consider $F_{p-2l}(x)=x^{\frac{q-p}2+l}F_q(x)+(1+x+\dotsb+x^{l-1})\Phi_{pq}(x)$. Again, let $u(x)=x^{\frac{q-p}2+l}F_q(x)$ and $v(x)=(1+x+\dotsb+x^{l-1})\Phi_{pq}(x)$, so $F_{p-2l}(x)=u(x)+v(x)$. Define $U_+,U_-,V_+,V_-$ as before. This time we must show that $U_+\cap V_+=U_-\cap V_-=\emptyset$. We have
\begin{align*}
U_+&=kp+1+l+\left(q\left\{0,1,\dotsc,\frac{p-3}2\right\}\right)\oplus(p\{0,1,\dotsc,kp-1\})\\
U_-&=\frac{q-p}2+l+\left(q\left\{0,1,\dotsc,\frac{p-3}2\right\}\right)\oplus(p\{0,1,\dotsc,kp\})\\
V_+&=(p\{0,1,\dotsc,q-2kl-1\})\oplus(q\{0,1,\dotsc,l-1\})\\
V_-&=l+(p\{0,1,\dotsc,2kl-1\})\oplus(q\{0,1,\dotsc,p-l-1\}).
\end{align*}

To show $U_+\cap V_+=\emptyset$, consider each modulo $p$, again showing that they occupy different rows in the L diagram. We have $U_+\equiv l+1+\{0,1,\dotsc,\frac{p-3}2\}\pmod p$, while $V_+\equiv\{0,1,\dotsc,l-1\}\pmod p$. Since $l\le\frac{p-1}2$, $l+1+\frac{p-3}2<p$ and these sets are disjoint, as desired.

To show $U_-\cap V_-=\emptyset$, consider each modulo $q$, again showing that they occupy different columns in the L diagram. We have $U_-\equiv\frac{q-p}2+l+p\{0,1,\dotsc,kp\}\pmod q$, while $V_-\equiv l+p\{0,1,\dotsc,2kl-1\}\pmod q$. Cancelling the $l$, if there is nonempty intersection between these sets, we must have $\frac{q-p}2\equiv pa\pmod q$ for some $-kp\le a\le 2kl-1$. We can solve for $a$: $a\equiv\frac{q-p}{2p}\equiv(-k)(q-p)\equiv kp\equiv-kp-1$. But $2kl-1<k(p-1)<kp$ so neither of these is in the required range, and no such $a$ exists, as desired.

Therefore, we have shown that $F_j(x)$ is flat for all $j$, and thus that $A(pqr)=1$, as desired.\end{proof}

Again, the pseudocyclotomic analogs of Theorem \ref{r=2thm} and Corollary \ref{pq(1+x)q=1cor} are easy to see and the proofs are identical. However, when we generalize Corollary \ref{pq(1+x)q=1cor}, we cannot assume both $p$ and $q$ are odd in general, so we write $q=kp+1$ rather than $q=2kp+1$.

\begin{prop}\label{pseudor=2prop}Let $p$, $q$ and $r$ be pairwise relatively prime positive integers such that $r\equiv\pm2\pmod{pq}$. Then $\tilde\Phi_{p,q,r}(x)$ is flat if and only if $q\equiv1\pmod p$.\end{prop}

\begin{prop}\label{pseudopq(1+x)q=1prop}Let $p$, $k$, and $q=kp+1$ be positive integers, and $1\le l\le p$ an integer. Then
\begin{multline}
(1+x+\dotsb+x^{l-1})\tilde\Phi_{p,q}(x)=(1+x^p+\dotsb+x^{p(q-kl-1)})(1+x^q+\dotsb+x^{q(l-1)})\\
-x^l(1+x^p+\dotsb+x^{p(kl-1)})(1+x^q+\dotsb+x^{q(p-l-1)}).\label{pseudopq(1+x)q=1}
\end{multline}\end{prop}

\subsection{Broadhurst's Type II}\label{BroadhurstIIsection}

David Broadhurst has also conjectured \cite{BH2} that all flat ternary cyclotomic polynomials $\Phi_{pqr}(x)$ fall into one of three categories. The first category, already proven to be flat, is $r\equiv\pm1\pmod{pq}$. The second category is the following generalization of one direction of Theorem \ref{r=2thm}. We will prove that polynomials in this category are indeed flat using pseudocyclotomic polynomials. The third category is rephrased as a conjecture in \cite{Kap10} and Section \ref{openquestions3}.

\begin{thm}\label{BroadhurstII}Let $p<q<r$ be primes and $w$ a positive integer such that $r\equiv w\pmod{pq}$, $p\equiv1\pmod w$, and $q\equiv1\pmod{pw}$. Then $A(pqr)=1$.\end{thm}
\begin{proof}To prove this, we will explicitly compute the $F_{pq,r,j}(x)$, and show that each of these is flat. In order to compute these, we make use of the pseudocyclotomic polynomial $\tilde\Phi_{p,q,w}(x)$ to calculate $\tilde F_{p,q,w,j}(x)$, which is the same as $F_{pq,r,j}(x)$ for $0\le j<w$ by Corollary \ref{Phi(x^w)/Phi(x)cor}. We can compute $F_{p,q,w,j}(x)$ by extracting the coefficients from $\tilde\Phi_{p,q,w}(x)$. Finally, $\tilde\Phi_{p,q,w}(x)=\tilde\Phi_{w,p,q}(x)$, and we already know how to compute $\tilde\Phi_{w,p,q}(x)$ since $q\equiv1\pmod{wp}$. We now run this entire setup in reverse.

\subsubsection{Utilizing A Pseudocyclotomic Polynomial}
\vspace{.25cm}\noindent{\large\textsc{The Starting Point: $\tilde F_{w,p,q,j}(x)$}}

By Proposition \ref{p=1pseudoprop}, $\tilde F_{w,p,q,j}(x)\equiv x^{-j}\pmod{\tilde\Phi_{w,p}(x)}$. Focusing on $0\le j<wp$, we calculate that
\begin{equation}
\tilde F_{w,p,q,j}(x)=\begin{cases}
1&\qquad j=0\\
x^{-j}(1-(1+x+\dotsb+x^{j-1})\tilde\Phi_{w,p}(x))&\qquad 1\le j<w\\
x^{-j}(1-\tilde\Phi_w(x^p))&\qquad w\le j\le p\\
x^{wp-j}-(1+x+\dotsb+x^{w+p-1-j})\tilde\Phi_{w,p}(x)&\qquad p<j\le w+p-1\\
x^{wp-j}&\qquad w+p\le j<wp.
\end{cases}\label{F_w,p,q,j}\end{equation}

For $wp\le j<q$, $\tilde F_{w,p,q,j}(x)=\tilde F_{w,p,q,j-wp}(x)$ by the pseudocyclotomic analog to Corollary \ref{F=Fcor}. Indeed, each of these is a polynomial of degree less than $(w-1)(p-1)$, and clearly congruent to $x^{-j}$ mod $\tilde\Phi_{w,p}(x)$.

To compute more explicitly those multiples of $\tilde\Phi_{w,p}(x)$, we need Proposition \ref{pseudopq(1+x)q=1prop}. Applied to this particular case, $p=aw+1$, so we get for $1\le l\le w$,
\begin{multline}(1+x+\dotsb+x^{l-1})\tilde\Phi_{w,p}(x)=(1+x^w+\dotsb+x^{w(p-al-1)})(1+x^p+\dotsb+x^{p(l-1)})\\
-x^l(1+x^w+\dotsb+x^{w(al-1)})(1+x^p+\dotsb+x^{p(w-l-1)}).\label{wpl}\end{multline}

\vspace{.25cm}\noindent{\large\textsc{Calculating $\tilde\Phi_{w,p,q}(x)$}}

We now reconstruct $\tilde\Phi_{w,p,q}(x)=\sum_{j=0}^{q-1}x^j\tilde F_{w,p,q,j}(x^q)=\sum_{j=0}^wx^j\tilde F_{p,q,w,j}(x^w)$. We begin by adjusting equation \eqref{F_w,p,q,j} to yield these particular summands:
\begin{equation}
x^j\tilde F_{w,p,q,j}(x^q)=\begin{cases}
1&\qquad j=0\\
x^{-j(q-1)}(1-(1+x^q+\dotsb+x^{q(j-1)})\tilde\Phi_{w,p}(x^q))&\qquad 1\le j<w\\
x^{-j(q-1)}(1-\tilde\Phi_w(x^{pq}))&\qquad w\le j\le p\\
x^{wpq-j(q-1)}-x^j(1+x^q+\dotsb+x^{q(w+p-1-j)})\tilde\Phi_{w,p}(x^q)&\qquad p<j\le w+p-1\\
x^{wpq-j(q-1)}&\qquad w+p\le j<wp.
\end{cases}\label{F_w,p,q,j'}\end{equation}

We know that $\tilde\Phi_{w,p,q}(x)$ is the sum of this expression over all $0\le j<q$, but equation \eqref{F_w,p,q,j'} only gives us $x^jF_j(x^q)$ for $0\le j<wp$. To collect all the terms, we first sum $x^j\tilde F_{w,p,q,j}(x^q)$ over all values of $j$ between $0$ and $q-1$ in the residue class of $k$ modulo $wp$, for $0\le k<wp$. For $k>0$, there are $b$ values of $j$ in the range $[0,q-1]$, and for $k=0$, there are $b+1$ values. Since $\tilde F_j(x^q)=\tilde F_k(x^q)$ within each residue class, we simply need to sum up $(1+x^{wp}+\dotsb+x^{(b-1)wp})x^k\tilde F_k(x^q)$ for $k>0$, along with $(1+x^{wp}+\dotsb+x^{bwp})x^0\tilde F_0(x^q)=1+x^{wp}+\dotsb+x^{bwp}$ for $k=0$. Substituting the expressions for $F_k(x^q)$, we get
\begin{align}
\tilde\Phi_{w,p,q}(x)&=(1+x^{wp}+\dotsb+x^{bwp})\label{pqw1}\\
&\qquad+(1+x^{wp}+\dotsb+x^{(b-1)wp})\Biggl[\sum_{k=1}^{w-1}(x^{-k(q-1)}(1-(1+x^q+\dotsb+x^{q(k-1)})\tilde\Phi_{w,p}(x^q)))\label{pqw2}\\
&\qquad+\sum_{k=w}^p(x^{-k(q-1)}(1-\tilde\Phi_w(x^{pq})))\label{pqw3}\\
&\qquad+\sum_{k=p+1}^{w+p-1}(x^{wpq-k(q-1)}-x^k(1+x^q+\dotsb+x^{q(w+p-1-k)})\tilde\Phi_{w,p}(x^q))\label{pqw4}\\
&\qquad+\sum_{k=w+p}^{wp-1}x^{wpq-k(q-1)}\Biggr].\label{pqw5}
\end{align}
We will use expressions \eqref{pqw1} through \eqref{pqw5} to denote these five components, with the $1+x^{wp}+\dotsb+x^{(b-1)wp}$ suitably distributed through \eqref{pqw2} through \eqref{pqw5}, of course.

First, we use equation \eqref{wpl} to rewrite expressions \eqref{pqw2} and \eqref{pqw4}. For expression \eqref{pqw2}, $l=k$, and for expression \eqref{pqw4}, $l=w+p-k$. Note that we have replaced $p-al-1$ with $a(w-l)$.
\begin{multline}\eqref{pqw2}=(1+x^{wp}+\dotsb+x^{(b-1)wp})\sum_{k=1}^{w-1}[x^{-k(q-1)}(1-(1+x^{wq}+\dotsb+x^{wqa(w-k)})(1+x^{pq}+\dotsb+x^{pq(k-1)})\\
+x^{kq}(1+x^{wq}+\dotsb+x^{wq(ak-1)})(1+x^{pq}+\dotsb+x^{pq(w-k-1)}))].\label{pqw2'}\end{multline}
\begin{multline}\eqref{pqw4}=(1+x^{wp}+\dotsb+x^{(b-1)wp})\sum_{k=p+1}^{w+p-1}[x^{wpq-k(q-1)}-x^k((1+x^{wq}+\dotsb+x^{wqa(k-p)})(1+x^{pq}+\dotsb+x^{pq(w+p-k-1)})\\
-x^{(w+p-k)q}(1+x^{wq}+\dotsb+x^{wq(a(w+p-k)-1)})(1+x^{pq}+\dotsb+x^{pq(k-p-1)}))].\label{pqw4'}\end{multline}

\subsubsection{Calculating $\tilde F_{p,q,w,j}(x)$}
We have computed $\tilde\Phi_{p,q,w}(x)=\tilde\Phi_{w,p,q}(x)$, which we now split up a different way with $\tilde F_{p,q,w,j}(x)=\sum_{i\ge0}x^i[x^{iw+j}]\tilde\Phi_{p,q,w}(x)$ for $0\le j<w$.

\vspace{.25cm}\noindent{\large\textsc{The $j=0$ Case}}

$\tilde F_{p,q,w,0}(x)$ is a bit different from $\tilde F_{p,q,w,j}(x)$ for $j>0$. Distributing across expressions \eqref{pqw1}, \eqref{pqw3}, and \eqref{pqw5}, we immediately see that
\begin{align*}
\sum_{i\ge0}x^i[x^{iw}]\eqref{pqw1}&=1+x^p+\dotsb+x^{bp};\\
\sum_{i\ge0}x^i[x^{iw}]\eqref{pqw3}&=\sum_{k=w}^px^{-kbp}\sum_{i\ge0}x^i[x^{iw}](-x^{pq}-x^{2pq}-\dotsb-x^{(w-1)pq})=0;\\
\sum_{i\ge0}x^i[x^{iw}]\eqref{pqw5}&=(1+x^p+\dotsb+x^{(b-1)p})\sum_{k=w+p}^{wp-1}x^{pq-kbp}\\
&=(1+x^p+\dotsb+x^{(b-1)p})(x^{pq-(q-1)p+bp}+x^{pq-(q-1)p+bp+p}+\dotsc+x^{p(q-bw-bp)})\\
&=x^{p(b+1)}+x^{p(b+2)}+\dotsb+x^{p(q-bw-bp+b-1)}.\end{align*}
Adding them all together, we get
\[\sum_{i\ge0}x^i[x^{iw}](\eqref{pqw1}+\eqref{pqw3}+\eqref{pqw5})=1+x^p+\dotsb+x^{p(q-bw-bp+b-1)}=1+x^p+\dotsb+x^{pabw(w-1)}.\]

Expressions \eqref{pqw2} and \eqref{pqw4}, written as expressions \eqref{pqw2'} and \eqref{pqw4'}, are a bit trickier to decompose. For both, we remove the factor of $(1+x^p+\dotsb+x^{(b-1)p})$ and consider those exponents divisible by $w$ for each $k$ in the sum.

For expression \eqref{pqw2}, we notice that since $w\mid(q-1)$, the terms with exponents divisible by $w$ include the cancelled 1. Since $pq\equiv1\pmod w$, the residues of exponents of the first product run from $0$ to $k-1$ and those in the second product run from $k$ to $w-1$. Therefore, the terms with exponents divisible by $w$ must occur by choosing 1 out of $1+x^{pq}+\dotsb+x^{pq(k-1)}$ in expression \eqref{pqw2'}, making that product just $1+x^{wq}+\dotsb+x^{wqa(w-k)}$. One can visualize this as choosing the terms with exponents in the bottom row of the L diagram for $w$ and $p$. Thus, cancelling the 1, we obtain
\begin{align*}
\sum_{i\ge0}x^i[x^{iw}]\eqref{pqw2}&=-(1+x^p+\dotsb+x^{(b-1)p})\sum_{k=1}^{w-1}x^{-kbp}(x^q+x^{2q}+\dotsb+x^{qa(w-k)})\\
&=-x(1+x^p+\dotsb+x^{(b-1)p})\sum_{k=1}^{w-1}x^{(w-k)bp}(1+x^q+\dotsb+x^{q(a(w-k)-1)})\\
&=-x(1+x^p+\dotsb+x^{(b-1)p})\sum_{k=1}^{w-1}x^{kbp}(1+x^q+\dotsb+x^{q(ak-1)})\\
&=-x\sum_{k=1}^{w-1}(x^{kbp}+x^{(kb+1)p}+\dotsb+x^{((k+1)b-1)p})(1+x^q+\dotsb+x^{q(ak-1)}).
\end{align*}

Similarly, in expression \eqref{pqw4}, the terms with exponents divisible by $w$ include the cancelled $x^{wpq-k(q-1)}=x^{w(pq-kbp)}$. This cancels with the product of the last terms of the first pair of geometric series since
\[x^kx^{wqa(k-p)}x^{pq(w+p-k-1)}=x^{k+q(p-1)(k-p)+pq(w+p-k-1)}=x^{k+pq(k-p)+q(p-k)+pq(w+p-k-1)}=x^{wpq-k(q-1)}.\]
Hence, the residues of the exponents of the first product mod $w$ are $k,k+1,\dotsc,w-1,0$, while those of the second product are $1,2,\dotsc,k-1$. Thus, the terms with exponents divisible by $w$ are obtained by choosing $x^{pq(w+p-k-1)}$ in expression \eqref{pqw4'}. Cancelling the $x^{wpq-k(q-1)}$ and dividing all the exponents by $w$, we get
\begin{align*}
\sum_{i\ge0}x^i[x^{iw}]\eqref{pqw4}&=-(1+x^p+\dotsb+x^{(b-1)p})\sum_{k=p+1}^{w+p-1}x^{(k+pq(w+p-k-1))/w}(1+x^q+\dotsb+x^{q(a(k-p)-1)})\\
&=-(1+x^p+\dotsb+x^{(b-1)p})\sum_{k=1}^{w-1}x^{(w+p-k+pq(k-1))/w}(1+x^q+\dotsb+x^{q(a(w-k)-1)})\\
&=-(1+x^p+\dotsb+x^{(b-1)p})\sum_{k=1}^{w-1}x^{(w+k(pq-1)-p(q-1))/w}(1+x^q+\dotsb+x^{q(a(w-k)-1)})\\
&=-(1+x^p+\dotsb+x^{(b-1)p})\sum_{k=1}^{w-1}x^{1+k(aq+bp)-bp^2}(1+x^q+\dotsb+x^{q(a(w-k)-1)})\\
&=-x^{a+1}(1+x^p+\dotsb+x^{(b-1)p})\sum_{k=1}^{w-1}x^{(k-1)(aq+bp)}(1+x^q+\dotsb+x^{q(a(w-k)-1)}).
\end{align*}
In the final step, we used the following useful conversion:
\begin{equation}bp^2+a=\frac{p(q-1)+(p-1)}w=\frac{pq-1}w=\frac{q(p-1)+(q-1)}w=aq+bp.\label{aq+bp}\end{equation}

Combining our results, we have produced the following formula for $\tilde F_{p,q,w,0}(x)=F_{pq,r,0}(x)$:
\begin{multline}F_{pq,r,0}(x)=(1+x^p+\dotsb+x^{pabw(w-1)})\\
-x\sum_{k=1}^{w-1}(x^{kbp}+x^{(kb+1)p}+\dotsb+x^{((k+1)b-1)p})(1+x^q+\dotsb+x^{q(ak-1)})\\
-x^{a+1}(1+x^p+\dotsb+x^{(b-1)p})\sum_{k=1}^{w-1}x^{(k-1)(aq+bp)}(1+x^q+\dotsb+x^{q(a(w-k)-1)}).\label{Fpqw0}\end{multline}

\vspace{.25cm}\noindent{\large\textsc{The $0<j<w$ Case}}

Now let $0<j<w$. Since as we just found, all exponents of expressions \eqref{pqw1} and \eqref{pqw5} are multiples of $w$, we just need to calculate $\tilde F_{p,q,w,j}(x)=\sum_{i\ge0}x^i[x^{iw+j}](\eqref{pqw2}+\eqref{pqw3}+\eqref{pqw4})$. In all cases, we can again factor out the $1+x^p+\dotsb+x^{(b-1)p}$.

The easiest to calculate is $\sum_{i\ge0}x^i[x^{iw+j}]\eqref{pqw3}$. Recall that $1-\tilde\Phi_w(x^{pq})=-x^{pq}-x^{2pq}-\dotsb-x^{(w-1)pq}$. Since $pq=1+w(aq+bp)$, the relevant term is $-x^{jpq}=-x^jx^{wj(aq+bp)}$. Hence,
\begin{equation}\sum_{i\ge0}x^i[x^{iw+j}]\eqref{pqw3}=-(1+x^p+\dotsb+x^{(b-1)p})\sum_{k=w}^px^{j(aq+bp)-kbp}.\label{pqw3j>0}\end{equation}

Now we calculate $\sum_{i\ge0}x^i[x^{iw+j}]\eqref{pqw2}$. We will distinguish between two cases: $0<k\le j$ and $j<k<w$. These essentially correspond to whether the exponents congruent to $j$ mod $w$ are in one quadrant of the L diagram or the other. Since $w\mid(q-1)$, the $x^{-k(q-1)}$ factors out as $x^{-kbp}$, as usual.

If $k\le j$, $x^{kq}x^{pq(j-k)}=x^jx^{w(j(aq+bp)-kaq)}$ is one of the terms in $x^{kq}(1+x^{pq}+\dotsb+x^{w-k-1})$ of expression \eqref{pqw2'}. Therefore, the result for these $k$ is $x^{-kbp}x^{j(aq+bp)-kaq}(1+x^q+\dotsb+x^{q(ak-1)})=x^{(j-k)(aq+bp)}(1+x^q+\dotsb+x^{q(ak-1)})$.

If $k>j$, $x^{pqj}=x^jx^{wj(aq+bp)}$ is one of the terms in $1+x^{pq}+\dotsb+x^{pq(k-1)}$ of expression \eqref{pqw2'}. Therefore, the result for these $k$ is $-x^{-kbp}x^{j(aq+bp)}(1+x^q+\dotsb+x^{qa(w-k)})$. In all, this gives us
\begin{align*}
\sum_{i\ge0}x^i[x^{iw+j}]\eqref{pqw2}&=(1+x^p+\dotsb+x^{(b-1)p})\Biggl[\sum_{k=1}^jx^{(j-k)(aq+bp)}(1+x^q+\dotsb+x^{q(ak-1)})\\
&\qquad-\sum_{k=j+1}^{w-1}x^{j(aq+bp)-kbp}(1+x^q+\dotsb+x^{qa(w-k)})\Biggr]\\
&=(1+x^p+\dotsb+x^{(b-1)p})\Biggl[\sum_{k=0}^{j-1}x^{k(aq+bp)}(1+x^q+\dotsb+x^{q(a(j-k)-1)})\\
&\qquad-\sum_{k=j+1}^{w-1}x^{j(aq+bp)-kbp}-\sum_{k=j+1}^{w-1}x^{j(aq+bp)-kbp}(x^q+x^{2q}+\dotsb+x^{qa(w-k)})\Biggr]\\
&=(1+x^p+\dotsb+x^{(b-1)p})\Biggl[\sum_{k=0}^{j-1}x^{kbp}(x^{qak}+x^{q(ak+1)}+\dotsb+x^{q(aj-1)})\\
&\qquad-\sum_{k=j+1}^{w-1}x^{j(aq+bp)-kbp}-\sum_{k=1}^{w-j-1}x^{j(aq+bp)-(w-k)bp+q}(1+x^q+\dotsb+x^{q(ak-1)})\Biggr]\\
&=(1+x^p+\dotsb+x^{(b-1)p})\Biggl[\sum_{k=0}^{j-1}x^{kbp}(x^{qak}+x^{q(ak+1)}+\dotsb+x^{q(aj-1)})\\
&\qquad-\sum_{k=j+1}^{w-1}x^{j(aq+bp)-kbp}-x\sum_{k=1}^{w-j-1}x^{j(aq+bp)+kbp}(1+x^q+\dotsb+x^{q(ak-1)})\Biggr].
\end{align*}

The last sum here is empty when $j=w-1$. We separated out the middle sum to combine it with $\sum_{i\ge0}x^i[x^{iw+j}]\eqref{pqw3}$. Before doing so, we calculate and manipulate $\sum_{i\ge0}x^i[x^{iw+j}]\eqref{pqw4}$. As before, we need two cases, $p<k<p+j$ and $p+j\le k<p+w$, and $x^{wpq-k(q-1)}$ factors out as $x^{pq-kbp}$.

If $k<p+j$, $-x^{k+pq(j+p-k-1)}=-x^jx^{w((j-k)(aq+bp)+apq)}$ is one of the terms in $-x^k(1+x^{pq}+\dotsb+x^{pq(w+p-k-1)})$ of expression \eqref{pqw4'}. Therefore, the result for these $k$ is $-x^{(j-k)(aq+bp)+apq}(1+x^q+\dotsb+x^{qa(k-p)})$.

If $k\ge p+j$, $x^{k+(w+p-k)q+pq(j-1)}=x^jx^{w(j(aq+bp)+q-kbp)}$ is one of the terms in $x^kx^{(w+p-k)q}(1+x^{pq}+\dotsb+x^{pq(k-p-1)})$ of expression \eqref{pqw4'}. Therefore, the result for these $k$ is $x^{j(aq+bp)+q-kbp}(1+x^q+\dotsb+x^{q(a(w+p-k)-1)})$. In all, this gives us
\begin{align*}
\sum_{i\ge0}x^i[x^{iw+j}]\eqref{pqw4}
&=(1+x^p+\dotsb+x^{(b-1)p})\Biggl[-\sum_{k=p+1}^{p+j-1}x^{(j-k)(aq+bp)+apq}(1+x^q+\dotsb+x^{qa(k-p)})\\
&\qquad+\sum_{k=p+j}^{w+p-1}x^{j(aq+bp)+q-kbp}(1+x^q+\dotsb+x^{q(a(w+p-k)-1)})\Biggr]\\
&=(1+x^p+\dotsb+x^{(b-1)p})\Biggl[-\sum_{k=p+1}^{p+j-1}x^{(j-k)(aq+bp)+apq}(1+x^q+\dotsb+x^{q(a(k-p)-1)})\\
&\qquad-\sum_{k=p+1}^{p+j-1}x^{(j-k)(aq+bp)+kaq}+\sum_{k=1}^{w-j}x^{j(aq+bp)+q-(w+p-k)bp}(1+x^q+\dotsb+x^{q(ak-1)})\Biggr]\\
&=(1+x^p+\dotsb+x^{(b-1)p})\Biggl[-\sum_{k=1}^{j-1}x^{k(aq+bp)-bp^2}(1+x^q+\dotsb+x^{q(a(j-k)-1)})\\
&\qquad-\sum_{k=p+1}^{p+j-1}x^{j(aq+bp)-kbp}+\sum_{k=1}^{w-j}x^{j(aq+bp)-(aq+bp-a)+kbp+1}(1+x^q+\dotsb+x^{q(ak-1)})\Biggr]\\
&=(1+x^p+\dotsb+x^{(b-1)p})\Biggl[-\sum_{k=1}^{j-1}x^{k(aq+bp)-bp^2}(1+x^q++\dotsb+x^{q(a(j-k)-1)})\\
&\qquad-\sum_{k=p+1}^{p+j-1}x^{j(aq+bp)-kbp}+x^{a+1}\sum_{k=1}^{w-j}x^{(j-1)(aq+bp)+kbp}(1+x^q+\dotsb+x^{q(ak-1)})\Biggr].
\end{align*}
In going from the second to the third line, we have replaced $bp^2$ (which results from expanding $(w+p-k)bp$) with $aq+bp-a$ by equation \eqref{aq+bp}. Of course, the first sum is empty when $j=1$. Now we combine the two middle sums we have separated out with equation \eqref{pqw3j>0}. The result is
\begin{align*}
-(1+x^p+\dotsb+x^{(b-1)p})\sum_{k=j+1}^{p+j-1}x^{j(aq+bp)-kbp}&=-(1+x^p+\dotsb+x^{(b-1)p})\sum_{k=0}^{p-2}x^{j(aq+bp)-(p+j-1-k)bp}\\
&=-x^{jaq-bp(p-1)}(1+x^p+\dotsb+x^{(b-1)p})\sum_{k=0}^{p-2}x^{kbp}\\
&=-x^{a+(j-1)aq}(1+x^p+\dotsb+x^{(b(p-1)-1)p}).
\end{align*}
In the last step, we have used equation \eqref{aq+bp} to write $-bp(p-1)=bp-bp^2=a-aq$. Note that everything else is a multiple of $(1+x^q+\dotsb+x^{q(a-1)})$ as well as $(1+x^p+\dotsb+x^{p(b-1)})$. Therefore, combining all three expressions, we have
\begin{multline}F_{pq,r,j}(x)=-x^{a+(j-1)aq}(1+x^p+\dotsb+x^{(b(p-1)-1)p})+(1+x^p+\dotsb+x^{(b-1)p})(1+x^q+\dotsb+x^{q(a-1)})\\
\Biggl[\sum_{k=0}^{j-1}x^{kbp}(x^{qak}+x^{qa(k+1)}+\dotsb+x^{qa(j-1)})+x^{a+1}\sum_{k=1}^{w-j}x^{(j-1)(aq+bp)+kbp}(1+x^{qa}+\dotsb+x^{qa(k-1)})\\
-\sum_{k=1}^{j-1}x^{k(aq+bp)-bp^2}(1+x^{qa}+\dotsb+x^{qa(j-k-1)})-x\sum_{k=1}^{w-j-1}x^{j(aq+bp)+kbp}(1+x^{qa}+\dotsb+x^{qa(k-1)})\Biggr].\label{wpqFj}\end{multline}

\subsubsection{Calculating $F_{pq,r,j}(x)$ for $j\ge w$ and showing flatness}\label{shortcuts}
Of course, $F_{pq,r,j}(x)=\tilde F_{p,q,w,j}(x)$ is flat for $0\le j<w$, since $\tilde\Phi_{p,q,w}(x)$ is flat by Proposition \ref{pseudor=1prop}. Our task is now to show that $F_{pq,r,j}(x)$ is flat for the other $w\le j<r$. At this point, we have no more use for the $\tilde F_{p,q,w,j}(x)$, so we drop the $pq$ and $r$ in the subscripts of the $F_{pq,r,j}(x)$.

Since the $F_j(x)$ are periodic, i.e. $F_{j+pq}(x)=F_j(x)$ for $0\le j<r-pq$ (by Corollary \ref{F=Fcor}), it suffices to consider $w\le j<pq$. We have a few more shortcuts to take. From Section \ref{Fj(0)section}, recall that it always suffices only to consider those $F_j(x)$ that have nonzero constant term, since other $F_j(x)$ are just these multipled by $x^k$ for various $k$. Because $\Psi_{pq}(x)=-1-x-\dotsb-x^{p-1}+x^q+x^{q+1}+\dotsb+x^{p+q-1}$, we know that $F_j(0)\neq0$ with $0\le j<pq$ only for the two ranges $0\le j<p$ and $q\le j<p+q$. We now use the reciprocity results of Section \ref{reciprocity} to lump these cases together into the larger case $q\le j<pq$. Clearly the second case already lies in this range. Since $r>(p-1)(q-1)$, when we write $-(p-1)(q-1)=yr+z$, $y=-1$ and $z=r-(p-1)(q-1)$. Then using equation \eqref{reciprocitygeneral} and equation \eqref{F=Fextended}, when $j<r-(p-1)(q-1)$,
\[V(F_j(x))=V(F_{z-j}(x))=V(F_{r-(p-1)(q-1)-j}(x))=V(F_{pq+w-(p-1)(q-1)-j}(x))=V(F_{w+p+q-1-j}(x)).\]
Since $r\ge pq+w$, we have $r-(p-1)(q-1)\ge w+p+q-1>p$. Therefore, when $j$ is in the first of the ranges we must check, $0\le j<p$, $q\le w+q-1<w+p+q-1-j\le w+p+q-1<pq$, or $w+p+q-1-j\in[q,pq)$. We have shown that $V(\Phi_{pqr}(x))=\bigcup_{j=q}^{pq-1}V(F_j(x))$.

Next we actually compute these $F_j(x)$. The relationships between the (extended) $F_{pq,r,j}(x)$ reveal that for $0\le j<w$,
\begin{equation}xF_j(x)=F_{j-r}(x)\equiv F_{j+pq-w}(x)\pmod{\Phi_{pq}(x)},\label{Fj+pq-w}\end{equation}
and for $w\le j<pq$,
\begin{equation}xF_j(x)=F_{j-r}(x)\equiv F_{j-w}(x)\pmod{\Phi_{pq}(x)}.\label{Fj-w}\end{equation}
Now consider some $q\le j<pq$. Then we have $bpw\le j-1<w(aq+bp)$. Parameterize with $j-1=cw+d$, where $0\le d<w$ and therefore, $bp\le c<aq+bp$. By equation \eqref{Fj-w} applied $aq+bp-c-1$ times and equation \eqref{Fj+pq-w} applied once,
\begin{equation}F_j(x)\equiv x^{aq+bp-c-1}F_{d+1+w(aq+bp)-w}(x)\equiv x^{aq+bp-c-1}F_{d+pq-w}(x)\equiv x^{aq+bp-c}F_d(x)\pmod{\Phi_{pq}(x)}.\label{Fcw+d+1}\end{equation}

Since $0\le d<w$, we have already calculated $F_d(x)$. Given that equations \eqref{Fpqw0} and \eqref{wpqFj}, with $d$ substituted for $j$, distinguish between $d=0$ and $d>0$, we must continue that distinction. First, let $d=0$.

\vspace{.25cm}\noindent{\large\textsc{The $d=0$ Case}}

We begin by finding $\deg F_0(x)$. Equation \eqref{Fpqw0} has three components, with degrees of
\begin{align*}
pabw(w-1)&=(p-1-a)(q-1)\\
&=(p-1)(q-1)-aq+a\\
1+(wb-1)p+q(a(w-1)-1)&=-p+q(p-1)-aq\\
&=(p-1)(q-1)-aq-1\\
a+1+(b-1)p+(w-2)(aq+bp)+q(a-1)&=a+1-p+(pq-1)-(aq+bp)-q\\
&=(p-1)(q-1)-aq+a-bp-1.
\end{align*}
The largest of these is the first, so $\deg F_0(x)=(p-1)(q-1)-aq+a$. Therefore, $\deg x^{aq+bp-c}F_0(x)<(p-1)(q-1)$ if and only if $aq+bp-c<aq-a$, or $c>bp+a$. Hence, when $c>bp+a$, we have $F_j(x)=x^{aq+bp-c}F_0(x)$, which is flat because $F_0(x)$ is, so we are done when $c>bp+a$.

We group all the remaining cases together: Let $bp\le c\le bp+a$. Then we claim that
\begin{equation}F_j(x)=x^{aq+bp-c}F_0(x)-(1+x+\dotsb+x^{bp+a-c})\Phi_{pq}(x).\label{wpqFjd=0}\end{equation}
To prove this, note that it clearly satisfies congruence \eqref{Fcw+d+1}, so it suffices to check the degree. Since $c\ge bp$, we know $aq+bp-c\le aq$, which means that the terms of the second and third components of $F_0(x)$ still have appropriate degrees after multiplying by $x^{aq+bp-c}$. The two largest degrees in the first component are $(p-1)(q-1)-aq+a$ and $(p-1)(q-1)-aq+a-p$, so as $a<p$, the only degree at least $(p-1)(q-1)$ is the first, whose degree becomes $(p-1)(q-1)-aq+a+aq+bp-c=(p-1)(q-1)+bp+a-c$. As $bp+a-c\le a<p$, this is also the only degree of $(1+x+\dotsb+x^{bp+a-c})\Phi_{pq}(x)$ at least $(p-1)(q-1)$. Therefore, as both are positive terms, the difference is a polynomial of degree less than $(p-1)(q-1)$, as desired. We have established equation \eqref{wpqFjd=0}.

Proving that this is flat is a bit trickier. First we rewrite equation \eqref{pq(1+x)q=1}, replacing $2k$ with $bw$ as $q=bwp+1$, rather than $2kp+1$: For $1\le l<p+q$,
\begin{multline}(1+x+\dotsb+x^{l-1})\Phi_{pq}(x)=(1+x^p+\dotsb+x^{p(q-bwl-1)})(1+x^q+\dotsb+x^{q(l-1)})\\
-x^l(1+x^p+\dotsb+x^{p(bwl-1)})(1+x^q+\dotsb+x^{q(p-l-1)}).\label{wpql}\end{multline}
In this situation, we are letting $l=bp+a-c+1$ so $1\le l\le a+1<p+q$.

Now we define $u(x)=x^{aq+bp-c}F_0(x)$ and $v(x)=(1+x+\dotsb+x^{l-1})\Phi_{pq}(x)$. Of course, both $u(x)$ and $v(x)$ are flat. Define $U_-=\{i|[x^i]u(x)=-1\},U_+,V_-,V_+$ as in the proof of Theorem \ref{r=2thm}. To prove that $F_j(x)=u(x)-v(x)$ is flat, we must show that $U_+\cap V_-=U_-\cap V_+=\emptyset$. We have
\begin{align*}
U_+&=aq+bp-c+p\{0,1,\dotsc,q-bw-bp+b-1\}.\\
U_-&=U_{-,1}\cup U_{-,2}\text{, where}\\
U_{-,1}&=\bigcup_{k=1}^{w-1}U_{-,1,k}\text{, and}\\
U_{-,1,k}&=aq+bp-c+1+(p\{kb,kb+1,\dotsc,(k+1)b-1\})\oplus(q\{0,1,\dotsc,ak-1\});\\
U_{-,2}&=\bigcup_{k=1}^{w-1}U_{-,2,k}\text{, and}\\
U_{-,2,k}&=aq+bp-c+a+1+(k-1)(aq+bp)+(p\{0,1,\dotsc,(b-1)\})\oplus(q\{0,1,\dotsc,a(w-k)-1\}).\\
V_+&=(p\{0,1,\dotsc,q-bwl-1\})\oplus(q\{0,1,\dotsc,l-1\}).\\
V_-&=l+(p\{0,1,\dotsc,bwl-1\})\oplus(q\{0,1,\dotsc,p-l-1\}).
\end{align*}

To show that $U_+\cap V_-=\emptyset$, we show that the elements of $U_+$ can be written as a nonnegative linear combination of $p$ and $q$, while those of $V_-$ cannot. In the L diagram, the numbers that can be written as a nonnegative linear combination of $p$ and $q$ are below and to the left of the diagonal line connecting the upper left corner to the lower right, and those that cannot are above and to the right of this line.

Indeed, $aq+bp-c=aq+l-a-1=(l-1)q+(a+1-l)bwp$, which is nonnegative because $l\ge1$ and $l\le a+1$. All elements of $U_+$ can be written as a sum of a nonnegative multiple of $p$ and $aq+bp-c$, so they are nonnegative linear combinations of $p$ and $q$, as desired. Each element of $V_-$ can be written as $l+p\alpha+q\beta$ where $0\le\alpha<bwl$ and $0\le\beta<p-l$. This equals $l(q-bwp)+p\alpha+q\beta=p(\alpha-bwl)+q(\beta+l)$. Since $\alpha-bwl<0$ and $0\le\beta+l<p$, elements of $V_-$ cannot be written as a nonnegative linear combination of $p$ and $q$, so $U_+\cap V_-=\emptyset$, as desired.

To show that $U_-\cap V_+=\emptyset$, we separately show that $U_{-,1}\cap V_+=\emptyset$ and $U_{-,2}\cap V_+=\emptyset$. For the first, consider each modulo $p$ (we show they occupy different rows of the L diagram). We have
\begin{align*}
V_+&\equiv\{0,1,\dotsc,l-1\}\pmod p\\
U_{-,1,k}&\equiv l+\{0,1,\dotsc,ak-1\}\pmod p.
\end{align*}
Since $l\le a+1$, we have $l+ak-1\le a(k+1)\le aw<p$, and $U_{-,1,k}\cap V_+=\emptyset$.

For the second, consider each modulo $q$ (we show they occupy different columns of the L diagram). We have
\begin{align*}
V_+&\equiv p\{0,1,\dotsc,q-bwl-1\}\pmod q\\
U_{-,2,k}&\equiv l+(k-1)bp+p\{0,1,\dotsc,b-1\}\pmod q\\
&\equiv pq-bwpl+(k-1)bp+p\{0,1,\dotsc,b-1\}\pmod q\\
&\equiv p\{q-bwl+b(k-1),q-bwl+b(k-1)+1,\dotsc,q-bwl+bk-1\}\pmod q.
\end{align*}
If we cancel the $p$ (for instance by multiplying by $-bw$), we must show that $\{0,1,\dotsc,q-bwl-1\}$ and $\{q-bwl+b(k-1),q-bwl+b(k-1)+1,\dotsc,q-bwl+bk-1\}$ are disjoint modulo $q$. As integers, all of these terms are nonnegative and less than $q$ since $k<w\le wl$, but the terms in the second sequence are at least $q-bwl$ while those in the first are less than $q-bwl$, so they are disjoint. Hence, $U_{-,2,k}\cap V_+=\emptyset$, $U_{-,2}\cap V_+=\emptyset$ and $U_-\cap V_+=\emptyset$, and expression \eqref{wpqFjd=0} is flat, as desired.

\vspace{.25cm}\noindent{\large\textsc{The $0<d<w$ Case}}

Now we repeat many of the same arguments for $0<d<w$. The degrees of each component of $F_d(x)$ in expression \eqref{wpqFj} with $d$ replacing $j$, are
\begin{align*}
&a+(d-1)aq+(b(p-1)-1)p=daq-a(bpw)+bp(aw)-p=qad-p\\
&\qquad\le q(p-1-a)-p=(p-1)(q-1)-aq-1\\
&(b-1)p+(a-1)q+(d-1)(bp+aq)=w(bp+aq)-(w-d)(bp+aq)-p-q\\
&\qquad\le(p-1)(q-1)-aq-bp-2\\
&(b-1)p+(a-1)q+a+1+(d-1)(aq+bp)+(w-d)bp+aq(w-d-1)\\
&\qquad=(aq+bp)(1+d-1+w-d)-p-q+a+1-aq=(p-1)(q-1)-aq+a-1\\
&(b-1)p+(a-1)q+(d-1)(aq+bp)-bp^2=d(aq+bp)-p-q+(a-aq-bp)\\
&\qquad=(p-1)(q-1)-(w-d+1)(aq+bp)+a-2\le(p-1)(q-1)-2(aq+bp)+a-2\\
&(b-1)p+(a-1)q+1+d(aq+bp)+(w-d-1)bp+qa(w-d-2)\\
&\qquad=-p-q+1+(aq+bp)(1+d+w-d-1)-aq=(p-1)(q-1)-aq-1.
\end{align*}
We see that the largest of these is $(p-1)(q-1)-aq+a-1$ from the third component, and the second largest $(p-1)(q-1)-aq-1$, from either the first or the last component. (When the first component achieves that upper bound, $d=w-1$, which makes the last component zero, since $1\le k\le w-d-1=0$.) Also, the second largest degree in the third component is $(p-1)(q-1)-aq+a-1-p<(p-1)(q-1)-aq-1$, so there is only one term with degree above $(p-1)(q-1)-aq-1$.

We have $\deg F_d(x)=(p-1)(q-1)-aq+a-1$, which makes $\deg x^{aq+bp-c}F_d(x)<(p-1)(q-1)$ when $aq+bp-c\le aq-a$, or $c\ge bp+a$. Hence, when $c\ge bp+a$, $F_j(x)=x^{aq+bp-c}F_d(x)$, which is flat since $F_d(x)$ is.

Therefore, the final case is $bp\le c<bp+a$. In this case, we claim that
\begin{equation}F_j(x)=x^{aq+bp-c}F_d(x)-(1+x+\dotsb+x^{bp+a-c-1})\Phi_{pq}(x).\label{wpqFjd>0}\end{equation}

Let $l=bp+a-c$, so we are looking at $(1+x+\dotsb+x^{l-1})\Phi_{pq}(x)$, and $0<l\le a$. Indeed, as $l\le a<p$, the leading terms of $(1+x+\dotsb+x^{bp+a-c-1})\Phi_{pq}(x)$ are $x^{(p-1)(q-1)+bp+a-c-1}-x^{(p-1)(q-1)-1}$. Meanwhile, from the degree bounds above, the leading terms of $x^{aq+bp-c}F_j(x)$ are
\[x^{aq+bp-c}(x^{(p-1)(q-1)-aq+a-1}-x^{(p-1)(q-1)-aq-1})=x^{(p-1)(q-1)+bp+a-c-1}-x^{(p-1)(q-1)+bp-c-1}.\]
The $x^{(p-1)(q-1)+bp+a-c-1}$ cancel, and since $c\ge bp$, all other terms have degrees below $(p-1)(q-1)$, as desired. Thus, equation \eqref{wpqFjd>0} holds.

Now we show that expression \eqref{wpqFjd>0} is flat. First recall equation \eqref{wpql}:
\begin{multline*}(1+x+\dotsb+x^{l-1})\Phi_{pq}(x)=(1+x^p+\dotsb+x^{p(q-bwl-1)})(1+x^q+\dotsb+x^{q(l-1)})\\
-x^l(1+x^p+\dotsb+x^{p(bwl-1)})(1+x^q+\dotsb+x^{q(p-l-1)}).\end{multline*}

Let $u(x)=x^{aq+bp-c}F_j(x)$ and $v(x)=(1+x+\dotsb+x^{l-1})\Phi_{pq}(x)$. Define as usual $U_-=\{i|[x^i]u(x)=-1\}$ and $U_+=\{i|[x^i]u(x)=1\}$, and $V_-,V_+$ similarly. Then
\begin{align*}
U_+&=U_{+,1}\cup U_{+,2};\quad U_{+,1}=\bigcup_{k=0}^{d-1}U_{+,1,k};\quad U_{+,2}=\bigcup_{k=1}^{w-d}U_{+,2,k};\\
U_{+,1,k}&=aq+bp-c+kbp+(p\{0,1,\dotsc,(b-1)\})\oplus(q\{ak,ak+1,\dotsc,ad-1\});\\
U_{+,2,k}&=aq+bp-c+a+1+(d-1)(aq+bp)+kbp+p\{0,1,\dotsc,(b-1)\}\oplus(q\{0,1,\dotsc,ak-1\}).\\
U_-&=U_{-,1}\cup U_{-,2}\cup U_{-,3};\quad U_{-,2}=\bigcup_{k=1}^{d-1}U_{-,2,k};\quad U_{-,3}=\bigcup_{k=1}^{w-d-1}U_{-,3,k};\\
U_{-,1}&=aq+bp-c+a+(d-1)aq+p\{0,1,\dotsc,b(p-1)-1\};\\
U_{-,2,k}&=aq+bp-c+k(aq+bp)-bp^2+(p\{0,1,\dotsc,b-1\})\oplus(q\{0,1,\dotsc,a(d-k)-1\});\\
U_{-,3,k}&=aq+bp-c+1+d(aq+bp)+kbp+(p\{0,1,\dotsc,b-1\})\oplus(q\{0,1,\dotsc,ak-1\}).\\
V_+&=(p\{0,1,\dotsc,q-bwl-1\})\oplus(q\{0,1,\dotsc,l-1\}).\\
V_-&=l+(p\{0,1,\dotsc,bwl-1\})\oplus(q\{0,1,\dotsc,p-l-1\}).
\end{align*}

To show that $U_{+,1,k}\cap V_-=\emptyset$, we show that elements of $U_{+,1,k}$ can be written as a nonnegative linear combination of $p$ and $q$, while we have already shown (from the case $d=0$) that those of $V_-$ cannot. Indeed, it suffices to write $aq+bp-c=l+abpw=l(q-bpw)+abpw=(a-l)bwp+lq$ which is nonnegative since $0<l\le a$.

To show that $U_{+,2,k}\cap V_-=\emptyset$, we consider each modulo $q$ (showing they are in different columns of the L diagram). We have
\begin{align*}
V_-&\equiv l+p\{0,1,\dotsc,bwl-1\}\pmod q\\
U_{+,2,k}&\equiv bp-c+a+1+(d-1)bp+kbp+p\{0,1,\dotsc,b-1\}\pmod q\\
&\equiv l+(q-bpw)+bp(d-1+k)+p\{0,1,\dotsc,b-1\}\pmod q\\
&\equiv l+p\{q-b(w-d+1-k),q-b(w-d+1-k)+1,\dotsc,q-b(w-d-k)-1\}\pmod q.
\end{align*}
Subtracting $l$ and dividing by $p$, we must show that $\{0,1,\dotsc,bwl-1\}$ and $\{q-b(w-d+1-k),q-b(w-d+1-k)+1,\dotsc,q-b(w-d-k)-1\}$ do not intersect modulo $q$. Since $l\le a<p$, $w-d+1-k\le w\le wp$ and $k\le w-d$, both of these sets are nonnegative and less than $q$. But $q-b(w-d+1-k)=b(wp-w+d-1+k)+1\ge baw^2+1\ge bwl+1>bwl-1$ so these sets are disjoint modulo $q$, as desired, and $U_{+,2,k}\cap V_-=\emptyset$. Therefore, $U_+\cap V_-=\emptyset$.

To show that $U_-\cap V_+=\emptyset$, consider $U_{-,1,k}$, $U_{-,2,k}$, $U_{-,3,k}$ and $V_+$ modulo $p$. We have
\begin{align*}
V_+&\equiv\{0,1,\dotsc,l-1\}\pmod p\\
U_{-,1}&\equiv l+a+a(d-1)\equiv l+ad\pmod p\\
U_{-,2,k}&\equiv l+ak+\{0,1,\dotsc,a(d-k)-1\}\equiv\{l+ak,l+ak+1,\dotsc,l+ad-1\}\pmod p\\
&\subseteq\{l+a,l+a+1,\dotsc,l+ad-1\}\pmod p\\
U_{-,3,k}&\equiv l+1+ad+\{0,1,\dotsc,ak-1\}\equiv\{l+ad+1,l+ad+2,\dotsc,l+ad+ak\}\pmod p\\
&\subseteq\{l+ad+1,l+ad+2,\dotsc,l+a(w-1)\}\pmod p\text{, so}\\
U_-&\subseteq\{l+a,l+a+1,\dotsc,l+a(w-1)\}\pmod p.
\end{align*}
Since $l+a>l-1$ and $l+a(w-1)=l+p-a-1<p$ as $l\le a$, we have $U_-\cap V_+=\emptyset$.

Therefore, we have shown that expression \eqref{wpqFjd>0} is flat. With the shortcuts we introduced earlier, we have shown that $F_{pq,r,j}(x)$ is flat for all $j$. We are done.\end{proof}

Note that Theorem \ref{BroadhurstII} provides flat ternary cyclotomic polynomials where the largest prime, $r$, is arbitrarily far from a multiple of the other two, $pq$, the first such family. Indeed, for any $w$, Dirichlet's theorem on primes in arithmetic progressions guarantees the existence of primes $p\equiv1\pmod w$, $q\equiv1\pmod{pw}$, and $r\equiv w\pmod{pq}$.

Finally, the pseudocyclotomic analog of Theorem \ref{BroadhurstII} is again easy to see and the proof is identical:

\begin{prop}\label{pseudoBroadhurstII}Let $p<q<r$ be pairwise relatively prime positive integers, and $w$ a positive integer such that $r\equiv w\pmod{pq}$, $p\equiv1\pmod w$, and $q\equiv 1\pmod{wp}$. Then $\tilde\Phi_{p,q,r}(x)$ is flat.\end{prop}

\subsection{A bound on $A(pqr)$ for small $w$}

Let $n=pq$ and $\frac{1-pq}2\le w\le\frac{pq-1}2$ be the least residue $w\equiv r\pmod{pq}$.

Theorem \ref{r=1thm} analyzed $\lvert w\rvert=1$, Theorem \ref{r=2thm} analyzed $\lvert w\rvert=2$, and Theorem \ref{BroadhurstII} covered a specific case for all $w$. We can use the $F^*_j(x)$ from Section \ref{Fj(0)section} to prove some more general results for small $\lvert w\rvert$. In particular, we have the following result, which also appears as Theorem 2.7 in \cite{ZhaoZhang}.
\begin{thm}\label{|w|thm}Let $p,q,r$ be primes and $w$ an integer such that $r\equiv w\pmod{pq}$. Then $A(pqr)\le\lvert w\rvert$.\end{thm}
\begin{proof}By Theorem \ref{signedperiodicitythm}, it suffices to prove this for $w>0$.

Since $-\Psi_{pq}(x)=1+x+\dotsb+x^{p-1}-x^q-\dotsb-x^{q+p-1}$,
\begin{equation}F_j(0)=\begin{cases}1\text{ if }0\le j<p\\-1\text{ if }q\le j<p+q\\0\text{ if }p\le j<q\text{ or }p+q\le j<pq.\end{cases}\label{pqF(0)}\end{equation}
Of course, other $F_j(0)$ can be computed using $F_j(x)=F_{j+pq}(x)$. Fix $0\le j<pq$, and we will show that $V(F_j(x))\subseteq\{-w,-w+1,\dotsc,w\}$. By equation \eqref{F'coeffseq}, $[x^k]F^*_{j+k}(x)=\sum_{i=0}^kF^*_{j+i}(0)[x^i]\Phi_n(x)$. We remarked in Section \ref{Fj(0)section} that this makes $[x^k]F^*_{j+k}(x)$ into a sequence of partial sums, $\{[x^k]F^*_{j+k}(x)\}_{k\in\Z}$. At $k=-1$, it is zero, and at $k=(p-1)(q-1)$ it is also zero, since $\deg F^*_{j+k}(x)<(p-1)(q-1)$. In between, $F^*_{j+i}(0)$ takes on values of $\pm1$ and $0$, as does $[x^i]\Phi_{pq}(x)$.

Since $w\equiv r\pmod{pq}$, $F^*_{j+i}(0)=F_{-r(j+i)}(0)=F_{-w(j+i)}(0)$. Consider $-w(j+i)$ as $i$ goes from $0$ to $(p-1)(q-1)$. It passes through at most $w$ regions in equation \eqref{pqF(0)} in which $F_{-w(j+i)}(0)=1$ and at most $w$ regions in which $F_{-w(j+i)}(0)=-1$. Since the coefficients of $\Phi_{pq}(x)$ alternate in sign, the sum of consecutive coefficients of $\Phi_{pq}(x)$ is at most 1 in absolute value. Therefore, the partial sum $[x^k]F^*_{j+k}(x)$ can change by at most 1 in each of these $2w$ regions and is constant outside of them. As it starts and ends at 0, $\lvert[x^k]F^*_{j+k}(x)\rvert\le w$ for all $j,k$. Therefore, $A(pqr)\le\lvert w\rvert$ as desired.\end{proof}

\begin{cor}\label{r=2cor}If $p<q<r$ are primes such that $r\equiv\pm2\pmod{pq}$ and $q\not\equiv1\pmod p$, then $A(pqr)=2$.\end{cor}
\begin{proof}By Theorem \ref{r=2thm}, $A(pqr)>1$, but by Theorem \ref{|w|thm}, $A(pqr)\le 2$. Therefore, $A(pqr)=2$.\end{proof}
As far as I can tell, this is the first infinite family of cyclotomic polynomials with height exactly 2, without fixing $p$.

\subsection{Open Questions about Ternary Cyclotomic Polynomials}\label{openquestions3}

Numerical data suggests that most ternary cyclotomic polynomials are not flat. For some pairs of primes $(p,q)$, the only primes $r>q$ for which $A(pqr)=1$ are $r\equiv\pm1\pmod{pq}$, which were proven flat by Theorem \ref{r=1thm}. Proving the following conjecture would verify this observation.

\begin{conj}\label{notflat}If $p<q<r$ are odd primes such that $A(pqr)=1$, then either $q\equiv\pm1\pmod p$ or $r\equiv\pm1\pmod{pq}$.\end{conj}

Theorem \ref{r=2thm} provides a family of $r\not\equiv\pm1\pmod{pq}$ such that $A(pqr)=1$ when $q\equiv1\pmod p$. An analogous family when $q\equiv-1\pmod p$ does not exist for all $(p,q)$ with $q\equiv-1\pmod p$; $(p,q)=(3,5)$ and $(p,q)=(5,19)$ are counterexamples. On the other hand, there do exist flat examples with $q\equiv-1\pmod p$ and $r\neq\pm1\pmod{pq}$ such as $(3,11,41)$ and $(5,29,499)$.

\begin{ques}\label{+-1onlyques}For which pairs of primes $(p,q)$ with $q\equiv-1\pmod p$ does $A(pqr)=1$ imply $r\equiv\pm1\pmod{pq}$?\end{ques}

Broadhurst's \cite{BH2} third category of flat ternary cyclotomic polynomials, as documented in \cite{Kap10}, amounts to Corollary \ref{Phi(x^w)/Phi(x)cor}, Conjecture \ref{notflat}, and the following additional claims.

\begin{conj}\label{BroadhurstIII}Let $p<q<r$ be odd primes such that $A(pqr)=1$, and let $w$ be the smallest positive integer such that $r\equiv\pm w\pmod{pq}$. Suppose that $p\not\equiv1\pmod w$. Then $w>p$, $q>p^2-p$, $q\equiv\pm1\pmod p$ and $w\equiv\pm1\pmod p$. Moreover, if $w\equiv1\pmod p$, then $q\not\equiv\pm1\pmod{wp}$.\end{conj}

Note that the converse of this conjecture is not true. For instance, let $p=3$, $q=7$ and $r=29$. Then $29\equiv8\pmod{3\cdot7}$ so $w=8$. We have $w=8>3=p$, $q=7>6=3^2-3=p^2-p$, $q=7\equiv1\pmod 3$, and $8\equiv-1\pmod3$, but an easy calculation shows that $A(3\cdot7\cdot29)=2$.

We can generalize these conjectures to pseudocyclotomic polynomials, where they appear from limited data to hold as well.

\begin{conj}\label{pseudonotflat}If $p<q<r$ are pairwise relatively prime positive integers, and $\tilde\Phi_{p,q,r}(x)$ is flat, then $q\equiv\pm1\pmod p$ or $r\equiv\pm1\pmod{pq}$.\end{conj}

\begin{conj}\label{pseudoBroadhurstIII}Let $p<q<r$ be relatively prime positive integers such that $\tilde\Phi_{p,q,r}(x)$ is flat, and let $w$ be the smallest positive integer such that $r\equiv\pm w\pmod{pq}$. Suppose that $p\not\equiv1\pmod w$. Then $w>p$, $q>p^2-p$, $q\equiv\pm1\pmod p$ and $w\equiv\pm1\pmod p$. Moreover, if $w\equiv1\pmod p$, then $q\not\equiv1\pmod{wp}$.\end{conj}

However, while the ternary cyclotomic polynomials $\Phi_{pqr}(x)$ have been computed for many prime triples $(p,q,r)$ (for instance, by Arnold and Monagan \cite{AM}), the pseudocyclotomic polynomials $\tilde\Phi_{p,q,r}(x)$ have not, so there could still be a small counterexample to Conjectures \ref{pseudonotflat} or \ref{pseudoBroadhurstIII}. Since the proof methods of this paper apply equally well to pseudocyclotomic polynomials, such a counterexample would be very informative.

\subsection{Forbidden Binomials}

In this section we look at efforts to apply our method to prove Conjecture \ref{notflat}. We let $n=pq$ be the product of two primes and $r>n$ a third prime. We let $r>n$ for simplicity. In proving that polynomials are not flat, we simply try to find a coefficient of absolute value 2. To calculate individual coefficients, recall equation \eqref{x^kFj}:
\[[x^k]F^*_j(x)=\sum_{i=0}^kF^*_{j-k+i}(0)[x^i]\Phi_n(x).\]

Recall that $F^*_j(x)=F_{-jr}(x)\equiv x^jF_0(x)\pmod{\Phi_n(x)}$ where the subscript $-jr$ is taken modulo $n$ (rather than the usual extension of the $F_j(x)$). Also recall from equation \eqref{pqF(0)} that since $n=pq$, for $0\le j<pq$, $F_j(0)=1$ for $0\le j<p$, $F_j(0)=-1$ for $q\le j<p+q$, and $F_j(0)=0$ otherwise.

Given $0\le j<n$, let $F^\circ_j(x)=\sum_{i\ge0}x^iF^*_{j-i}(0)=\sum_{i\ge0}x^iF_{r(i-j)}(0)$. This is not actually a polynomial; it is a power series, in fact, a multiple of $1+x^n+\dotsb$ due to the periodicity in the $F_j(x)$. This definition allows us to write equation \eqref{x^kFj} as
\begin{align*}
[x^k]F^*_j(x)&=\sum_{i=0}^k([x^{k-i}]F^\circ_j(x))([x^i]\Phi_n(x))=[x^k](F^\circ_j(x)\Phi_n(x))\\
F^*_j(x)&=F^\circ_j(x)\Phi_n(x),
\end{align*}
which motivates this definition. The definition also implies that $[x^i]F^\circ_j(x)=[x^{i+1}]F^\circ_{j+1}(x)$, so all power series $F^\circ_j(x)$ are derived from the same cyclic list of coefficients, starting at different places. Of course, this list is the $F^*_j(0)$, scrambled in a way that depends on the residue of $r$ modulo $n$.

Returning to the problem of discovering particular coefficients of absolute value 2, we are therefore interested in which flat polynomials or power series, when multiplied by $\Phi_n(x)$, yield non-flat polynomials or power series. This motivates the following definition:

\begin{defn}Given a flat cyclotomic polynomial $\Phi_n(x)$, define a \emph{forbidden binomial} to be a binomial of the form $x^a\pm x^b$, where $a<b$, such that $(x^a\pm x^b)\Phi_n(x)$ is not flat.\end{defn}

Under appropriate conditions, the presence of a forbidden binomial as two of the terms of $F^\circ_j(x)$ can cause $F^*_j(x)$ to be not flat. In the following proposition, these conditions are satisfied for $1-x$, which is a forbidden binomial because $[x]((1-x)\Phi_{pq}(x))=[x](1-x)^2=-2$.

\begin{prop}\label{q-p<w<q+pprop}Let $p<q<r$ be odd primes with $r>pq$ and $w<pq/2$ the positive integer such that $r\equiv\pm w\pmod{pq}$. Then if $q-p<w<q+p$, we have $A(pqr)>1$.\end{prop}
\begin{proof}We will show that there exists some $F^\circ_j(x)$ with smallest degree terms $\pm(1-x)$. Then since the smallest degree terms of $\Phi_n(x)$ are similarly $1-x$, we have $[x]F^*_j(x)=[x](F^\circ_j(x)\Phi_n(x))=\pm[x]((1-x)^2)=\mp2$, so both $F^*_j(x)$ and $\Phi_{pqr}(x)$ are not flat, as desired. Since the coefficients of the $F^\circ_j(x)$ are derived from the same cyclic list of coefficients, we must merely show that two consecutive terms of $F^\circ_0(x)$ have opposite signs and differ in degree by one, that is, to show that two consecutive $F^*_j(0)$ have opposite signs.

By signed periodicity (Theorem \ref{signedperiodicitythm}), it suffices to consider $r\equiv w\pmod{pq}$. When $q-p<w<q$, $F^*_j(0)=F_q(0)$ and $F^*_{j+1}(0)=F_{q-r}(0)=F_{q-w}(0)$ are consecutive, and $F_q(0)=-1$ while $F_{q-w}(0)=1$ since $0<q-w<p$. When $q<w<q+p$ ($w\neq q$ as $(q,r)=1$), $F^*_j(0)=F_w(0)$ and $F^*_{j+1}(0)=F_{w-r}(0)=F_0(0)$ are consecutive, and $F_0(0)=1$ while $F_w(0)=-1$ since $q<w<q+p$. These cover all cases, as $(q,r)=1$ so $w\neq q$, so we are done.\end{proof}

Unfortunately, $1-x$ is the only very well-behaved forbidden binomial. Proposition \ref{q-p<w<q+pprop} only covers $4p-8$ residues modulo $pq$, as an easy calculation shows, so this method is of limited utility.

In individual cases, forbidden binomials can be used to demonstrate that a particular ternary cyclotomic polynomial is not flat, without calculating the entire polynomial. One simply identifies the appropriate coefficient that is $\pm2$ and uses equation \eqref{x^kFj} to prove that it is so.

\section{Quaternary Cyclotomic Polynomials}

Let $p<q<r<s$ be odd primes, so $\Phi_{pqrs}(x)$ is a quaternary cyclotomic polynomial. In order to use our developed theory, we then let $n=pqr$, so $\varphi(n)=(p-1)(q-1)(r-1)$. We can then write $\Phi_{pqrs}(x)$ in terms of $\Phi_{pqr}(x)$, but this is only useful when we have already computed $\Phi_{pqr}(x)$. Additionally, we only have tools to investigate the case where $s\equiv\pm1\pmod{pqr}$ at this point. We therefore consider the special case where $s\equiv\pm1\pmod{pqr}$ and $r\equiv\pm1\pmod{pq}$.

In 2010, Kaplan \cite{Kap10} took the smallest flat quaternary cyclotomic polynomial, $\Phi_{3\cdot5\cdot31\cdot929}(x)$, and used Theorem \ref{periodicitythm} to produce an infinite family of flat quaternary cyclotomic polynomials: For any prime $s\equiv929\equiv-1\pmod{3\cdot5\cdot31}$, $A(3\cdot5\cdot31\cdot s)=1$. Of course, Theorem \ref{signedperiodicitythm}, signed periodicity, implies that $A(3\cdot5\cdot31\cdot s)=1$ whenever $s\equiv\pm1\pmod{3\cdot5\cdot31}$, but this could also have been easily verified by checking one prime $s\equiv1\pmod{3\cdot 5\cdot 31}$ and applying Kaplan's (unsigned) periodicity. Such generalizations are therefore not especially interesting.

However, in this special case, notice that $q\equiv-1\pmod p$ and $r\equiv1\pmod{pq}$. The goal of this section is to prove the following massive generalization:

\begin{thm}\label{pqrsflatthm}Let $p<q<r<s$ be primes such that $r\equiv\pm1\pmod{pq}$ and $s\equiv\pm1\pmod{pqr}$. Then $A(pqrs)=1$ if and only if $q\equiv-1\pmod p$.\end{thm}
\begin{proof}As we did with the proof of Broadhurst's Type II, we split this proof up into sections.

\subsection{Preliminary Simplifications}
By Theorem \ref{signedperiodicitythm} (signed periodicity), it suffices to consider $s\equiv1\pmod{pqr}$. We established in Proposition \ref{p=1generalprop} that $F_{pqr,s,j}(x)\equiv x^{-j}\pmod{\Phi_{pqr}(x)}$. As usual, it suffices to consider $F_{pqr,s,j}(x)$ for $0\le j<pqr$, in which case $F_{pqr,s,j}(x)\equiv x^{pqr-j}\pmod{\Phi_{pqr}(x)}$.

We can continue to eliminate some easy cases. If $j>pqr-(p-1)(q-1)(r-1)=n-\varphi(n)$, then $F_{pqr,s,j}(x)=x^{pqr-j}$, which is flat. $F_0(x)=1$ is also flat. As in the proof of Theorem \ref{r=1thm}, for every other value of $j$, we will construct a polynomial $f'(x)$ which is divisible by $\Phi_{pqr}(x)$ and has leading term $x^{pqr-j}$ but no other terms of degree at least $(p-1)(q-1)(r-1)$. Then $F_j(x)=x^{pqr-j}-f'(x)$, so $f'(x)$ is flat if and only if $F_j(x)$ is flat.

\subsubsection{$(p-1)(q+r-1)<j\le qr$}
First we consider the case $(p-1)(q+r-1)<j\le qr$. Here we utilize the fact that $\Phi_{pqr}(x)\mid\Phi_p(x^{qr})=1+x^{qr}+\dotsb+x^{(p-2)qr}+x^{(p-1)qr}$. Let
\[f'(x)=x^{qr-j}\Phi_p(x^{qr})=x^{qr-j}+x^{2qr-j}+\dotsb+x^{(p-1)qr-j}+x^{pqr-j}.\]
This is a polynomial, and divisible by $\Phi_p(x^{qr})$, because $j\le qr$. The other constraint on $j$, $(p-1)(q+r-1)<j$, implies that $(p-1)qr-j<(p-1)(qr-q-r+1)=(p-1)(q-1)(r-1)$, so no term besides $x^{pqr-j}$ has too large a degree, as desired. This is clearly flat.

\subsubsection{Reciprocity}
The remaining cases are $1\le j\le(p-1)(q+r-1)$ and $qr+1\le j\le pqr-(p-1)(q-1)(r-1)$. We next use reciprocity as formulated in Section \ref{reciprocity} to match these two ranges. Recall that we must first represent $-\varphi(n)=ys+z$ with $0\le z<s$. Since $s>n>\varphi(n)$, we have $y=-1$ and $z=s-(p-1)(q-1)(r-1)$. For all $1\le j\le pqr-(p-1)(q-1)(r-1)<z$, equations \eqref{reciprocitygeneral} and \eqref{F=Fextended} imply that
\[V(F_j(x))=V(F_{z-j}(x))=V(F_{s-(p-1)(q-1)(r-1)-j}(x))=V(F_{pqr+1-(p-1)(q-1)(r-1)-j}(x)).\]
Therefore, as $(p-1)(q+r-1)+qr=pqr-(p-1)(q-1)(r-1)$, the sets of coefficients in each remaining case are equal:
\[\{V(F_j(x))\}_{j=1}^{(p-1)(q+r-1)}=\{V(F_{pqr+1-(p-1)(q-1)(r-1)-j}(x))\}_{j=1}^{(p-1)(q+r-1)}=\{V(F_j(x))\}_{j=qr+1}^{pqr-(p-1)(q-1)(r-1)}.\]
It thus suffices to consider $qr+1\le j\le pqr-(p-1)(q-1)(r-1)$.

\subsection{Parameterizing with $a$ and $b$ and the case $b\ge(p-1)(q-1)$}
We now parameterize these values of $j$ with two parameters, $a$ and $b$. The case of $b\ge(p-1)(q-1)$ is easy to characterize, so we will do that first. With $b<(p-1)(q-1)$, we first completely analyze the case of $a=0$, which takes several pages, and then use these results to calculate $F_j(x)$ for $a>0$, completing the proof.

Let $m=pqr-j$, so we are finding $f'(x)$ with leading term $x^m$ for $(p-1)(q-1)(r-1)\le m<(p-1)qr$. Then let $m-(p-1)(q-1)(r-1)=ar+b$ for $0\le b<r$. For bounds on $a$, notice that $ar<(p-1)(q+r-1)=(p-1)(q-1)-r+pr<pr$, so $0\le a<p$.

Large $b\ge(p-1)(q-1)$ is another easy case. This time, we utilize the construction given by Lemma \ref{pq(1+x)lem}. Let $l=a+1<p+1<p+q-1$ and let $f'(x)$ be the polynomial obtained by replacing $x$ with $x^r$ in equation \eqref{pq(1+x)}, then multiplying by $x^{b-(p-1)(q-1)}$, we have
\begin{align*}
f'(x)&=x^{b-(p-1)(q-1)}(1+x^r+\dotsb+x^{r(l-1)})\Phi_{pq}(x^r)\\
&=x^{b-(p-1)(q-1)}(1+x^{rp}+\dotsb+x^{rp(\mu-1)})(1+x^{rq}+\dotsb+x^{rq(\lambda-1)})\\
&\qquad-x^{r(a+1)+b-(p-1)(q-1)}(1+x^{rp}+\dotsb+x^{rp(q-\mu-1)})(1+x^{rq}+\dotsb+x^{rq(p-\lambda-1)})
\end{align*}
where $pq+l=\mu p+\lambda q$ and $\mu,\lambda$ are positive integers as usual. Since $\Phi_{pqr}(x)\mid\Phi_{pq}(x^r)$, $\Phi_{pqr}(x)$ divides $f'(x)$. Now we consider the highest degree terms.
\begin{itemize}
\item The leading term of $f'(x)$ has degree
\begin{align*}&\quad b-(p-1)(q-1)+rp(\mu-1)+rq(\lambda-1)=b-(p-1)(q-1)+r(p\mu+q\lambda-p-q)\\
&=b-(p-1)(q-1)+r(pq+l-p-q)=b-(p-1)(q-1)+r(pq-p-q+a+1)\\
&=(p-1)(q-1)(r-1)+ar+b=m.
\end{align*}
\item The next highest degree positive term of $f'(x)$ has degree
\[m-pr<(p-1)qr-pr=(pq-p-q)r=(p-1)(q-1)r-r<(p-1)(q-1)(r-1),\]
since $r>(p-1)(q-1)$.
\item The highest degree negative term of $f'(x)$ has degree
\begin{align*}
&\quad r(a+1)+b-(p-1)(q-1)+rp(q-\mu-1)+rq(p-\lambda-1)\\
&=b-(p-1)(q-1)+r(a+1+2pq-(p\mu+q\lambda)-p-q)\\
&=b-(p-1)(q-1)+r(a+1+2pq-(pq+l)-p-q)\\
&=b-(p-1)(q-1)+r(pq-p-q)=(p-1)(q-1)(r-1)+b-r\\
&<(p-1)(q-1)(r-1),
\end{align*} since $b<r$.
\end{itemize}
Thus, $f'(x)$ satisfies the desired properties, and is flat by the construction given.

\subsection{$a=0$ and $b<(p-1)(q-1)$}
\subsubsection{A Formula for $f'(x)$}
Now consider $b<(p-1)(q-1)$. We will first let $a=0$, and later will extend that case to other values of $a$. This assumption makes $m=(p-1)(q-1)(r-1)+b<(p-1)(q-1)r$. This time, we have a slightly more complicated formula for $f'(x)$:
\[f'(x)=\Phi_{pqr}(x)\sum_{i=0}^bx^{b-i}[x^i]\Phi_{pq}(x).\]

We must prove that this formula satisfies the usual properties. It is monic with $\deg f'(x)=(p-1)(q-1)(r-1)+b=m$ since $[x^0]\Phi_{pq}(x)=1$. By the reciprocal property on $\Phi_{pq}(x)$,
\[f'(x)=\Phi_{pqr}(x)\sum_{i=0}^bx^{b-i}[x^{(p-1)(q-1)-i}]\Phi_{pq}(x).\]
Recall that $\Phi_{pq}(x)\Phi_{pqr}(x)=\Phi_{pq}(x^r)$. The coefficient of $x^{(p-1)(q-1)(r-1)+c}$ in $f'(x)$ for $0\le c<b$ is then given by
\begin{align*}
[x^{(p-1)(q-1)(r-1)+c}]f'(x)&=\sum_{i=0}^b([x^{(p-1)(q-1)(r-1)+c-i}]\Phi_{pqr}(x))([x^{b-i}]\Phi_{pq}(x))\\
&=\sum_{i=0}^b([x^{i-c}]\Phi_{pqr}(x))([x^{b-i}]\Phi_{pq}(x))=\sum_{i=c}^b([x^{i-c}]\Phi_{pqr}(x))([x^{b-i}]\Phi_{pq}(x))\\
&=\sum_{i=0}^{b-c}([x^i]\Phi_{pqr}(x))([x^{b-c-i}]\Phi_{pq}(x))=[x^{b-c}]\Phi_{pq}(x^r)=0,
\end{align*}
since $0<b-c<(p-1)(q-1)<r$. Therefore, $f'(x)$ satisfies the desired properties and we now investigate its coefficients.

\subsubsection{The $F'_j(x)$}\label{F'j}
For the remainder of the proof of Theorem \ref{pqrsflatthm}, consider all polynomial congruences as modulo $\Phi_{pq}(x)$ and numerical congruences as modulo $pq$, unless otherwise stated. We proceed by defining another series of polynomials $\{F'_j(x)\}_{j=0}^{r-1}$ from $f'(x)$ in the usual way: $F'_j(x)=\sum_{i\ge0}x^i[x^{j+ir}]f'(x)$, so $f'(x)=\sum_{j=0}^{r-1}x^jF'_j(x^r)$. Since $f_{pq,r}(x)=\Phi_{pqr}(x)\mid f'(x)$, $F_{pq,r,j}(x)$ and $F'_j(x)$ are closely related. At this point we must distinguish between $r\equiv1$ and $r\equiv-1$.

\subsubsection{Handling $r\equiv\pm1$}\label{r=-1r=1}
\begin{lem}\label{pqrslem}If $r\equiv1$, then $F'_j(x)$ is the unique polynomial of degree less than $(p-1)(q-1)$ congruent to $F_{pq,r,j}(x)\sum_{i=0}^bx^{b-i}[x^i]\Phi_{pq}(x)$. If $r\equiv-1$, then $F'_j(x)$ is the unique polynomial of degree less than $(p-1)(q-1)$ congruent to $F_{pq,r,j}(x)\sum_{i=0}^bx^{i-b}[x^i]\Phi_{pq}(x)$.\end{lem}
\begin{proof}As we showed above, $\deg f'(x)=m<(p-1)(q-1)r$, so $j+r\deg F'_j(x)<(p-1)(q-1)r$, which makes $\deg F'_j(x)<(p-1)(q-1)$ as desired. To obtain the desired congruences, we need to manipulate some sums, after recalling that the usual extension of $F_j(x)$ to negative $j$ makes $x^jF_j(x^r)=x^{j-r}F_{j-r}(x^r)$:
\begin{align*}
f'(x)&=\left(\sum_{j=0}^{r-1}x^jF_j(x^r)\right)\left(\sum_{i=0}^bx^i[x^{b-i}]\Phi_{pq}(x)\right)=\sum_{i=0}^b\sum_{j=0}^{r-1}x^{j+i}F_j(x^r)[x^{b-i}]\Phi_{pq}(x)\\
&=\sum_{i=0}^b\sum_{j+i=0}^{r-1}x^{j+i}F_j(x^r)[x^{b-i}]\Phi_{pq}(x)=\sum_{i=0}^b\sum_{j=0}^{r-1}x^jF_{j-i}(x^r)[x^{b-i}]\Phi_{pq}(x)\\
F'_j(x)&=\sum_{i=0}^bF_{j-i}(x)[x^{b-i}]\Phi_{pq}(x).
\end{align*}
Now recall the explicit formulas for $F_{pq,r,j}(x)$ when $r\equiv\pm1$ in Propositions \ref{p=1generalprop} and \ref{p=-1generalprop}. If $r\equiv1$, we have $F_j(x)\equiv x^{-j}$. Then
\[F'_j(x)\equiv\sum_{i=0}^bx^{i-j}[x^{b-i}]\Phi_{pq}(x)\equiv x^{-j}\sum_{i=0}^bx^i[x^{b-i}]\Phi_{pq}(x)\equiv F_j(x)\sum_{i=0}^bx^{b-i}[x^i]\Phi_{pq}(x),\]
as desired. If $r\equiv-1$, we have $F_j(x)\equiv-x^{j+(p-1)(q-1)}$. Then
\[F'_j(x)\equiv\sum_{i=0}^b-x^{j-i+(p-1)(q-1)}[x^{b-i}]\Phi_{pq}(x)\equiv-x^{j+(p-1)(q-1)}\sum_{i=0}^bx^{-i}[x^{b-i}]\Phi_{pq}(x)\equiv F_j(x)\sum_{i=0}^bx^{i-b}[x^i]\Phi_{pq}(x),\]
as desired.\end{proof}

Following this distinction, we bring the two cases back together. Since the $F_j(x)$ are periodic, this lemma establishes that the $F'_j(x)$ are also periodic. Therefore,
\[V(f'(x))=\cup_{j=0}^{r-1}V(F'_j(x))=\cup_{j=0}^{pq-1}V(F'_j(x)).\]
Applying Lemma \ref{pqrslem} to this set, and expanding to all $j$ as $x^{pq}\equiv1$,
\[\{F'_j(x)\}_{j=0}^{pq-1}\equiv\begin{cases}\left\{x^j\sum_{i=0}^bx^{b-i}[x^i]\Phi_{pq}(x)\right\}_{j\in\Z}\qquad r\equiv1\\\left\{-x^j\sum_{i=0}^bx^{i-b}[x^i]\Phi_{pq}(x)\right\}_{j\in\Z}\qquad r\equiv-1.\end{cases}\]
Now note that the first of these is congruent to
\begin{align*}
\left\{x^j\sum_{i=0}^bx^{b-i}[x^i]\Phi_{pq}(x)\right\}_{j\in\Z}&=\left\{x^j\sum_{i=0}^bx^{(p-1)(q-1)-i}[x^{(p-1)(q-1)-i}]\Phi_{pq}(x)\right\}_{j\in\Z}\\
&=\left\{x^j\sum_{i=(p-1)(q-1)-b}^{(p-1)(q-1)}x^i[x^i]\Phi_{pq}(x)\right\}_{j\in\Z}\\
&\equiv\left\{-x^j\sum_{i=0}^{(p-1)(q-1)-b-1}x^i[x^i]\Phi_{pq}(x)\right\}_{j\in\Z}=\left\{-x^j\sum_{i=0}^{b'}x^{i-b'}[x^i]\Phi_{pq}(x)\right\}_{j\in\Z},
\end{align*}
where $b'=(p-1)(q-1)-b-1$. Therefore, the coefficients produced when $r\equiv1$ for $b$ are same as the coefficients produced when $r\equiv-1$ for $b'=(p-1)(q-1)-b-1$. As $b$ runs over $[0,(p-1)(q-1)-1]$, $b'$ runs through the same range (in the reverse order), so the coefficients that show up for this particular range for $r\equiv1$ are the same (up to sign) as those that show up for this range when $r\equiv-1$. We can therefore restrict our attention to one of these two sets, say, the second, written as
\[\left\{x^j\sum_{i=0}^bx^{i-b}[x^i]\Phi_{pq}(x)\right\}_{j\in\Z}=\left\{x^j\sum_{i=0}^bx^i[x^i]\Phi_{pq}(x)\right\}_{j\in\Z}\]

\subsubsection{$e_b(x)$ and the $E_j(x)$}\label{Ejsection}
We are hence considering this expression reduced to a polynomial of degree less than $(p-1)(q-1)$ modulo $\Phi_{pq}(x)$. Let
\[e_{pq,b}(x)=\sum_{i=0}^bx^i[x^i]\Phi_{pq}(x);\quad E_{pq,j,b}(x)\equiv x^je_{pq,b}(x),\text{ such that }\deg E_{pq,j,b}(x)<(p-1)(q-1).\]
[Note that the $E_j(x)$ are \emph{not} related to $e(x)$ in the same way as the $F_j(x)$ are related to $f(x)$.] We will show that $E_{pq,j,b}(x)$ is flat for all $(j,b)$ when $q\equiv-1\pmod p$. But first, we show that when $q\not\equiv-1\pmod p$, either $E_{pq,p+1,(p-1)(q-1)-p-1}(x)$ or $E_{pq,q+1,(p-1)(q-1)-q-1}(x)$ is not flat.

\subsubsection{$q\not\equiv-1\pmod p\Rightarrow A(pqrs)>1$}
To do this, we first pick out a couple coefficients out of $\Phi_{pq}(x)$, given by Proposition \ref{Ldiagramprop}:
\[\Phi_{pq}(x)=(1+x^p+\dotsb+x^{p(\mu-1)})(1+x^q+\dotsb+x^{q(\lambda-1)})-x(1+x^p+\dotsb+x^{p(q-\mu-1)})(1+x^q+\dotsb+x^{q(p-\lambda-1)}),\]
where $\mu$ is the inverse of $p$ modulo $q$ and $\lambda$ the inverse of $q$ modulo $p$, so $p\mu+q\lambda=pq+1$. Since $q\ge p+2$,\footnote{This is true for odd $p,q$ like we are assuming in this theorem. In the pseudocyclotomic analog, we will need to exclude the case $q=p+1$ from this statement as well. In fact, this is analogous to the case $q\equiv-1\pmod p$ with the roles of $p$ and $q$ reversed, so those pseudocyclotomic polynomials are in fact flat. See Proposition \ref{pseudopqrsprop}.} $\mu<q-1$ and the geometric series $1+x^p+\dotsb+x^{p(q-\mu-1)}$ appearing above has at least two terms, and $[x^{(p-1)(q-1)-p-1}]\Phi_{pq}(x)=-1$. Similarly, when $q\not\equiv-1\pmod p$, we also have $\lambda<p-1$ and therefore, the geometric series $1+x^q+\dotsb+x^{q(p-\lambda-1)}$ appearing above has at least two terms, and $[x^{(p-1)(q-1)-q-1}]\Phi_{pq}(x)=-1$.

Then $x^{p+1}e_{(p-1)(q-1)-p-1}(x)$ has degree $(p+1)+(p-1)(q-1)-p-1=(p-1)(q-1)$ and leading coefficient $-1$, so
\begin{align*}
E_{pq,p+1,(p-1)(q-1)-p-1}&=\Phi_{pq}(x)+x^{p+1}\sum_{i=0}^{(p-1)(q-1)-p-1}x^i[x^i]\Phi_{pq}(x)\\
&=(1+x^{p+1})\Phi_{pq}(x)-x^{(p-1)(q-1)+p+1}+x^{(p-1)(q-1)+p}-x^{(p-1)(q-1)+1}\\
&=\sum_{i=0}^{(p-1)(q-1)-1}x^i[x^i]((1+x^{p+1})\Phi_{pq}(x)).
\end{align*}
We also have $\deg x^{q+1}e_{(p-1)(q-1)-q-1}(x)=(p-1)(q-1)$ with leading coefficient also $-1$, so
\begin{align*}
E_{pq,q+1,(p-1)(q-1)-q-1}&=\Phi_{pq}(x)+x^{q+1}\sum_{i=0}^{(p-1)(q-1)-p-1}x^i[x^i]\Phi_{pq}(x)\\
&=(1+x^{q+1})\Phi_{pq}(x)-x^{(p-1)(q-1)+q+1}+x^{(p-1)(q-1)+q}-\dotsb-x^{(p-1)(q-1)+1}\\
&=\sum_{i=0}^{(p-1)(q-1)-1}x^i[x^i]((1+x^{q+1})\Phi_{pq}(x)).
\end{align*}
Both $(1+x^{p+1})\Phi_{pq}(x)$ and $(1+x^{q+1})\Phi_{pq}(x)$ are reciprocal, so $V(E_{pq,p+1,(p-1)(q-1)-p-1}(x))=V((1+x^{p+1})\Phi_{pq}(x))$ and $V(E_{pq,q+1,(p-1)(q-1)-q-1}(x))=V((1+x^{q+1})\Phi_{pq}(x))$. It thus suffices to show that either $1+x^{p+1}$ or $1+x^{q+1}$ is a forbidden binomial for all primes $p$ and $q$.

We obtain the dichotomy by considering $\mu$ and $\lambda$. In short, either $2\lambda\le p-1$ or $2\mu\le q-1$. If neither of these hold, $2\mu\ge q+1$ and $2\lambda\ge p+1$, making $2pq+2=2p\mu+2q\lambda\ge2pq+p+q$, a contradiction.\footnote{In the pseudocyclotomic case, we have to consider either $2\lambda=p$ or $2\mu=q$ (but not both, since $(p,q)=1$), which make the contradictions just $2pq+2\ge2pq+q$ and $2pq+2\ge2pq+p$ instead.}

In the case that $2\lambda\le p-1$, we claim that $[x^{1+q(p-\lambda-1)}]((1+x^{p+1})\Phi_{pq}(x))=-2$. Indeed, $-x^{1+q(p-\lambda-1)}$ appears in expression \eqref{pq} by choosing the summands of $1$ and $x^{q(p-\lambda-1)}$ from the second product. Moreover,
\begin{align*}
1+q(p-\lambda-1)-(p+1)&=pq-\lambda q-q-p=p(q-\mu-1)+\mu p-\lambda q-q\\
&=p(q-\mu-1)+(pq+1-\lambda q)-\lambda q-q=1+p(q-\mu-1)+q(p-2\lambda-1),
\end{align*}
and $-x^{1+p(q-\mu-1)+q(p-2\lambda-1)}$ appears in expression \eqref{pq} by choosing the summands of $x^{p(q-\mu-1)}$ and $x^{q(p-2\lambda-1)}$ since $p-1\ge 2\lambda$. Therefore, $[x^{1+q(p-\lambda-1)}](x^{p+1}\Phi_{pq}(x))=-1$, so summing, $[x^{1+q(p-\lambda-1)}]((1+x^{p+1})\Phi_{pq}(x))=-2$ as desired.

The other case is completely analogous. When $2\mu\le q-1$, we claim that $[x^{1+p(q-\mu-1)}]((1+x^{q+1})\Phi_{pq}(x))=-2$. Indeed, $-x^{1+p(q-\mu-1)}$ appears in expression \eqref{pq} by choosing the summands of $1$ and $x^{p(q-\mu-1)}$ from the second product. Moreover,
\begin{align*}
1+p(q-\mu-1)-(q+1)&=pq-\mu p-p-q=q(p-\lambda-1)+\lambda q-\mu p-p\\
&=q(p-\lambda-1)+(pq+1-\mu p)-\mu p-p=1+q(p-\lambda-1)+p(q-2\mu-1),
\end{align*}
and $-x^{1+q(p-\lambda-1)+p(q-2\mu-1)}$ appears in expression \eqref{pq} by choosing the summands of $x^{q(p-\lambda-1)}$ and $x^{p(q-2\mu-1)}$ since $q-1\ge 2\mu$. Therefore, $[x^{1+p(q-\mu-1)}]((1+x^{q+1})\Phi_{pq}(x))=-2$ as desired.

Therefore, we conclude that $A(pqrs)>1$ when $q\not\equiv-1\pmod p$. For the remainder of this proof, we return to consider $q\equiv-1\pmod p$. For now, we are still assuming that $a=0$.

\subsubsection{$[x^b]\Phi_{pq}(x)=1$ and the $j_i$}
We must show that $E_{pq,j,b}(x)$ is flat for all $0\le j<pq$ and $1\le b<(p-1)(q-1)$. We will need to split up the values of $j$ into different intervals, and for this purpose will define the subscripted $j_i$, $i=1,2,3,4,5$.

Fix $b$, and let $j$ run from $0$ to $pq-1$. We only need to consider the cases where $[x^b]\Phi_{pq}(x)\neq0$, because otherwise $e_b(x)=e_{b-1}(x)$. We can therefore distinguish between $[x^b]\Phi_{pq}(x)=1$ and $[x^b]\Phi_{pq}(x)=-1$. Suppose that $[x^b]\Phi_{pq}(x)=1$.

Let $q=kp-1$, so the inverse of $q$ mod $p$ is $\lambda=p-1$ and that of $p$ mod $q$ is $\mu=k$. (Note that in the pseudocyclotomic analog, we might have $k=1$.) Proposition \ref{Ldiagramprop} now states that
\begin{equation}\Phi_{pq}(x)=(1+x^p+\dotsb+x^{p(k-1)})(1+x^q+\dotsb+x^{q(p-2)})-x(1+x^p+\dotsb+x^{p(q-k-1)}).\label{pqq=-1}\end{equation}

If $[x^b]\Phi_{pq}(x)=1$ then $b=cq+dp$ where $0\le c\le p-2$ and $0\le d\le k-1$, but both upper bounds are not simultaneously realized as this would make $b=(p-2)q+(k-1)p=pq-2q+q+1-p=(p-1)(q-1)$.

We now divide the remaining possible values for $j$ into 6 cases by defining cutoffs $0\le j_1\le j_2\le j_3\le j_4\le j_5\le pq$. That is, case $k$ will be between $j_{k-1}$ and $j_k$, with either strict or nonstrict inequalities on either side. The $j_i$ will be defined relative to each other and take several different forms, so we will define them as we go along.

\vspace{.25cm}\noindent{\large\textsc{Case 1. $0\le j<j_1$}}

For our first case, if $j+b<(p-1)(q-1)$, then $E_{pq,j,b}(x)=x^j\sum_{i=0}^bx^i[x^i]\Phi_{pq}(x)$, which is flat. So let $j_1=(p-1)(q-1)-b$, and we are done for $j<j_1$.

\vspace{.25cm}\noindent{\large\textsc{Case 2. $j_1\le j\le j_2$}}

Since $p(k-1)=kp-p<kp-1=q$, we can write
\begin{multline}e_{pq,b}(x)=\sum_{i=0}^bx^i[x^i]\Phi_{pq}(x)=(1+x^p+\dotsb+x^{p(k-1)})(1+x^q+\dotsb+x^{q(c-1)})+\\
x^{qc}(1+x^p+\dotsb+x^{pd})-x(1+x^p+\dotsb+x^{p(kc+d-1)}).\label{erep}\end{multline}

First consider $j=j_1=(p-1)(q-1)-b$. Since the leading coefficient of $e(x)$ is 1,
\[E_{(p-1)(q-1)-b}(x)=x^{(p-1)(q-1)-b}e_{pq,b}(x)-\Phi_{pq}(x).\]
We can generalize this to the values immediately after $j_1$.
Let
\begin{align*}
j_2&=j_1+p-c-2=(p-1)(q-1)-b+p-c-2=pq-p-q+1-(cq+dp)+p-c-2\\
&=(kp-1)(p-c-1)-dp-c-1=p(q-k(c+1)-d)=p(k(p-c-1)-d-1),
\end{align*}
and $j_1\le j\le j_2$. Then we claim that
\begin{equation}E_j(x)=x^je(x)-(1+x+\dotsb+x^{j-j_1})\Phi_{pq}(x).\label{Esmallj}\end{equation}
Clearly this satisfies the congruence requirement for $E_j(x)$, so we examine its highest degree terms. By Lemma \ref{pq(1+x)lem}, since $j-j_1\le j_2-j_1=p-c-2<p+q-1$, the highest degree terms of $(1+x+\dotsb+x^{j-j_1})\Phi_{pq}(x)$ are $x^{j-j_1+(p-1)(q-1)}-x^{(p-1)(q-1)-1}=x^{j+b}-x^{(p-1)(q-1)-1}$. In examining $x^je(x)$, we note that because the coefficients of $\Phi_{pq}(x)$ alternate in sign, so do those of $x^je(x)$. By equation \eqref{erep}, therefore, the highest degree terms of this are $x^j(x^b-x^{1+p(kc+d-1)})=x^{j+b}-x^{j+1+p(kc+d-1)}$. So we must show that $j+1+p(kc+d-1)<(p-1)(q-1)$. Indeed,
\[j+1+p(kc+d-1)\le j_2+1+p(kc+d-1)=p(k(p-1)-2)+1=(p-1)(q-1)-1,\]
as desired. Thus, we have established equation \eqref{Esmallj}.

We will show that this is flat whenever $j<j_2$, and save the case $j=j_2$ for later. To do this, we will use the same technique as the proofs of Theorems \ref{r=2thm} and \ref{BroadhurstII}. Let $u(x)=x^je(x)$ and $v(x)=(1+x+\dotsb+x^{j-j_1})\Phi_{pq}(x)$. Then $E_j(x)=u(x)-v(x)$. Both $u(x)$ and $v(x)$ are flat from equations \eqref{erep} and \eqref{pq(1+x)}. Let $U_+=\{i|[x^i]u(x)=1\}$, $U_-=\{i|[x^i]u(x)=-1\}$. From equation \eqref{erep},
\begin{align*}
U_+&=U_{+,1}\cup U_{+,2};\quad U_{+,1}=j+(p\{0,1,\dotsc,k-1\})\oplus(q\{0,1,\dotsc,c-1\});\\
U_{+,2}&=j+cq+p\{0,1,\dotsc,d\};\quad U_-=j+1+p\{0,1,\dotsc,kc+d-1\}.
\end{align*}
Define $V_+$ and $V_-$ analogously. To utilize equation \eqref{pq(1+x)}, we let $l=j-j_1+1$, so $1\le l<p-c-1<p$. Then we need to find positive integers $\mu$ and $\lambda$ such that $pq+l=\mu p+\lambda q$. Letting $\mu=kl$ and $\lambda=p-l$ makes
\[\mu p+\lambda q=kpl+q(p-l)=ql+l+pq-ql=pq+l,\]
as desired. Therefore,
\begin{multline}(1+x+\dotsb+x^{l-1})\Phi_{pq}(x)=(1+x^p+\dotsb+x^{p(kl-1)})(1+x^q+\dotsb+x^{q(p-l-1)})\\
-x^l(1+x^p+\dotsb+x^{p(q-kl-1)})(1+x^q+\dotsb+x^{q(l-1)}).\label{pql}\end{multline}
As a result,
\begin{align*}
V_+&=(p\{0,1,\dotsc,kl-1\})\oplus(q\{0,1,\dotsc,p-l-1\})\\
V_-&=l+(p\{0,1,\dotsc,q-kl-1\})\oplus(q\{0,1,\dotsc,l-1\}).
\end{align*}
To show that $E_j(x)$ is flat, we must show that $U_-\cap V_+=U_+\cap V_-=\emptyset$. To show $U_-\cap V_+=\emptyset$, suppose otherwise that there exist integers $0\le\alpha\le kc+d-1$, $0\le\beta\le kl-1$ and $0\le\gamma\le p-l-1$ such that $j+1+p\alpha=p\beta+q\gamma$. Taking this modulo $p$, we obtain $\gamma\equiv-j-1\pmod p$.

Since $p\mid j_2$, we can write this as $\gamma\equiv j_2-j-1\pmod p$ and $0\le j_2-j-1<p$, so $\gamma=j_2-j-1$. Plugging back in, $(j+1)(q+1)+p\alpha=p\beta+j_2q$, or
\begin{align*}
p(\beta-\alpha)&=(j+1)(q+1)-j_2q=(j_1+l)(q+1)-j_2q=(p-1)(q-1)-b+(j_1-j_2)q+l(q+1)\\
&=(p-1)(q-1)-cq-dp-(p-c-2)q+kpl=pq-p-q+1-cq-dp-pq+cq+2q+kpl\\
&=-p+q+1-dp+kpl\ge p(-1+k-d+kl)>p(kl-1),
\end{align*}
a contradiction. Therefore, $U_-\cap V_+=\emptyset$, as desired.

To show that $U_+\cap V_-=\emptyset$, consider each modulo $p$ (showing they occupy different rows of the L diagram). Combining $U_{+,1}$ and $U_{+,2}$, we have $U_+\equiv j-\{0,1,\dotsc,c\}\pmod p$ and \begin{align*}
V_-&\equiv l-\{0,1,\dotsc,l-1\}\equiv j-j_1+1-\{0,1,\dotsc,j-j_1\}\equiv j-\{j_1-1,j_1,\dotsc,j-1\}\\
&\subseteq j-\{j_1-1,j_1+1,\dotsc,j_2-2\}=j-\{c+1,c+2,\dotsc,p-2\}\pmod p,
\end{align*}
since $p\mid j_2$ and $j_2-j_1=p-c-2$. This does not intersect $j-\{0,1,\dotsc,c\}$, so $U_+\cap V_-=\emptyset$, as desired. We have shown that $E_j(x)$ is flat for $j_1\le j<j_2$.

\vspace{.25cm}\noindent{\large\textsc{Case 2/3. $j=j_2$}}

We can easily manipulate the case of $j=j_2=(p-1)(q-1)-b+p-c-2$. Here $l=j_2-j_1+1=p-c-1$, and plugging this into equation \eqref{pql}, utilizing $q=kp-1$, tells us that we have
\begin{multline}(1+x+\dotsb+x^{p-c-2})\Phi_{pq}(x)=(1+x^p+\dotsb+x^{p(q-k(c+1))})(1+x^q+\dotsb+x^{qc})\\
-x^{p-c-1}(1+x^p+\dotsb+x^{p(k(c+1)-2)})(1+x^q+\dotsb+x^{q(p-c-2)}).\label{pq(p-c-2)}\end{multline}
We will combine equations \eqref{erep} and \eqref{pq(p-c-2)} into a massive formula, but need to make these expressions more compatible first. To do so, we rearrange the terms of equation \eqref{erep} to obtain the $1+x^q+\dotsb+x^{qc}$ in equation \eqref{pq(p-c-2)}:
\begin{multline}e(x)=(1+x^p+\dotsb+x^{pd})(1+x^q+\dotsb+x^{qc})\\
+(x^{p(d+1)}+x^{p(d+2)}+\dotsb+x^{p(k-1)})(1+x^q+\dotsb+x^{q(c-1)})-x(1+x^p+\dotsb+x^{p(kc+d-1)})\label{erep'}\end{multline}
Next we note that one portion of equation \eqref{pq(p-c-2)} is
\begin{align*}
&x^{p-c-1}(1+x^p+\dotsb+x^{p(k(c+1)-2)})x^{q(p-c-2)}\\
&\qquad=x(x^{pk(p-c-2)}+x^{p(k(p-c-2)+1)}+\dotsb+x^{p(k(p-c-2)+k(c+1)-2)})\\
&\qquad=x(x^{p(q-k(c+2)+1)}+x^{p(q-k(c+2)+2)}+\dotsb+x^{p(q-k-1)}).
\end{align*}
With this, we finally combine equations \eqref{erep'} and \eqref{pq(p-c-2)}:
\begin{align*}
E_{j_2}(x)&=x^{j_2}e(x)-(1+x+\dotsb+x^{p-c-2})\Phi_{pq}(x)\\
&=(x^{p(q-k(c+1)-d)}+x^{p(q-k(c+1)-d+1)}+\dotsb+x^{p(q-k(c+1))})(1+x^q+\dotsb+x^{qc})\\
&\qquad+(x^{p(q-k(c+1)+1)}+x^{p(q-k(c+1)+2)}+\dotsb+x^{p(q-kc-d-1)})(1+x^q+\dotsb+x^{q(c-1)})\\
&\qquad-x(x^{p(q-k(c+1)-d)}+x^{p(q-k(c+1)-d+1)}+\dotsb+x^{p(q-k-1)})\\
&\qquad-(1+x^p+\dotsb+x^{p(q-k(c+1))})(1+x^q+\dotsb+x^{qc})\\
&\qquad+x^{p-c-1}(1+x^p+\dotsb+x^{p(k(c+1)-2)})(1+x^q+\dotsb+x^{q(p-c-3)})\\
&\qquad+x(x^{p(q-k(c+2)+1)}+x^{p(q-k(c+2)+2)}+\dotsb+x^{p(q-k-1)})\\
&=(x^{p(q-k(c+1)+1)}+x^{p(q-k(c+1)+2)}+\dotsb+x^{p(q-kc-d-1)})(1+x^q+\dotsb+x^{q(c-1)})\\
&\qquad-(1+x^p+\dotsb+x^{p(q-k(c+1)-d-1)})(1+x^q+\dotsb+x^{qc})\\
&\qquad+x^{p-c-1}(1+x^p+\dotsb+x^{p(k(c+1)-2)})(1+x^q+\dotsb+x^{q(p-c-3)})\\
&\qquad+x(x^{p(q-k(c+2)+1)}+x^{p(q-k(c+2)+2)}+\dotsb+x^{p(q-k(c+1)-d-1)})\\
&=x(x^{p(q-k(c+2)+1)}+x^{p(q-k(c+2)+2)}+\dotsb+x^{p(q-k(c+1)-d-1)})(x^q+x^{2q}+\dotsb+x^{qc})\\
&\qquad-(1+x^p+\dotsb+x^{p(q-k(c+1)-d-1)})(1+x^q+\dotsb+x^{qc})\\
&\qquad+x^{p-c-1}(1+x^p+\dotsb+x^{p(k(c+1)-2)})(1+x^q+\dotsb+x^{q(p-c-3)})\\
&\qquad+x(x^{p(q-k(c+2)+1)}+x^{p(q-k(c+2)+2)}+\dotsb+x^{p(q-k(c+1)-d-1)})\\
&=x(x^{p(q-k(c+2)+1)}+x^{p(q-k(c+2)+2)}+\dotsb+x^{p(q-k(c+1)-d-1)})(1+x^q+\dotsb+x^{qc})\\
&\qquad-(1+x^p+\dotsb+x^{p(q-k(c+1)-d-1)})(1+x^q+\dotsb+x^{qc})\\
&\qquad+x^{p-c-1}(1+x^p+\dotsb+x^{p(k(c+1)-2)})(1+x^q+\dotsb+x^{q(p-c-3)}).
\end{align*}
When $d=k-1$, the first product is zero, and otherwise, all are nonzero.

First, we must show that this is flat. Fortunately, this merely consists of showing that the first and last products have no common terms. To do so, we simply look at the exponents modulo $p$. The first set is $1-\{0,1,\dotsc,c\}=\{1-c,2-c,\dotsc,1\}$. The last set is $p-c-1-\{0,1,\dotsc,p-c-3\}=\{2,3,\dotsc,p-c-1\}$. These sets do not intersect; thus, $E_{j_2}(x)$ is flat.

\vspace{.25cm}\noindent{\large\textsc{Case 3. $j_2\le j\le j_3$}}

We are also interested in the degree of this polynomial. We calculate the degrees of each of the three products to be
\begin{align*}
1+p(q-k(c+1)-d-1)+qc&=pq-(q+1)(c+1)+qc-dp-p+1\\
&=(p-1)(q-1)-c-dp-1\\
p(q-k(c+1)-d-1)+qc&=pq-(q+1)(c+1)+qc-dp-p\\
&=(p-1)(q-1)-c-dp-2\\
p-c-1+p(k(c+1)-2)+q(p-c-3)&=(q+1)(p-c-1)+kp(c+1)-2p-2q\\
&=(q+1)p-2p-2q=pq-p-2q.
\end{align*}
At this point it is necessary to distinguish between $d=k-1$ and $d<k-1$. The easier case is $d=k-1$, when the first product is zero. Since both upper bounds for $c$ and $d$ are not simultaneously satisfied, we must also have $c<p-2$. The degree of the second product is $(p-1)(q-1)-c-(k-1)p-2=pq-q+1-c-kp-2=pq-2q-c-2>pq-p-2q$, so $\deg E_{j_2}(x)=pq-2q-c-2$. Let
\begin{align*}
j_3&=j_2+(p-1)(q-1)-1-(pq-2q-c-2)\\
&=((p-1)(q-1)-b+p-c-2)+pq-p-q-pq+2q+c+2=(p-1)(q-1)-b+q\\
&=pq-p+1-cq-(k-1)p=pq+1-cq-q-1=q(p-c-1).
\end{align*}
Therefore, for all $j_2\le j\le j_3$, $x^{j-j_2}E_{j_2}(x)$ has degree at most $(p-1)(q-1)-1$, and being congruent to $x^je(x)$, is therefore $E_j(x)$.

When $d<k-1$, the first product is not zero, and has the largest degree: It clearly has a larger degree than the second, and $(p-1)(q-1)-c-dp-1\ge(p-1)(q-1)-(p-2)-(k-2)p-1=pq-q-(q+1)+2p>pq-p-2q$. Therefore, when $d<k-1$, $\deg E_{j_2}(x)=(p-1)(q-1)-c-dp-1$. Let
\begin{align*}
j_3&=j_2+(p-1)(q-1)-1-((p-1)(q-1)-c-dp-1)=((p-1)(q-1)-b+p-c-2)+c+dp\\
&=(p-1)(q-1)-(cq+dp)+(d+1)p-2=q(p-c-1)-1=p(q+d-k)-b,
\end{align*}
and for $j_2\le j\le j_3$, $E_j(x)=x^{j-j_2}E_{j_2}$.

In both cases, since we showed that $E_{j_2}(x)$ is flat, $E_j(x)$ is flat for all $j_2\le j\le j_3$.

\vspace{.25cm}\noindent{\large\textsc{Case 6. $j_5\le j<pq$}}

To this point we have started from small $j$ and built $E_j(x)$ up from $x^je(x)$, working all the way up to $j=j_3$. Now we will take the opposite approach and start from large $j$ and build $E_j(x)$ up from $x^{j-pq}(e(x)-\Phi_{pq}(x))$, working down from $j=pq-1$ to $j=j_3$.

Note that the next nonzero coefficient in $\Phi_{pq}(x)$ after $[x^b]\Phi_{pq}(x)$ is $[x^{b+c+1}]\Phi_{pq}(x)=-1$ because the coefficients of $\Phi_{pq}(x)$ alternate in sign, all negative coefficients of $\Phi_{pq}(x)$ have exponents that are 1 mod $p$, and $b+c=c(q+1)+dp=p(kc+d)$ is the next multiple of $p$ after $b$. Letting $j_5=pq-b-c-1$, this implies that for $j_5\le j<pq$, $x^{j-pq}(e(x)-\Phi_{pq}(x))$ is a polynomial with degree less than $(p-1)(q-1)$, so it must be equal to $E_j(x)$ for those $j$. Since $\Phi_{pq}(x)-e(x)$ simply consists of terms of $\Phi_{pq}(x)$, it is flat, and so is $E_j(x)$ for $j_5\le j<pq$.

\vspace{.25cm}\noindent{\large\textsc{Case 5. $j_4\le j<j_5$ Case}}

Now we consider the remaining $j_3<j<j_5$. As the subscripts suggest, we will make use of one last delimiter, $j_4$. As was the case with $j_3$, we must distinguish between $d=k-1$ and $d<k-1$: Let
\[j_4=\begin{cases}j_5-p+c+2=pq-p-b+1=q(p-c-1)=j_3&\quad d=k-1\\
j_5-p+c+1=pq-p-b=j_3+p(k-d-1)&\quad d<k-1.\end{cases}\]

With this definition, for any $j_4\le j<j_5$, we claim that
\begin{equation}E_j(x)=x^{j-pq}(e(x)-\Phi_{pq}(x))-x^{j-j_5}(1+x+\dotsb+x^{j_5-j-1})\Phi_{pq}(x).\label{Ejlargej}\end{equation}
Again, the congruence is satisfied trivially. We must show that the right side is actually a polynomial, and of the appropriate degree. We start with the degree. Since $j<pq$, $x^{j-pq}(e(x)-\Phi_{pq}(x))$ only has terms of degree less than $(p-1)(q-1)$. The remainder is $x^{j-j_5}(1+x+\dotsb+x^{j_5-j-1})\Phi_{pq}(x)$, which conveniently has degree exactly $(p-1)(q-1)-1$. Therefore, it only remains to verify that this is a polynomial.

Whether $d=k-1$ or $d<k-1$, $j_5-j\le j_5-j_4\le p-c-1<p$, so we can use Lemma \ref{pq(1+x)lem} or equation \eqref{pql} to calculate the second portion of expression \eqref{Ejlargej}. This tells us that the lowest degree terms of $x^{j-j_5}(1+x+\dotsb+x^{j_5-j-1})\Phi_{pq}(x)$ are $x^{j-j_5}(1-x^{j_5-j})=x^{j-j_5}-1$.

Meanwhile, the lowest degree terms of $x^{j-pq}(e(x)-\Phi_{pq}(x))=-x^{j-pq}\sum_{i=b+c+1}^{(p-1)(q-1)}x^i[x^i]\Phi_{pq}(x)$ depend on whether $d=k-1$ or $d<k-1$. If $d=k-1$, they are $x^{j-pq}(x^{b+c+1}-x^{(c+1)q})=x^{j-pq+b+c+1}-x^{j-pq+cq+q}=x^{j-j_5}-x^{j-j_4}$. Since $j\ge j_4$, the $x^{j-j_5}$ cancel, leaving a polynomial, as desired. If $d<k-1$, they are $x^{j-pq}(x^{b+c+1}-x^{b+p})=x^{j-pq+b+c+1}-x^{j-pq+b+p}=x^{j-j_5}-x^{j-j_4}$. Again, $j\ge j_4$ so the $x^{j-j_5}$ cancel, leaving a polynomial, as desired. This calculation justifies the choices of $j_4$.

Therefore, we have proved \eqref{Ejlargej}. Now, analogous to the case $j_1\le j\le j_2$, we will first assume that $j_4<j\le j_5$ and show that the resulting $E_j(x)$ is flat. When $d=k-1$, this will be enough since we already showed that $E_{j_3}(x)=E_{j_4}(x)$ is flat. When $d<k-1$, we will have more work to do.

We begin by rewriting $e(x)-\Phi_{pq}(x)=-\sum_{i=b+c+1}^{(p-1)(q-1)}x^i[x^i]\Phi_{pq}(x)$. Subtracting equation \eqref{pq} from equation \eqref{erep}, we have
\begin{multline}e(x)-\Phi_{pq}(x)=\sum_{i=0}^bx^i[x^i]\Phi_{pq}(x)=-(1+x^p+\dotsb+x^{p(k-1)})(x^{q(c+1)}+x^{q(c+2)}+\dotsb+x^{q(p-2)})\\
-x^{qc}(x^{p(d+1)}+x^{p(d+2)}+\dotsb+x^{p(k-1)})+x(x^{p(kc+d)}+x^{p(kc+d+1)}+\dotsb+x^{p(q-k-1)}).\label{-erep}\end{multline}

Let $l=j_5-j\le j_5-j_4\le p-c-1$. For convenience, we will actually show that $x^lE_j(x)$ is flat. Letting $u(x)=x^{l+j-pq}(e(x)-\Phi_{pq}(x))=x^{-b-c-1}(e(x)-\Phi_{pq}(x))$ and $v(x)=(1+x+\dotsb+x^{l-1})\Phi_{pq}(x)$, we have $x^lE_j(x)=u(x)-v(x)$. Again, both $u(x)$ and $v(x)$ are flat, given explicitly by equations \eqref{-erep} and \eqref{pql}, respectively. Define $U_-,U_+,V_-,V_+$ as before. Then
\begin{align*}
U_+&=-b-c+p\{kc+d,kc+d+1,\dotsc,q-k-1\};\\
U_-&=U_{-,1}\cup U_{-,2};\\
U_{-,1}&=-b-c-1+(p\{0,1,\dotsc,k-1\})\oplus(q\{c+1,c+2,\dotsc,p-2\});\\
U_{-,2}&=-b-c-1+qc+p\{d+1,d+2,\dotsc,k-1\};\\
V_+&=(p\{0,1,\dotsc,kl-1\})\oplus(q\{0,1,\dotsc,p-l-1\});\\
V_-&=l+(p\{0,1,\dotsc,q-kl-1\})\oplus(q\{0,1,\dotsc,l-1\}).
\end{align*}

First we show that $U_+\cap V_-=\emptyset$. To do so, consider each modulo $p$. We have $U_+\equiv-b-c\equiv0\pmod p$ since as we found before, $b+c=p(kc+d)$. On the other hand, $V_-\equiv l-\{0,1,\dotsc,l-1\}=\{1,2,\dotsc,l\}\pmod p$. Since $l\le p-c-1<p$, these sets do not intersect, as desired.

To show that $U_-\cap V_+=\emptyset$, we use a different method from the typical congruences modulo $p$ and $q$. Note that every member of $V_+$ can, by definition, be written as a sum of a nonnegative multiple of $p$ and a nonnegative multiple of $q$. We will show that this is not possible for members of $U_-$, establishing their nonintersection. Write $b+c+1=p(k(c+1)+d)-q$. Then $U_{-,1}=(p\{-k(c+1)-d,-k(c+1)-d+1,\dotsc,-kc-d-1\})\oplus(q\{c+2,c+3,\dotsc,p-1\})$. Similarly, $U_{-,2}=p\{-k(c+1)+1,-k(c+1)+2,\dotsc,-kc-d-1\}+(c+1)q$. Therefore, members of $U_-$ can be written as $\alpha p+\beta q$ where $1\le\beta<p$ and $\alpha<0$. If this could be also written as $\alpha'p+\beta'q$ where both $\alpha'$ and $\beta'$ are nonnegative, then as $(p,q)=1$, $\alpha'\ge\alpha+q$, making $\beta'\le\beta-p<0$, a contradiction. So $U_-\cap V_+=\emptyset$.

Thus, we have shown that $E_j(x)$ is flat for $j_4<j<j_5$. When $d=k-1$, $j_4=j_3$, so we are done. For $d<k-1$, we simply need to look at $E_{j_4}(x)$ a little more closely. Much of the following will parallel the analysis of $E_{j_2}(x)$ above.

\vspace{.25cm}\noindent{\large\textsc{Case 4/5. $j=j_4$}}

Instead of using equation \eqref{-erep}, we regroup the terms with $x^{p(d+1)}+x^{p(d+2)}+\dotsb+x^{p(k-1)}$ together:
\begin{multline}e(x)-\Phi_{pq}(x)=-(1+x^p+\dotsb+x^{dp})(x^{q(c+1)}+x^{q(c+2)}+\dotsb+x^{q(p-2)})\\
-(x^{p(d+1)}+x^{p(d+2)}+\dotsb+x^{p(k-1)})(x^{qc}+x^{q(c+1)}+\dotsb+x^{q(p-2)})\\+x(x^{p(kc+d)}+x^{p(kc+d+1)}+\dotsb+x^{p(q-k-1)}).\label{-erep'}\end{multline}
Note that in $E_{j_4}(x)$, this is multiplied by $x^{j_4-pq}=x^{-p-b}$. To keep the exponents nonnegative and the expressions simple, we will instead only multiply by $x^{j_5-pq}=x^{-b-c-1}$. That is, we will evaluate
\[x^{p-c-1}E_{j_4}(x)=x^{-b-c-1}(e(x)-\Phi_{pq}(x))-(1+x+\dotsb+x^{p-c-2})\Phi_{pq}(x).\]

Therefore, we must multiply $x^{-b-c-1}$ by each of the three products in expression \eqref{-erep'}. Recall that $b+c=p(kc+d)$. For the first product, we write $b+c+1=cq+dp+(c+1)(kp-q)=p(k(c+1)+d)-q=p(-q+k(c+1)+d)+q(p-1)$, so $x^{-b-c-1}=x^{p(q-k(c+1)-d)}x^{q(1-p)}$. Multiplying by this makes the first product into
\begin{equation}-(x^{p(q-k(c+1)-d)}+x^{p(q-k(c+1)-d+1)}+\dotsb+x^{p(q-k(c+1))})(x^{q(-p+c+2)}+x^{q(-p+c+3)}+\dotsb+x^{-q}).\label{Ejmediumj21}\end{equation}
For the second product, we write $b+c+1=cq+dp+c+1=cq+(d+1)p-(p-c-1)$, so $x^{-b-c-1}=x^{-cq}x^{-p(d+1)}x^{p-c-1}$. Then it becomes
\begin{equation}-x^{p-c-1}(1+x^p+\dotsb+x^{p(k-d-2)})(1+x^q+\dotsb+x^{q(p-c-2)}).\label{Ejmediumj22}\end{equation}
For the last product, we write $b+c+1=c(q+1)+dp+1=p(kc+d)+1$, so $x^{-b-c-1}=x^{-1}x^{-p(kc+d)}$. The last product is therefore now
\begin{equation}1+x^p+\dotsb+x^{p(q-k(c+1)-d-1)}.\label{Ejmediumj23}\end{equation}

We have actually already calculated $(1+x+\dotsb+x^{p-c-2})\Phi_{pq}(x)$ in equation \eqref{pq(p-c-2)}:
\begin{multline*}(1+x+\dotsb+x^{p-c-2})\Phi_{pq}(x)=(1+x^p+\dotsb+x^{p(q-k(c+1))})(1+x^q+\dotsb+x^{qc})\\
-x^{p-c-1}(1+x^p+\dotsb+x^{p(k(c+1)-2)})(1+x^q+\dotsb+x^{q(p-c-2)}).\end{multline*}
Now we combine this with expressions \eqref{Ejmediumj21}, \eqref{Ejmediumj22}, and \eqref{Ejmediumj23}:
\begin{align*}
x^{p-c-2}E_{j_4}(x)&=-(x^{p(q-k(c+1)-d)}+\dotsb+x^{p(q-k(c+1))})(x^{q(-p+c+2)}+x^{q(-p+c+3)}+\dotsb+x^{-q})\\
&\qquad-x^{p-c-1}(1+x^p+\dotsb+x^{p(k-d-2)})(1+x^q+\dotsb+x^{q(p-c-2)})\\
&\qquad+(1+x^p+\dotsb+x^{p(q-k(c+1)-d-1)})\\
&\qquad-(1+x^p+\dotsb+x^{p(q-k(c+1))})(1+x^q+\dotsb+x^{qc})\\
&\qquad+x^{p-c-1}(1+x^p+\dotsb+x^{p(k(c+1)-2)})(1+x^q+\dotsb+x^{q(p-c-2)})\\
&=-(x^{p(q-k(c+1)-d)}+\dotsb+x^{p(q-k(c+1))})(x^{q(-p+c+2)}+x^{q(-p+c+3)}+\dotsb+x^{-q})\\
&\qquad+x^{p-c-1}(x^{p(k-d-1)}+x^{p(k-d)}+\dotsb+x^{p(k(c+1)-2)})(1+x^q+\dotsb+x^{q(p-c-2)})\\
&\qquad-(x^{p(q-k(c+1)-d)}+x^{p(q-k(c+1)-d+1)}+\dotsb+x^{p(q-k(c+1))})\\
&\qquad-(1+x^p+\dotsb+x^{p(q-k(c+1))})(x^q+x^{2q}+\dotsb+x^{qc})\\
&=-(x^{p(q-k(c+1)-d)}+\dotsb+x^{p(q-k(c+1))})(x^{q(-p+c+2)}+x^{q(-p+c+3)}+\dotsb+1)\\
&\qquad+x^{p-c-1}(x^{p(k-d-1)}+x^{p(k-d)}+\dotsb+x^{p(k(c+1)-2)})(1+x^q+\dotsb+x^{q(p-c-2)})\\
&\qquad-(1+x^p+\dotsb+x^{p(q-k(c+1))})(x^q+x^{2q}+\dotsb+x^{qc}).
\end{align*}

Let us show that this is indeed flat. We must simply show that the exponent sets of the first and last products are disjoint. Take each modulo $p$. The first is $-\{-p+c+2,-p+c+3,\dotsc,0\}\equiv\{0,1,\dotsc,p-c-2\}\pmod p$. The last is $-\{1,2,\dotsc,c\}\equiv\{p-c,p-c+1,\dotsc,p-1\}\pmod p$. Since these are indeed disjoint, $E_{j_4}(x)$ is flat.

\vspace{.25cm}\noindent{\large\textsc{Case 4. $j_3<j\le j_4$}}

Now we want to know the smallest degree of any of these terms. For the three remaining products, these degrees are
\begin{align*}
p(q-k(c+1)-d)+q(-p+c+2)&=(c+2)q-(c+1)(q+1)-dp=q-dp-c-1,\\
p-c-1+p(k-d-1)&=q-dp-c,
\end{align*}
and $q$. Clearly, the smallest of these is $q-dp-c-1$.

Since we are looking at $x^{p-c-1}E_{j_4}(x)$, the smallest degree in $E_{j_4}(x)$ proper is $q-(d+1)p=p(k-d-1)-1$. Recall that $j_3=j_4-p(k-d-1)$. Therefore, for any $j_3<j\le j_4$, $x^{j-j_4}E_{j_4}(x)$ is a polynomial, and with degree at most $\deg E_{j_4}(x)<(p-1)(q-1)$, and clearly satisfies the congruence, making it $E_j(x)$. Once again, since we showed earlier that $E_{j_4}(x)$ is flat, so must be $E_j(x)$ for all $j_3<j\le j_4$. This completes the characterization of $E_j(x)$, and we have shown that all the $E_j(x)$ are flat, as desired.

To summarize, we have characterized $E_j(x)$ and proven that it is flat in all cases both when $d=k-1$ and when $d<k-1$. When $d=k-1$,
\begin{equation}E_j(x)=\begin{cases}
x^je(x)&\qquad 0\le j<j_1\\
x^je(x)-(1+x+\dotsb+x^{j-j_1})\Phi_{pq}(x)&\qquad j_1\le j\le j_2\\
x^je(x)-x^{j-j_2}(1+x+\dotsb+x^{j_2-j_1})\Phi_{pq}(x)&\qquad j_2\le j\le j_3\\
x^{j-pq}(e(x)-\Phi_{pq}(x))-x^{j-j_5}(1+x+\dotsb+x^{j_5-j-1})\Phi_{pq}(x)&\qquad j_3\le j<j_5\\
x^{j-pq}(e(x)-\Phi_{pq}(x))&\qquad j_5\le j<pq.
\end{cases}\label{Ejalld=k-1}\end{equation}
When $d<k-1$,
\begin{equation}E_j(x)=\begin{cases}
x^je(x)&\qquad 0\le j<j_1\\
x^je(x)-(1+x+\dotsb+x^{j-j_1})\Phi_{pq}(x)&\qquad j_1\le j\le j_2\\
x^je(x)-x^{j-j_2}(1+x+\dotsb+x^{j_2-j_1})\Phi_{pq}(x)&\qquad j_2\le j\le j_3\\
x^{j-pq}(e(x)-\Phi_{pq}(x))-x^{j-j_5}(1+x+\dotsb+x^{j_5-j_4-1})\Phi_{pq}(x)&\qquad j_3<j\le j_4\\
x^{j-pq}(e(x)-\Phi_{pq}(x))-x^{j-j_5}(1+x+\dotsb+x^{j_5-j-1})\Phi_{pq}(x)&\qquad j_4\le j<j_5\\
x^{j-pq}(e(x)-\Phi_{pq}(x))&\qquad j_5\le j<pq.
\end{cases}\label{Ejalld<k-1}\end{equation}

Therefore, we are done with the case $[x^b]\Phi_{pq}(x)=1$.

\subsubsection{$[x^b]\Phi_{pq}(x)=-1$}
Fortunately, we can avoid repeating similar analysis when $[x^b]\Phi_{pq}(x)=-1$ with a trick. Recall that $\Phi_{pq}(x)$ is reciprocal, so $\Phi_{pq}(x)=x^{(p-1)(q-1)}\Phi_{pq}(x^{-1})$.

Let $[x^b]\Phi_{pq}(x)=-1$, so $e_b(x)$ has leading term $-x^b$. The nonzero coefficients of $\Phi_{pq}(x)$ alternate in sign, so let the next term after $-x^b$ in $\Phi_{pq}(x)$ be $x^{(p-1)(q-1)-b'}$. Then by the reciprocal property, $[x^{b'}]\Phi_{pq}(x)=1$, and $x^{(p-1)(q-1)}e_{b'}(x^{-1})=\Phi_{pq}(x)-e_b(x)$. Fix $j\in\Z$, and let $j'=-j-1$. Then $E_{j,b}(x)\equiv-x^{j+(p-1)(q-1)}e_{b'}(x^{-1})\equiv-x^{(p-1)(q-1)-j'-1}e_{b'}(x^{-1})$. Therefore, there exists some Laurent polynomial $t(x)$ such that $E_{j,b}(x)=t(x)\Phi_{pq}(x)-x^{(p-1)(q-1)-j'-1}e_{b'}(x^{-1})$. Replacing $x$ with $x^{-1}$, as $\Phi_{pq}(x)$ is reciprocal,
\[E_{j,b}(x^{-1})=x^{-(p-1)(q-1)}t(x^{-1})\Phi_{pq}(x)-x^{j'+1-(p-1)(q-1)}e_{b'}(x)\equiv-x^{1-(p-1)(q-1)}E_{j',b'}(x).\]
Since $E_{j',b'}(x)$ is a polynomial of degree at most $(p-1)(q-1)-1$, we know $-x^{1-(p-1)(q-1)}E_{j',b'}(x)$ is a polynomial in $x^{-1}$ with degree at most $(p-1)(q-1)-1$. Since this also describes $E_{j,b}(x^{-1})$, and the two are congruent modulo $\Phi_{pq}(x)$, they must be equal. Replacing $x$ with $x^{-1}$ again, $E_{j,b}(x)=x^{(p-1)(q-1)-1}E_{j',b'}(x^{-1})$. Since $[x^{b'}]\Phi_{pq}(x)=1$, $E_{j',b'}(x)$ is flat, so $E_{j,b}(x)$ is flat as well.

We have proven that $f'(x)$ is flat whenever $q\equiv-1\pmod p$ and $a=0$.

\subsection{$a>0$}
Since we will be comparing different $f'(x)$, for this discussion, let $f'_{a,b}(x)$ be the multiple of $\Phi_{pqr}(x)$ with leading term $x^{(p-1)(q-1)(r-1)+ar+b}$ and no other terms with degrees at least $(p-1)(q-1)(r-1)$. We had already shown that $f'_{a,b}(x)$ is flat when $b\ge(p-1)(q-1)$. All of the above analysis, starting with Lemma \ref{pqrslem}, showed that $f'_{0,b}(x)$ is flat for $0\le b<(p-1)(q-1)$.

Recall that we wrote $f'_{0,b}(x)=\sum_{j=0}^{r-1}x^jF_j'(x^r)$. Lemma \ref{pqrslem} gave us a method for finding the $F'_{j,b}(x)$ that we then transformed into the $E_{j,b}(x)$, which we showed were flat.

Suppose $0<a<p$ and $b\le(p-1)(q-1)$. We claim that
\begin{equation}f'_{a,b}(x)=f'_{0,b}(x)+x^{r-(p-1)(q-1)+b}(1+x^r+\dotsb+x^{r(a-1)})\Phi_{pq}(x^r).\label{f'_{a,b}}\end{equation}
and that this is flat. Since $\Phi_{pqr}(x)\mid\Phi_{pq}(x^r)$ and $\Phi_{pqr}(x)\mid f'_{0,b}(x)$, $\Phi_{pqr}(x)$ divides the right side. Moreover, the leading term of $x^{r-(p-1)(q-1)+b}(1+x^r+\dotsb+x^{r(a-1)})\Phi_{pq}(x^r)$ is $x^{r-(p-1)(q-1)+b+r(a-1)+(p-1)(q-1)r}=x^{(p-1)(q-1)(r-1)+ar+b}$ as desired.

The second-largest degree term is $-x^{r-(p-1)(q-1)+b+(p-1)(q-1)r-r}=-x^{(p-1)(q-1)(r-1)+b}$. This cancels the $x^{(p-1)(q-1)(r-1)+b}$ from $f'_{0,b}(x)$. The third-largest degree term has degree at least $r$ less, and $b<r$, so all other terms, including the rest of $f'_{0,b}(x)$, have degrees less than $(p-1)(q-1)(r-1)$.

Now we show that this is flat. Note that every exponent added is congruent to $r-(p-1)(q-1)+b$ modulo $r$. Translating to the language of the $F'_j(x)$, we must show that when $j=r-(p-1)(q-1)+b$, $F'_j(x)+(1+x+\dotsb+x^{a-1})\Phi_{pq}(x)$ is flat.

First consider the case $r\equiv1$. Then Lemma \ref{pqrslem} states that $F'_j(x)\equiv F_j(x)\sum_{i=0}^bx^{b-i}[x^i]\Phi_{pq}(x)$. Moreover, Proposition \ref{p=1generalprop} states that $F_j(x)\equiv x^{r-(r-(p-1)(q-1)+b)-1}\equiv x^{(p-1)(q-1)-b-1}$. Therefore,
\begin{align*}
F'_j(x)&\equiv x^{(p-1)(q-1)-b-1}\sum_{i=0}^bx^{b-i}[x^i]\Phi_{pq}(x)\\
&\equiv x^{-1}\sum_{i=0}^bx^{(p-1)(q-1)-i}[x^{(p-1)(q-1)-i}]\Phi_{pq}(x)\\
&\equiv x^{-1}\left(\Phi_{pq}(x)-\sum_{i=0}^{(p-1)(q-1)-b-1}x^i[x^i]\Phi_{pq}(x)\right)\\
&= x^{-1}(\Phi_{pq}(x)-e_{b'}(x))\text{, where }b'=(p-1)(q-1)-b-1.\\
F'_j(x)+(1+x+\dotsb+x^{a-1})\Phi_{pq}(x)&=x^{-1}((1+x+\dotsb+x^a)\Phi_{pq}(x)-e_{b'}(x)).
\end{align*}
Before showing that this is flat, we consider $r\equiv-1$. We have Lemma \ref{pqrslem}, which states that $F'_j(x)\equiv F_j(x)\sum_{i=0}^bx^{i-b}[x^i]\Phi_{pq}(x)$. Moreover, Proposition \ref{p=-1generalprop} states that $F_j(x)\equiv-x^{(p-1)(q-1)+(r-(p-1)(q-1)+b)}\equiv-x^{r+b}\equiv-x^{b-1}$. Therefore,
\begin{align*}
F'_j(x)&\equiv-x^{b-1}\sum_{i=0}^bx^{i-b}[x^i]\Phi_{pq}(x)\equiv-x^{-1}\sum_{i=0}^bx^i[x^i]\Phi_{pq}(x)\\
&\equiv x^{-1}(\Phi_{pq}(x)-e_b(x))\\
F'_j(x)&=x^{-1}(\Phi_{pq}(x)-e_b(x))\\
F'_j(x)+(1+x+\dotsb+x^{a-1})\Phi_{pq}(x)&=x^{-1}((1+x+\dotsb+x^a)\Phi_{pq}(x)-e_b(x)).
\end{align*}
Notice that this is the same form as above, except with $b'$ replaced by $b$. Since $0\le b\le (p-1)(q-1)-1$, $0\le b'\le(p-1)(q-1)-1$, it suffices to prove that $(1+x+\dotsb+x^a)\Phi_{pq}(x)-e_b(x)$ is flat.

Let $u(x)=e_b(x)$ and $v(x)=(1+x+\dotsb+x^a)\Phi_{pq}(x)$. Define $U_-,U_+,V_-,V_+$ as usual, and then use Lemma \ref{pq(1+x)lem} and equation \eqref{pql} to get
\begin{align*}
U_+&\subset(p\{0,1,\dotsc,k-1\})\oplus(q\{0,1,\dotsc,p-2\})\\
U_-&\subseteq1+(p\{0,1,\dotsc,q-k-1\})\\
V_+&=(p\{0,1,\dotsc,k(a+1)-1\})\oplus(q\{0,1,\dotsc,p-a-2\})\\
V_-&=a+1+(p\{0,1,\dotsc,q-k(a+1)-1\})\oplus(q\{0,1,\dotsc,a\}).
\end{align*}
To show that $U_-\cap V_+=V_-\cap U_+=\emptyset$, we will use the trick of considering which elements can be written as a nonnegative linear combination of $p$ and $q$. Clearly, $U_+$ and $V_+$ consist entirely of such elements. However, consider an element of $U_-$, $1+p\alpha$ where $0\le\alpha\le q-k-1$. Then $1+p\alpha=p(k+\alpha)-q$, and $0\le k+\alpha<q$. Since the coefficient on $q$ is negative and that on $p$ is between $0$ and $q-1$, this cannot be written as a nonnegative linear combination of $p$ and $q$. Therefore, $U_-\cap V_+=\emptyset$.

Consider an element of $V_-$, $a+1+p\alpha+q\beta$ where $0\le\alpha\le q-k(a+1)-1$ and $0\le\beta\le a$. Then $a+1+p\alpha+q\beta=p(\alpha+k(a+1))-q(a+1-\beta)$, and $0\le\alpha+k(a+1)<q$ while $a+1-\beta<0$. By the same logic, this cannot be written as a nonnegative linear combination of $p$ and $q$. Therefore, $V_-\cap U_+=\emptyset$, as desired.

This shows that $f'_{a,b}(x)$ is flat for $a>0$ and $b\le(p-1)(q-1)$. This is the last case, so finally, this proof is done.\end{proof}

We have established a very general family of flat quaternary cyclotomic polynomials. In fact, all flat quaternary cyclotomic polynomials, as calculated by Arnold and Monagan \cite{AM}, with $pqrs<5\times10^6$, are of this form. Additionally, we have the following pseudocyclotomic analog, with exactly the same proof.

\begin{prop}\label{pseudopqrsprop}Let $p,q,r,s$ be pairwise relatively prime positive integers greater than 1 such that $r\equiv\pm1\pmod{pq}$ and $s\equiv\pm1\pmod{pqr}$. Then if $q\equiv-1\pmod p$ or $p\equiv-1\pmod q$, $\tilde\Phi_{p,q,r,s}(x)$ is flat but if $q\not\equiv-1\pmod p$ and $p\not\equiv-1\pmod q$, $\tilde\Phi_{p,q,r,s}(x)$ is not flat.\end{prop}

\subsection{Open Questions about Quaternary Cyclotomic Polynomials}

First, I conjecture that these are the only possible cyclotomic polynomials.

\begin{conj}\label{pqrsallflat}If $p<q<r<s$ are odd primes such that $\Phi_{pqrs}(x)$ is flat, then $q\equiv-1\pmod p$, $r\equiv\pm1\pmod{pq}$, and $s\equiv\pm1\pmod{pqr}$.\end{conj}

Additionally, the same numerical data \cite{AM} suggest the following conjecture, which should follow from reductions similar to the proof above.
\begin{conj}\label{pqrs2}If $p<q<r<s$ are primes such that $s\equiv\pm1\pmod{pqr}$, $r\equiv\pm1\pmod{pq}$, but $q\not\equiv-1\pmod p$, then $A(pqrs)=2$.\end{conj}

\section{Quinary and Higher Order Cyclotomic Polynomials}
If $2<p<q<r<s<t$ are primes, then $\Phi_{pqrst}(x)$ is a \emph{quinary cyclotomic polynomial}. Based on scant numerical data and the difficulty of proving Theorem \ref{pqrsflatthm}, we make the following conjecture.
\begin{conj}\label{pqrstnotflat}No quinary cyclotomic polynomial is flat.\end{conj}
Indeed, all known quinary cyclotomic polynomials are not flat, up to $pqrst<10^8$, \cite{AM}. Additionally, we can show that what appear to be the most likely candidates for flat quinary cyclotomic polynomials are not flat. The proof borrows many ideas from the proof of Theorem \ref{pqrsflatthm}.

\begin{thm}\label{pqrstnotflatthm}Let $2<p<q<r<s<t$ be primes such that $t\equiv\pm1\pmod{pqrs}$, $s\equiv\pm1\pmod{pqr}$, and $r\equiv\pm1\pmod{pq}$. Then $A(pqrst)>1$.\end{thm}
\begin{proof}After some preliminary simplifications, we will define polynomial families $e_{pqr,b}(x)$ and $E_{pqr,j,b}(x)$ which capture the coefficients of $\Phi_{pqrst}(x)$, and then compute a particular $E_{pqr,j,b}(x)$ in terms of $\Phi_{pqr}(x)$. To show that this is not flat, we then split it up as we have done with $\Phi_{pqr}(x)$ according to the residue of the exponents modulo $pq$, producing the $F'_j(x)$. We then show that a particular $F'_j(x)$ is not flat, so neither is the particular $E_{pqr,j,b}(x)$ or $\Phi_{pqrst}(x)$, completing the proof.

First, it suffices to consider $t\equiv1\pmod{pqrs}$ and by Proposition \ref{p=1generalprop}, $F_{pqrs,t,j}(x)\equiv x^{-j}\pmod{\Phi_{pqrs}(x)}$. Then suppose that $q\not\equiv-1\pmod p$. By Theorem \ref{pqrsflatthm}, $\Phi_{pqrs}(x)$ is not flat. Then $F_{pqrs,t,1}(x)=x^{-1}(1-\Phi_{pqrs}(x))$ since this is a polynomial of degree $(p-1)(q-1)(r-1)(s-1)-1$, and this is not flat. Therefore, since $t>1$, $\Phi_{pqrst}(x)$ is not flat.

Now assume that $q\equiv-1\pmod p$. Recall the definition in Section \ref{Ejsection} of the polynomials $e_{pq,b}(x)$ and $E_{pq,j,b}(x)$. We define the analogous
\[e_{pqr,b}(x)=\sum_{i=0}^bx^i[x^i]\Phi_{pqr}(x)\text{ and }
E_{pqr,j,b}(x)\equiv x^je_{pqr,b}(x)\pmod{\Phi_{pqr}(x)}\]
such that $\deg E_{pqr,j,b}(x)<(p-1)(q-1)(r-1)$. By the same logic as with $E_{pq,j,b}(x)$, since $s\equiv\pm1\pmod{pqr}$ and $t\equiv1\pmod{pqrs}$, the height of $F_{pqrs,t,(p-1)(q-1)(r-1)(s-1)+b}(x)$ is the maximum height of $E_{pqr,j,b}(x)$ for all $0\le j<pqr$. Therefore, it suffices to show that $E_{pqr,j,b}(x)$ is not flat for some $0\le j<pqr$ and $0\le b<(p-1)(q-1)(r-1)$.

Let $b=q$. Again recall equation \eqref{npreciprocal}, which in this case implies that
\begin{align*}
e_{pqr,q}(x)&=\sum_{i=0}^qx^i[x^i]\Phi_{pqr}(x)=\sum_{i=0}^qx^i[x^i](\Psi_{pq}(x)\Phi_{pq}(x^r)(1+x^{pq}+x^{2pq}+\dotsb))\\
&=\sum_{i=0}^qx^i[x^i](1+x+\dotsb+x^{p-1}-x^q-x^{q+1}-\dotsb-x^{p+q-1})=1+x+\dotsb+x^{p-1}-x^q.
\end{align*}

Then let $j=pqr-p-q-1$. We claim that
\begin{equation}E_{pqr,pqr-p-q-1,q}(x)=x^{-p-q-1}(1+x+\dotsb+x^{p-1}-x^q-(1+x^{q+1}-x^{q+p})\Phi_{pqr}(x)).\label{pqr-p-q-1}\end{equation}
Indeed, this has degree less than $(p-1)(q-1)(r-1)$ and is congruent to $x^je(x)$, so it remains to show it is a polynomial. From equation \eqref{npreciprocal}, the smallest degree terms of $(1+x^{q+1}-x^{q+p})\Phi_{pqr}(x)$ are
\begin{align*}&\qquad(1+x+\dotsb+x^{p-1})(1-x^q)(1+x^{q+1}-x^{q+p})\\
&=(1+x+\dotsb+x^{p-1})(1-x^q+x^{q+1}-x^{q+p}-x^{2q+1}+x^{2q+p})\\
&=(1+x+\dotsb+x^{p-1})-x^q+x^{q+p}-(1+x+\dotsb+x^{p-1})(x^{q+p}+x^{2q+1}-x^{2q+p})\\
&=e(x)-(x^{q+p+1}+x^{q+p+2}+\dotsb+x^{q+2p-1})-(1+x+\dotsb+x^{p-1})(x^{2q+1}-x^{2q+p}).
\end{align*}

This shows that all terms of $e(x)-(1+x^{q+1}-x^{q+p})\Phi_{pqr}(x)$ have degree at least $p+q+1$, as desired. Thus expression \eqref{pqr-p-q-1} is a polynomial, and therefore is satisfied.

Now we must show that expression \eqref{pqr-p-q-1} is not flat, or equivalently that $(1+x^{q+1}-x^{q+p})\Phi_{pqr}(x)$ is not flat. Let $f'(x)=(1+x^{q+1}-x^{q+p})\Phi_{pqr}(x)$. As we did in Section \ref{F'j}, for $0\le j<r-1$, define $F'_j(x)=\sum_{i\ge0}x^i[x^{jr+i}]f'(x)$. Then as in Section \ref{r=-1r=1}, we split up $r\equiv1\pmod{pq}$ and $r\equiv-1\pmod{pq}$ and introduce a lemma to give us congruences on the $F'_j(x)$.

\begin{lem}\label{pqrstlem}When $r\equiv1\pmod{pq}$, $F'_j(x)\equiv x^{-j}(1+x^{q+1}-x^{q+p})\pmod{\Phi_{pq}(x)}$. When $r\equiv-1\pmod{pq}$, $F'_j(x)\equiv x^{j+1-2(p+q)}(1-x^{p-1}-x^{q+p})\pmod{\Phi_{pq}(x)}$.\end{lem}
\begin{proof}This proof is completely analogous to the proof of Lemma \ref{pqrslem}. With the extension $F_j(x)=xF_{j+r}(x)$, so $x^jF_j(x)=x^{j+r}F_{j+r}(x^r)$,
\begin{align*}
f'(x)&=(1+x^{q+1}-x^{q+p})\sum_{j=0}^{r-1}x^jF_j(x^r)\\
&=\sum_{j=0}^{r-1}x^jF_j(x^r)+\sum_{j=0}^{r-1}x^{j+q+1}F_j(x^r)-\sum_{j=0}^{r-1}x^{j+q+p}F_j(x^r)\\
&=\sum_{j=0}^{r-1}x^jF_j(x^r)+\sum_{j=-q-1}^{r-q-2}x^{j+q+1}F_j(x^r)-\sum_{j=-q-p}^{r-q-p-1}x^{j+p+q}F_j(x^r)\\
&=\sum_{j=0}^{r-1}x^j(F_j(x^r)+F_{j-q-1}(x^r)-F_{j-q-p}(x^r))\\
F'_j(x)&=F_j(x)+F_{j-q-1}(x)-F_{j-q-p}(x).
\end{align*}

Now, when $r\equiv1\pmod{pq}$, by Proposition \ref{p=1generalprop}, $F_j(x)\equiv-x^j\pmod{\Phi_{pq}(x)}$. Therefore,
\[F'_j(x)\equiv x^{-j}+x^{-j+q+1}-x^{-j+q+p}\equiv x^{-j}(1+x^{q+1}-x^{q+p})\pmod{\Phi_{pq}(x)},\]
as desired. When $r\equiv-1\pmod{pq}$, by Proposition \ref{p=-1generalprop}, $F_j(x)\equiv x^{j+(p-1)(q-1)}\pmod{\Phi_{pq}(x)}$. Then
\begin{align*}
F'_j(x)&\equiv -x^{j+(p-1)(q-1)}-x^{j-q-1+(p-1)(q-1)}+x^{j-q-p+(p-1)(q-1)}\\
&\equiv x^{j+1-2(p+q)}(1-x^{p-1}-x^{q+p})\pmod{\Phi_{pq}(x)},
\end{align*}
as desired.\end{proof}

As $2<p<q$, we have $p\ge3$ and $q\ge5$. The degree limitations on the $F'_j(x)$ are complicated by the case when $p=3$ and $q=5$. We have $j+r\deg F'_j(x)\le (p-1)(q-1)(r-1)+p+q=(p-1)(q-1)r-(pq-2q-2p+1)$, and $pq-2q-2p+1=(p-2)(q-2)-3\ge0$. Therefore, when $(p,q)=(3,5)$, $\deg F'_j(x)<(p-1)(q-1)$ when $0<j<r$ and $\deg F'_0(x)=(p-1)(q-1)$. Otherwise, $\deg F'_j(x)<(p-1)(q-1)$ for all $0\le j<r$.

Let $q=kp-1$, where $k>1$ as $p<q$. We recall equation \eqref{pqq=-1} for $\Phi_{pq}(x)$:
\[\Phi_{pq}(x)=(1+x^p+\dotsb+x^{p(k-1)})(1+x^q+\dotsb+x^{q(p-2)})-x(1+x^p+\dotsb+x^{p(q-k-1)}).\]

When $r\equiv1\pmod{pq}$, we claim that $F'_1(x)$ is not flat. Indeed, $F'_1(x)\equiv x^{-1}(1+x^{q+1}-x^{q+p})\pmod{\Phi_{pq}(x)}$ so $F'_1(x)=x^{-1}(1+x^{q+1}-x^{q+p}-\Phi_{pq}(x))$ since this is a polynomial of degree less than $(p-1)(q-1)$ (as $(p-1)(q-1)-p-q+1=(p-2)(q-2)-2>0$). Now $2<p<q$ makes $k-1,p-2\ge1$, so $[x^{q+p}]\Phi_{pq}(x)=1$. Therefore, $[x^{q+p-1}]F'_1(x)=-2$, and $F'_1(x)$ is not flat, as desired. Since $r>1$, this means that $f'(x)$ is not flat, as desired.

When $r\equiv-1\pmod{pq}$, we claim that $F'_{2(p+q-1)}(x)$ is not flat. Indeed, $F'_{2(p+q-1)}(x)\equiv x^{-1}(1-x^{p-1}-x^{q+p})\pmod{\Phi_{pq}(x)}$ so $F'_{2(p+q-1)}(x)=x^{-1}(1+x^{p-1}-x^{q+p}-\Phi_{pq}(x))$. Again, $[x^{q+p}]\Phi_{pq}(x)=1$, so $[x^{q+p-1}]F'_{2(p+q-1)}(x)=-2$ and $F'_{2(p+q-1)}(x)$ is not flat, as desired. However, we must show that $2(p+q-1)<r$. Except when $(p,q)=(3,5)$, $(p-2)(q-2)>3$ so $2(p+q-1)<pq-1\le r$ as desired. When $(p,q)=(3,5)$, $2(p+q-1)=14$. With $r\equiv-1\pmod{15}$ prime, $r\ge29>14=2(p+q-1)$ as desired.\end{proof}

If Conjecture \ref{pqrstnotflat} is true, we would also expect that higher order cyclotomic polynomials would not be flat. The following conjecture would be useful in proving these things, and seems likely.
\begin{conj}\label{np>n}If $n$ is a positive integer, $p$ a prime and $A(n)>1$, then $A(np)>1$.\end{conj}
This conjecture, if true, would provide an important step in the classification of all flat cyclotomic polynomials, eliminating many candidates. For instance, if both Conjectures \ref{pqrstnotflat} and \ref{np>n} hold, then all flat cyclotomic polynomials would have order at most 4.

The stronger claim that $A(np)\ge A(n)$ for all $n$ and $p$ is actually false. There are eleven counterexamples with $p=3$ and $n<20000$:
\begin{center}\begin{tabular}{c|c|c|c}
$n$&$3n$&$A(n)$&$A(3n)$\\\hline
$4745=5\cdot13\cdot73$&$14235$&$3$&$2$\\
$7469=7\cdot11\cdot97$&$22407$&$4$&$3$\\
$10439=11\cdot13\cdot73$&$31317$&$6$&$4$\\
$14231=7\cdot19\cdot107$&$42693$&$4$&$3$\\
$14443=11\cdot13\cdot101$&$43329$&$5$&$4$\\
$14707=7\cdot11\cdot191$&$44121$&$4$&$3$\\
$16027=11\cdot31\cdot47$&$48081$&$5$&$4$\\
$16523=13\cdot31\cdot41$&$49569$&$6$&$4$\\
$18791=19\cdot23\cdot43$&$56373$&$5$&$4$\\
$19129=11\cdot37\cdot47$&$57387$&$6$&$5$\\
$19499=17\cdot31\cdot37$&$58497$&$8$&$7$
\end{tabular}\end{center}
By periodicity, just one example guarantees that there are infinitely many such $n$ for $p=3$. There are no examples where $A(5n)<A(n)$ for $n<20000$. A natural question arises, which we end with:

\begin{ques}\label{np<pques}For which primes $p$ do there exist $n$ such that $A(np)<A(n)$?\end{ques}

\section{Acknowledgements}

This research was done at the University of Minnesota Duluth REU and was supported by the National Science Foundation (grant number DMS-0754106) and the National Security Agency (grant number H98230-06-1-0013).

I would like to thank Joe Gallian for suggesting the problem, supervising the research, and encouraging me to continue working on this problem despite some early difficulties. Most of all, I would like to thank Nathan Kaplan for more closely supervising my research, giving very useful feedback along the way, and reading over this entire paper after it was written. I would also like to thank Tiankai Liu and Nathan Pflueger for additional comments and suggestions.

\end{document}